\documentclass[12pt,notitlepage]{amsart}
\usepackage{latexsym,amsfonts,amssymb,amsmath,amsthm,graphicx} 
\newcommand{\N}{\ensuremath{\mathbb N}}
\newcommand{\K}{\ensuremath{\mathbf K}}
\newcommand{\Z}{\ensuremath{\mathbb Z}}

\newcommand{\R}{\ensuremath{\mathbb R}}

\newcommand{\Q}{\ensuremath{\mathbb Q}}
\newcommand{\C}{\ensuremath{\mathbb C}}

\newcommand{\ww}{\ensuremath{\mathfrak W}}

\newcommand{\fc}{\ensuremath{\mathfrak c}}

\newcommand{\ri}{\ensuremath{\mathcal{O_{\mathbf{K}}}}} 

\newcommand{\cc}{c^{\circ}}

\newcommand\m{\mu}

\renewcommand\r{\rho}

\newcommand{\ec}{\textrm}

\theoremstyle{plain}		
	\newtheorem{theorem}{Theorem}[section]
	\newtheorem*{theo}{Theorem}
	\newtheorem{prop}[theorem]{Proposition}
	\newtheorem{rem}[theorem]{Remark}
	\newtheorem{cor}[theorem]{Corollary}
    
     \newtheorem{lemma}[theorem]{Lemma}

	\newtheorem*{co}{Corollary}
	
\theoremstyle{remark}		
	
	\newtheorem*{remark}{Remark}
	\newtheorem{example}{Example}

	\newtheorem{hypoth}[theorem]{Hypothesis}
	
\numberwithin{equation}{section}
\begin{document}

\title[On the analytic properties of a cubic Dirichlet series]{On the analytic properties of a cubic Dirichlet series associated to a cubic metaplectic form}
\author{P. Edward Herman} 
\address{Department of Mathematics,
University of Chicago,
5734 S. University Avenue,
Chicago, Illinois 60637}
\email{peherman@math.uchicago.edu}
\thanks{Part of this research was supported by an NSF Mathematical Sciences Research Institutes Post-Doc and by the American Institute of Mathematics.}

\begin{abstract}
In this paper we study the analytic properties of a certain cubic Dirichlet series associated to a metaplectic form $f$ over the cubic cover of $GL_2.$ Such a sum generalizes the work of Shimura in studying a similar quadratic Dirichlet series for a half-weight modular form $f.$ Shimura connects the analytic properties of his Dirichlet series to the L-function of a holomorphic modular form via a converse theorem. This connection, and its higher cover generalizations, has been given the name: Shimura's correspondence. 

Even assuming Shimura's correspondence for the cubic cover of $GL_2,$ the analytic properties of our cubic Dirichlet series are intractable. However, using Langlands's beyond endoscopy idea and analytic number theory, we get nontrivial analytic continuation of the series. Specifically, we obtain an asymptotic for a spectral sum of these cubic Dirichlet series plus an error term. Assuming a certain uniformity hypothesis we can get analytic properties of an individual cubic Dirichlet series of a metaplectic form. In particular we show the cubic series has analytic continuation to $\Re(s)>\frac{9}{7}+\epsilon,$ for any $\epsilon$ with at most a pole at $s=\frac{3}{2}$ if the metaplectic form $f$ is the residual Eisenstein series. 

A key tool needed in studying this series is an identity relating cubic exponential sums to Kloosterman sums. While we do not make a traditional trace formula comparison in this paper, this very same identity is crucial to the fundamental lemma in work of Mao and Rallis.\end{abstract}



\maketitle

\section{Introduction}

Let us recall a result from Shimura's paper on half weight forms \cite{Shi}. Take a half weight form of $f$ of full level and of weight $\frac{k}{2}.$ Assume that $f$ is also an eigenform of the operators $T_{p^2}$ for $p$ prime, such that $p^2 \nmid t$ for some positive integer $t.$ We have associated to $f$ a Fourier expansion $f(z)=\sum_{n=0}^\infty a_n(f) e(nz)$ and an eigenvalue $\lambda_p$ of $T_{p^2}.$ Shimura then proved \begin{theo}\label{shhi}
For $t$ square-free, $$\sum_{n=1}^\infty \frac{a_{n^2t}(f)}{n^s}=\sum_{(n,p)=1}^\infty \frac{a_{n^2t}(f)}{n^s}  \left[ 1- \frac{\left(\frac{t}{p}\right)_2}{p^{s-\frac{k-3}{2}}}\right] \left[1-\frac{\lambda_p}{p^s}+\frac{1}{p^{2-2s-k}}\right]^{-1}.$$

\end{theo}
 
 An immediate consequence of this Theorem is \begin{co}\label{shhi1}
 For $t$ square-free, $$\sum_{n=1}^\infty \frac{a_{n^2t}(f)}{n^s}=a_t(f)\prod_p  \left[ 1- \frac{\left(\frac{t}{p}\right)_2}{p^{s-\frac{k-3}{2}}}\right] \left[1-\frac{\lambda_p}{p^s}+\frac{1}{p^{2-2s-k}}\right]^{-1}.$$

 \end{co}
 
 He then proved using converse theorems that the term $\prod_p \left[1-\frac{\lambda_p}{p^s}+\frac{1}{p^{2-2s-k}}\right]^{-1}$ is an L-function of a classical modular cusp form $F_f$ (the Shimura lift of $f$) of weight $k-1.$ 
 So we therefore understand the analytic continuation of the Dirichlet series $\sum_{n=1}^\infty \frac{a_{n^2}(f)}{n^s}$ via $$\sum_{n=1}^\infty \frac{a_{n^2}(f)}{n^s}=a_1(f)\frac{L(s,F_f)}{\zeta(s-\frac{k-3}{2})}.$$
 
 It is understood that half-weight modular forms can be understood as metaplectic forms over the double cover of $GL_2,$ \cite{GPS}. A natural question then to ask is about the analytic properties of the Dirichlet series $$\sum_{(n)} \frac{a_{n^3}(f)}{\N(n)^s}$$ for a metaplectic form $f$ over the cubic cover of $GL_2.$  Here the sum $(n)$ is over integral ideals in the field $\Q[\omega],\omega^3=1.$  We try to imitate the proof of Shimura in getting an Euler product as in Shimura's corollary above.

 Assume that $f$ is an eigenform for all Hecke operators $T_{p^3}$ with $p$ not dividing $t.$ So $T_{p^3}f(g)=\lambda_p f(g),$ with the relation between the Fourier coefficients and Hecke eigenvalues being for $(n^3 t,p)=1,$ \begin{equation}\label{eq:hecyo}\lambda_p a_{n^3t}(f)=a_{n^3tp^3}(f)+\frac{g_2(1,p)a_{n^3tp}(f)(-1,p^2)}{\N(p)}$$ and $$\lambda_p a_{n^3tp^{3m}}(f)=a_{n^3tp^{3m+3}}(f)+a_{n^3tp^{3m-3}}(f)\N(p^3)\end{equation} with $$g_k(a,p):=\sum_{x(p)} \left(\frac{x}{p}\right)_3^k e(Tr_{\Q[\omega]\slash\Q}(\frac{ax}{p})),$$ and $(a,b)$ the Hilbert symbol. 
 
Let $H_{n,t}(x):=\sum_{m=0} a_{n^3tp^{3m}}(f) x^m,$ then $$\sum_{(n)} \frac{a_{n^3}(f)}{\N(n)^s}=\sum_{\substack{(n)\\(n,p)=1}} \frac{H_{n,t}(\frac{1}{\N(p)^s})}{\N(n)^s}.$$ So using Hecke theory $$\lambda_p x \cdot H(x)=   a_{n^3tp^3}(f)x+ \frac{g_2(1,p)a_{n^3tp}(f)(-1,p^2)x}{\N(p)} + \sum_{m=1}^\infty \big[ a_{n^3tp^{3m+3}}(f)+\N(p^3)a_{p^{3m-3}}(f)\big]x^{m+1}.$$ 


This simplifies with some regrouping and writing $H_{n,t}(x)=H(x)$ to \begin{equation}
\lambda_p x \cdot H(x)=\frac{g_2(1,p)a_{tp}(f)(-1,p^2)x}{\N(p)}-a_{n^3t}(f)+H(x)+N(p^3)x^2H(x).
\end{equation}

This can be rewritten as $$H(x)= [a_{n^3t}(f)-\frac{g_2(1,p)a_{n^3tp}(f)(-1,p^2)x}{\N(p)}][1-x\lambda_p+N(p^3)x^2]^{-1}.$$

With $x=\frac{1}{\N(p)^s},$ we have $$\sum_{m=0} \frac{a_{n^3tp^{3m}}(f)}{\N(p)^{ms}}=[a_{n^3t}(f)-\frac{g_2(1,p)a_{n^3tp}(f)(-1,p^2)}{\N(p)^{1+s}}][1-\frac{\lambda_p}{\N(p)^s}+\frac{1}{\N(p^{2s-3})}]^{-1}.$$

So we relate this back to the cubic Dirichlet series as \begin{multline}\label{eq:hf}\sum_{(n)} \frac{a_{n^3}(f)}{\N(n)^s}= [1-\frac{\lambda_p}{\N(p)^s}+\frac{1}{\N(p^{2s-3})}]^{-1}  \left[\sum_{\substack{(n)\\(n,p)=1}}  \frac{a_{n^3}(f)}{\N(n)^s} - \frac{g_2(1,p)(-1,p^2)}{\N(p)^{1+s}}\sum_{\substack{(n)\\(n,p)=1}} \frac{a_{n^3p}(f)}{\N(n)^s} \right].\end{multline}
 
 If we compare \eqref{eq:hf} with Shimura's Theorem, we notice that in the latter we cannot separate $p$ from $n$ in the Fourier coefficient $a_{n^3 tp}(f),$ while the former has independence of these variables. In the double cover case, this independence of $n$ and $p$ allows us prime by prime to construct the Euler product. It is evident that we cannot associate an Euler product to $\sum_{(n)} \frac{a_{n^3}(f)}{\N(n)^s}$ as is done in Corollary \ref{shhi1}. So we also expect that obtaining any kind of non-trivial analytic continuation of this Dirichlet series will be very difficult. Surprisingly, even though we do not have an understanding of an Euler product, we can get nontrivial analytic continuation. We now describe the method to get this analytic continuation. 
 
 \section{Beyond endoscopy}
 
 Though we are not naturally studying automorphic representations in studying metaplectic forms, we try to use Langlands's  beyond endoscopy idea to understand the analytic continuation of the above cubic Dirichlet series. 
 
 Let us give a quick introduction to the beyond endoscopy idea for automorphic representations of $GL_2.$ Let $\rho$ denote a representation of the dual group of $GL_2,$ and let $$L(s,\Pi,\rho)=\sum_{n=1}^\infty \frac{a_n(\Pi,\rho)}{n^s}$$ denote the associated L-function. One can think of the Dirichlet coefficients $a_n(\Pi,\rho)$ as Fourier coefficients associated to the representation $\Pi.$ Then if $g \in C_0^{\infty}(\R^{+}), \int_0^\infty g(x)dx=1$, 
	the sum 
	$$ G_{\Pi,\rho}(X):=\frac{1}{X} \sum_{n }g(n/X) a_n(\Pi,\rho),$$ 
	and its size (in terms of $X$), gives data about the associated L-function.  If the associated L-function has a simple pole at $s=1,$ then $$G_{\Pi,\rho}(X)=Res_{s=1}L(s,\Pi,\rho) +O(X^{-\delta}), \quad \delta >0.$$ However, it is challenging to get such analytic information about a single automorphic representation. Langlands then incorporated the trace formula into this averaging of the coefficients. One studies, $$  \sum_{\Pi}G_{\Pi,\rho}(X).$$  Now the trace formula can be used as we have a sum over the spectrum of representations on $GL(2).$ We then expect to sieve the representations that have an L-function with non trivial residue at $s=1.$ A key point being when the L-function associated to $\Pi$ has a pole, it is ``detecting" a functorial transfer or that $\Pi$ is lifted from a lower degree.
 
 We do not anticipate that the analytic properties of the metaplectic cubic Diricihlet series contains any obvious functoriality, but we still study \begin{equation}\label{eq:nsum}  \frac{1}{X} \sum_{(n) }g(\N(n)/X)
\left(\sum_{\Pi \neq \mathbf{1}}
h(V,\nu_{\Pi})a_{n^3}(\Pi)\overline{\lambda_{p}(\Pi)}+
\{CSC_{n^3,1}\} \right),\end{equation} where the sum over $\Pi$ parametrizes metaplectic forms over the cubic cover of $GL_2,$ $(n)$ is a sum over integral ideals of $\K,$ and $p$ is prime with $\lambda_p(\Pi)$ the $p$-th Hecke eigenvalue of $\Pi.$ Here $V$ is a test function and $h(V,\nu_{\Pi})$ is a Bessel transform of $V$ needed for convergence. $CSC$ denotes the continuous spectrum contribution. We  describe these explicitly in section \ref{prel}.

 \section{What to expect from \eqref{eq:nsum}?}
 
 As we mentioned earlier, it is not clear what the analytic properties of the Dirichlet series $$\sum_{(n)} \frac{a_{n^3}(f)}{\N(n)^s}$$ for a general metaplectic form $f$ are. However, if the form is the residual Eisenstein series $f_{00},$ we do know what to expect for $\Re(s)\geq \frac{3}{2}.$ From the work of Kazhdan and Patterson \cite{KP} we know $a_{n^3}(f_{00})=\N(n)^{\frac{1}{2}}.$ This property is part of a much more general theorem that for any $k$-th degree cover of $GL_2,$ the Fourier coefficients of the associated residual Eisenstein series are $k$-th periodic. So the sum $$\sum_{(n)} \frac{a_{n^3}(f_{00})}{\N(n)^s}=\sum_{(n)} \frac{1}{\N(n)^{s-\frac{1}{2}}}$$ has a pole at $s=\frac{3}{2}.$
 
 For any other metaplectic form, the analytic properties of the cubic Dirichlet series are unknown. We prove using this beyond endoscopy idea and a standard matching using Hecke operators the following theorem and associated corollary (contingent on a reasonable hypothesis) which gives some analytic properties of these non-residual metaplectic Dirichlet series.
 

\begin{theorem}\label{theorem} Let $g \in C_0^\infty(\R^{+}),$ such that $\int_0^\infty g(t)\sqrt{t}dt=1.$ For a fixed $\epsilon>0$ and $p,D \in \ri$ with $p\equiv 1(3)$ a prime and $(D,p)=1$ square-free. Suppose $D$ is the level for the metaplectic representations $\Pi$ with Hecke eigenvalue $\lambda_p.$ We have then
$$\frac{1}{X} \sum_{(n) }g(\N(n)/X)
\left(\sum_{\Pi \neq \mathbf{1}}
h(V,\nu_{\Pi},p_{\Pi})a_{n^3}(\Pi)\overline{
\lambda_{p}(\Pi)}+
\{CSC_{n^3,p}\} \right)=h(V,1/3,0)K_{p,D}X^{1/2}+ O(X^{\frac{2}{7}+\epsilon}),$$
with $K_{p,D}$ defined in Section \ref{recall}. 
\end{theorem}

\begin{rem}
Note trivially we have the bound $O(X^2)$ by estimating exponential sums on the geometric side of the trace formula. While the Weil bound on the Kloosterman sum gives  $O(X^{\frac{3}{2}}).$ So this is a significant cancellation on the geometric side of the trace formula to get to a main term. 
\end{rem}

It is interesting to see such isolation of a spectral value, namely an eigenvalue of $s=4/3$ coming from the residual Eisenstein series with $\nu_{\Pi}=1/3.$ Usually getting any kind of power savings at all from the main term is an achievement in analytic number theory. 

Notice also via this method the Fourier coefficients for these residual Eisenstein series should be able to be deduced. It seems a hard problem in metaplectic representation theory for higher degree coverings to know exactly the Fourier coefficients as their is a lack of uniqueness of Whittaker models, which precludes standard Hecke relations, see \cite{P6}.

With such an asymptotic, one way to get analytic properties on a single cubic Dirichlet series attached to a metaplectic form is to assume the spectral sum is finite (for example, representations with eigenvalues in a finite set). If we assume the hypothesis: \begin{hypoth}\label{stu}
Assume for any $ \Pi,$ there exists $M_1,M_2 >0,$ such that \begin{equation}\label{eq:assump}\frac{1}{X}\sum_{n \in \ri}g(\N(n)/X) a_{n^3}(\Pi) \ll (1+|\nu_{\Pi}|)^{M_1} (1+|p_{\Pi}|)^{M_2} \end{equation} where $\{\nu_{\Pi},p_{\Pi}\}$ are the archimedean parameters associated to $\Pi,$\end{hypoth} then we can choose a test function in the trace formula such that the spectral sum is finite. Strong multiplicity proven by Flicker \cite{F} then implies we can isolate a single metaplectic form's Dirichlet series.
Such a hypothesis if we studied the same problem for automorphic representations would follow directly from the functional equation of the associated L-function. However, it is not clear how this cubic Dirichlet series associated to a metaplectic form is associated to an L-function or an object with a functional equation. It is plausible that the hypothesis can be proved by using the Shimura's correspondence (\cite{F},\cite{P6}) and relating the series to an L-function of an automorphic form.


Otherwise if we do not want to make this assumption, we need an asymptotic for all $\lambda_{p^n}$ which are Hecke eigenvalues of $T_{p^{3n}}$ for almost all primes $p.$ With that information we could apply a linear independence of characters argument analogous to \cite{F}. This approach we take up in another paper in removing Hypothesis \ref{stu}.


Either way, believing the uniformity hypothesis to be reasonable, we prove:

\begin{cor}\label{iso}
Assume Hypothesis \ref{stu} holds.  For any $\epsilon>0$ and $\Pi$ a representation associated to a non residual metaplectic form,  the Dirichlet series $$\sum_{(n)} \frac{a_{n^3}(\Pi)}{\N(n)^s}$$ has analytic continuation to $\Re(s)>\frac{9}{7}+\epsilon.$ If the metaplectic representation is associated to a residual Eisenstein series then the series has analytic continuation to $\Re(s)>\frac{9}{7}+\epsilon$ with a simple pole at $s=\frac{3}{2}.$
\end{cor}


We also attach an appendix to this paper, which studies the same calculation over $\Q,$ with trivial central character. The results are only partial there. Such a sum would be closely connected with the symmetric cube L-function over $\Q.$

\section{Outline of Paper}
 
 In the metaplectic world, there is a large difference from the Hecke eigenvalues and the Fourier coefficients for covers of $GL_2$ greater than $2.$ Take a metaplectic representation $\Pi.$ Normalize the Fourier coefficients so that $a_{1}(\Pi)=1.$ This gives the relation $$\lambda_p(\Pi) =a_{p^3}(\Pi)+\frac{g_2(1,p)a_{p}(\Pi)(-1,p^2)}{\N(p)}.$$ So the probem we consider is to look at $$\frac{1}{X} \sum_{(n) }g(\N(n)/X)
\sum_{\Pi \neq \mathbf{1}}
h(V,\nu_{\Pi},p_{\Pi})a_{n^3}(\Pi)\bigg(\overline{a_{p^3}(\Pi)+\frac{g_2(1,p)a_{p}(\Pi)(-1,p^2)}{\N(p)}}\bigg).$$ Up to some harmless factors independent of $n$ and the sum over $\Pi,$ it suffices to investigate \begin{equation}\label{eq:woww}\frac{1}{X} \sum_{(n) }g(\N(n)/X)
\sum_{\Pi \neq \mathbf{1}}
h(V,\nu_{\Pi},p_{\Pi})a_{n^3}(\Pi)\overline{a_{p}(\Pi)},\end{equation} and $$\frac{1}{X} \sum_{(n) }g(\N(n)/X)
\sum_{\Pi \neq \mathbf{1}}
h(V,\nu_{\Pi},p_{\Pi})a_{n^3}(\Pi)\overline{a_{p^3}(\Pi)}.$$

    For simplicity of understanding the proof, after applying the trace formula to the left hand side of say \eqref{eq:woww}, we apply Poisson summation to the $n$-sum. The dual sum we split into 2 cases: the zeroth term and anything else. For the zeroth term we can apply techniques of Patterson on cubic Gauss sums. 
 For the non-zero terms of this dual sum, we encounter cubic exponential sums, which in general are difficult to deal with. One of the main tools we use is an identity relating cubic exponential sums to twisted Kloosterman sums. For example, for $c\equiv 1(3), c\in \Z,$ we have $$ \sum_{k(c)} e(\frac{k^3+k}{c})=\sum_{y(c)^{*}} (\frac{y}{c})_3 e(\frac{y-\overline{3^3 y}}{c}).$$
This is one of two techniques the author knows to do to simplify cubic exponential sums. The other is expressing them into Gauss sums. Quadratic exponential sums, on the other hand, have a wealth of tools to help study them. This identity also is the cornerstone in proving the Fundamental Lemma of Mao and Rallis \cite{MR}. 

After the use of this identity, we simplify cubic exponential sums to Ramanujan sums, which we can more easily handle. From here, we still need much more cancellation to get analytic information of the Dirichlet series near $s=3/2,$ where one should encounter the residual Eisenstein series on the spectral side. Similar to recent papers in analytic number theory, not only do we need cancellation in the cubic exponential sum, Kloosterman sum, etc..; we need cancellation in the sum over the modulus of these sums.  

Thankfully, after the application of the cubic identity, one is left with relatively simple sums, which look more or less like:

$$\sum_{f} \frac{\mu(f) g(1,f)}{\N(f)} \sum_{\substack{\N(fc)\sim X^{3/2}}}  (\frac{\overline{p}f}{c})_3  \sum_{\substack{\N(m)\ll X^{1/2}\\(m,c)=1}}\sum e(\frac{ m^3\overline{p3^3}}{c}),$$ where $g(1,f)$ is a Gauss sum attached to the cubic residue character.
Using elementary reciprocity, $$\frac{\overline{A}}{B} +\frac{\overline{B}}{A} \equiv \frac{1}{AB} (1)$$ and cubic reciprocity, we can swap the modulus of the exponential sums and instead investigate a much simpler sum
$$\sum_{\N(f) \ll X^{3/2}} \frac{\mu(f) g(1,f)}{\N(f)} \sum_{\N(c)\sim \frac{X^{3/2}}{\N(f)}} \left(\frac{\overline{c}}{p}\right)_3\left(\frac{c}{f}\right)_3  \sum_{\substack{\N(m)\ll X^{1/2}\\ (m,c)=1}}    e(\frac{-m^3\overline{c}}{3^3p})e(\frac{ m^3}{3^3pc}).$$ Now the first exponential sum $$e(\frac{-m^3\overline{c}}{3^3p})$$ will have large cancellation due to the small modulus $3^3p$ as $p$ is fixed, and the second exponential sum $$e(\frac{ m^3}{3^3pc})$$ is ``flat" as $\N(c)\sim X^{3/2}$ and $\N(m^3)\ll X^{3/2}$ and can be included with the other archimedean test functions in the setup of the problem. This is standard technique with such analytic number theory problems, separate the oscillating part of the problem from the ``flat" part. After an application of Poisson summation in the $c$-variable now, the residual Eisenstein series shows up as the zeroth term, which one can think of as the $c$-sum is being replaced by an integral. 

The problem is the sum above is the simplest version of what we need to investigate. However, the approach to this easiest sum with finesse will work for the other harder cases.

Giving an asymptotic with an explicit error term for these sums gives Theorem \ref{theorem}. If we assume we can do the analogous study when $p$ is replaced by a cubic power $p^3,$ and assume Hypothesis \ref{stu}, we then through the use of having enough ``test" functions reduce to a finite dimensional sum on the spectral side of the trace formula and this particular asymptotic. From there we can apply Hecke operators to isolate a single representation $\Pi$ associated to a metaplectic form. The theorem's bound now for a single Dirichlet series associated to $\Pi$ we then associate to analytic information of the corresponding cubic Dirichlet series. 


Lastly in the first appendix, we try understand what is happening for a beyond endoscopy approach  of a cubic Dirichlet series over $\Q$. We can only deal with a small part of the analysis. Again we apply tools of Patterson in \cite{P1}, \cite{P2}, and \cite{P3}. The analysis seems familiar to that of Heath-Brown and Patterson's result on equidistribution of angles associated to cubic Gauss sums in \cite{HBP}.

In the second appendix, we prove the complex version of a result in the appendix of \cite{V} that given a nice test function on the geometric side of the trace formula, the Bessel transforms associated to it on the spectral side of the trace formula are sufficient to reduce the spectral sum to a finite dimensional sum.

 \section{Connection to the symmetric cube L-function of a $GL_2$ automorphic representation}
 
 Let $\pi$ be an automorphic representation over $\Q[\omega]$ with cubic central character $\chi$ of level $D.$ We then can consider the symmetric cube L-function $$L(s,\pi,sym^3)=\sum_{(n)}\frac{a_{n}(\pi,sym^3)}{\N(n)^s}$$ associated to it. If we consider just the square-free coefficients of this L-function it easy to show that they are equal to the Fourier coefficients at cubes, or $a_{n}(\pi,sym^3)=a_{n^3}(\pi).$ If one wanted to investigate the analytic properties of the symmetric cube L-function, say just to the left of $s=1$, one could perhaps study the cubic Dirichlet series $$\sum_{(n)}\frac{a_{n^3}(\pi)}{\N(n)^s}.$$ Indeed, an elementary argument shows that the symmetric cube L-function has analytic continuation for $\Re(s)>\frac{1}{2}$ iff the cubic Dirichlet series has analytic continuation for $\Re(s)>\frac{1}{2}.$ So we can then try to study this cubic Dirichlet series associated with the symmetric cube L-function via beyond endoscopy as we do exactly in this paper for metaplectic forms. As mentioned in the outline of the proof, we apply the trace formula and then Poisson summation in the $n$-sum originally associated to the Dirichlet series to get: 
 \begin{equation*}\label{eq:poiss88}  \sum_{c} \frac{1}{\N(c)^2}\sum_{x  (c)^{*}} \chi(x)
e(\frac{\overline{x}}{c}) \left\{\sum_m \sum_{k(c)} e(\frac{x k^3-mk}{c})
 \int_{\C} e( \frac{-tm}{c})g(t/X)
 V(\frac{4\pi
\sqrt{t^3 }}{c}) d^{+}t\right\}. \end{equation*}
 
 For $m\neq 0,$ the strategy applied to metaplectic forms is to use Katz's cubic identity to swap cubic exponential sums with more elementary Ramanujan sums. The key point for metaplectic forms is that the analogous equation above on the geometric side of the trace formula has a cubic residue character at {\it every} prime $p$ dividing $c$ in the sum of exponential sums. This is not true for the geometric side of the trace formula for automorphic representations and the analogous cubic character $\chi$ is seen {\it only} at the finite primes dividing the level $D.$ In other words, for $c 
\in \Z[\omega],$ we can apply Katz's identity at every prime dividing $c,$ but we only see the character $\chi$ from the $x$-sum only at the primes dividing the level in the automorphic setting. How is this different? In the metaplectic setting, where we see $``\chi"$ at every prime, interchanging the $x$-sum with Katz's identity, we get Ramanujan sums, while in the automorphic setting we get something much more difficult to deal with at primes not dividing the level. So if having a nontrivial character at every prime is crucial to the argument, the automorphic setting will be difficult. This shows how different studying the problem of analytic continuation via beyond endoscopy of the symmetric cube L-function of an automorphic representation is verses the cubic Dirichlet series of a cubic metaplectic form is.  
 
 Perhaps the difficulty seen on the geometric side of the trace formula is a reflection of approaching the symmetric cube L-function via this truncated Dirichlet series. Following this line of thought, does this make the cubic Dirichlet series for the metaplectic setting a more natural object to study, as the corresponding analysis on the geometric side of the trace formula is easier to deal with? 
  
 \section{Preliminaries}\label{prel}

 Let
$\mathbf{K}=\Q(\omega),$ with $\omega$ a cube root of unity.  The ring of integers
will be denoted $\mathcal{O_{\mathbf{K}}}.$ Here the discriminant
is denoted $D_{\mathbf{K}},$ and the different is generated by $\delta =
\sqrt{-3}.$ Likewise, define the absolute norm of ideal generated by $c$ as
$\mathbb{N}(c).$ We denote the
non-trivial automorphism in this field by $x \rightarrow x'.$ We
use the standard notation for the exponential $e(x) := exp(2\pi
i(\frac{x}{\delta} +\frac{x'}{\delta'})).$ 

Let $$\Gamma_1=\{ \gamma \in SL_2(\ri) : \gamma \equiv I (3)\},$$ and 
$$ \Gamma_{*}(D)=\{  \left( \begin{array}{cc}
a & b  \\
c & d  \\ \end{array} \right)  \in \Gamma_1 : D|c \}.$$

Of first importance for us are the subgroups of ${\rm SL}_2(\Z[i])$ defined by

\begin{align}
\Gamma_2 \ \
&=\{\gamma\in {\rm SL}_2(\Z[i]) \, : \, \exists g\in {\rm SL}_2(\Z), \gamma\equiv g \pmod 3\}, \label{Kloo:eq:2}\\
\Gamma_1(3) \ \
&=\{\gamma\in {\rm SL}_2(\Z[i]) \, : \, \gamma\equiv 1 \pmod 3\}, \label{Kloo:eq:3}\\
\Gamma_0(D) 
&= \{\gamma\in {\rm SL}_2(\Z[i]) \, : \, \gamma\equiv 
\left(\begin{smallmatrix} * & * \\ 0 & * \end{smallmatrix}\right)  \pmod D\}. \label{Kloo:eq:4}
\end{align}
The Kubota symbol $\kappa$ can now be introduced. It is defined on $\Gamma_1$ by

\begin{equation*}
\kappa (\gamma) = \begin{cases} 
\left(\frac{a}{c}\right)_{\!3} & \ec{if $c\neq 0$}\\
1 &\ec{if $c=0$},
\end{cases} \qquad \ec{ where } \gamma=\begin{pmatrix}a&b\\ c&d\end{pmatrix}\in \Gamma_1,
\end{equation*} with $\left(\frac{c}{a}\right)_{\!3}$ the cubic power residue symbol.
This definition is then extended to $\Gamma_2$ by defining $\kappa$ trivially on ${\rm SL}_2(\Z)$. More precisely, for any $\gamma_2\in\Gamma_2$, there exists $g\in {\rm SL}_2(\Z)$ and $\gamma_1\in \Gamma_1$ such that $\gamma_2=g \gamma_1$, and we define 

\begin{equation}\label{Kloo:eq:5}
\kappa(\gamma_2) = \kappa(\gamma_1).
\end{equation}

In the rest of this section we shall consider the discrete subgroup of $SL_2(\Z[i]),$ $\Gamma=\Gamma_1(3) \cap \Gamma_0(D)$ for $D$ square-free.  
Let $\Pi$ be a representation associated to a metaplectic form on $\Gamma \backslash G,$ then for $f \in \Pi$ we have
\begin{equation*}
f(\gamma g) = \kappa(\gamma) f(g) \qquad \ec{for all } \gamma \in \Gamma.
\end{equation*} 
Note this also implies $f(g)=f(-g)$ and therefore $a_{n}(\Pi)=a_{-n}(\Pi).$

We use the Kuznetsov trace formula over the imaginary quadratic field $\K$ from \cite{BMo}, \cite{L}. We chose to study the representations that transform by $\Gamma$ with square-free $D \equiv 1(3).$ This is very nearly the same family studied in \cite{LP}. For cusps $\sigma=g_{\sigma}\cdot \infty, \sigma'=g_{\sigma'}\cdot \infty,$ it takes the form
 \begin{equation}\sum_{\Pi \neq \mathbf{1}}
h(V,\nu_{\Pi},p_{\Pi})c^{\sigma}_{\mu}(\Pi)\overline{c^{\sigma'}_{\nu}(\Pi)}+
\{CSC_{\mu,\nu}\} = \end{equation} $$=
 \sum_{c\neq 0}^\infty
\frac{1}{\mathbb{N}(c)} S^{\sigma,\sigma'} ( \nu,\mu,c) V(4 \pi \sqrt{\mu
\nu}
 /c),$$ where $\mu, \nu \in \mathcal{O_{\mathbf{K}}}$  
and $V(x) \in C_0^{\infty}(\C^{*}).$ The transform associated to the archimedean parameters $(\nu_{\Pi},p_{\Pi})$ is \begin{equation}
 h(V,\nu_{\Pi},p_{\Pi})
 =          \int_{\C^{*}}
V(z)B_{{\nu_{\Pi},p_{\Pi}}}(z)\frac{d^{+}z}{|z|^2},  \end{equation} where $d^{+}z$ is the Lebesque measure on $\C.$  Here, $$B_{\nu,p}(z) = 
\text{ sin}(\pi(\nu-p))^{-1}(|\frac{z}{2}|^{-2\nu} (\frac{iz}{|z|})^{2p}J^*_{-\nu+p}( z)J^*_{-\nu-p}( \overline{z}) - |\frac{z}{2}|^{2\nu} (\frac{iz}{|z|})^{-2p}J^*_{\nu-p}( z)J^*_{\nu+p}( \overline{z}),$$ where
$J^*_{\mu}(z)=J_{\mu}(z)(\frac{z}{2})^{-\mu}$ and $J_{\mu}(z)$ is the standard $J$-Bessel function of index $\mu.$ Following \cite{BM} the choice of square root does not matter, provided $\arg \overline{\sqrt{\N(t)}}=-\arg \sqrt{\N(t)}.$

The
Kloosterman sum equals \begin{equation}\label{eq:truk} S^{\sigma,\sigma'} ( \nu,\mu,c):=\sum_{\substack{\gamma \in
\Gamma_{\sigma}\setminus \Gamma / \Gamma_{\sigma'}\\c(\gamma)=c}}
(\frac{a(\gamma)}{c(\gamma)})_3 e((\frac{\nu d(\gamma)+\mu a(\gamma)}{c})),\end{equation} where $\Gamma_{\sigma}:=\{g \in \Gamma \mid g\cdot \sigma=\sigma\}$ and
 $a=a(\gamma), b=b(\gamma), c=c(\gamma),$ and $d=d(\gamma)$ by 
$$\left( \begin{array}{cc}
a & b  \\
c & d  \\ \end{array} \right) = g_{\sigma} \gamma g_{\sigma'}^{-1}.$$

\section{Geometric side of Trace formula}\label{geot}

Let $g \in C_0^\infty(\R^{+}),$ such that $\int_0^\infty g(x)\sqrt{x}dx=1.$
Let us not restrict the Fourier coefficient $a_b(\Pi)$ to $a_p(\Pi)$ or $a_{p^3}(\Pi)$ yet. The Kuznetsov trace formula gives us that \eqref{eq:nsum} equals \begin{equation}\label{eq:kuznet}
\frac{1}{X} \sum_{(n) }g(\N(n)/X) \sum_{c\neq 0}^\infty
\frac{1}{\mathbb{N}(c)} S^{\sigma,\sigma'} ( n^3,b,c) V(\frac{4 \pi \sqrt{n^3b}}{c})
\end{equation}

We now refer to a explicit calculation of the Kloosterman sum following \cite{LP}. We consider the cusp $1-D$ which shall be represented as $g_{\sigma}\cdot \infty$ with $g_{\sigma}=\left( \begin{array}{cc}
D-1 & l  \\
1 & m  \\ \end{array} \right)$ with $l \equiv 1(3)$ and $m \equiv 0(3).$ Let $\sigma'$ be the identity matrix, then following Proposition 5.1 of \cite{LP}

\begin{prop}
For any $c$, the sum $S^{\sigma,\sigma'} ( \nu,\mu,c)$ is zero unless $c\equiv 1(3),$ in which case we have $$ S^{\sigma,\sigma'} ( \nu,\mu,c)=\sum_{a(c)^{*}} e(\frac{\mu a + \nu \overline{a}}{c})(\frac{a}{c})_3:=S_3(\mu,\nu,c).$$ Here $\overline{x}$ is the multiplicative inverse of $x (c).$
\end{prop}

\begin{remark}
Note for $a,b\equiv 1(3),$ cubic reciprocity implies $$(\frac{a}{b})_3=(\frac{b}{a})_3.$$ We will use this repeatedly through out the paper without anymore mention. 
\end{remark}

We now break up the Kloosterman sums and gather all the $n$-terms to get, 

\begin{equation}\label{eq:sm}  \sum_{\substack{c \equiv 1(3)\\c\equiv 0(D)}} \frac{1}{\N(c)}\sum_{x  (c)^{*}}
(\frac{x}{c})_3 e(\frac{b\overline{x}}{c}) \end{equation}
 $$\left\{
 \sum_{(n) }
 e(\frac{n^3 x}{c})g(\N(n)/X)
 V(\frac{4\pi
\sqrt{n^3b }}{c})  \right\}.
$$

To apply Poisson summation in the $n$-variable, we now ``unfold" the sum over the ideals $(n)$ to the integers of $\ri.$ As the nontrivial roots of unity $\mu$ of $\Q[\omega]$ satisfy $\mu^3=1,$ we are only changing the sum to study by $\frac{1}{3}.$ As the metaplectic forms we study are even and therefore have even Fourier coefficients, we also multiply by $\frac{1}{2}.$

As the term in brackets in \eqref{eq:sm} is a smooth function, we can
apply Poisson summation to the $n$-sum in residue classes mod $c$ to get,

\begin{equation}\label{eq:poisson} \frac{1}{X6\sqrt{D_{\mathbf{K}}}} \sum_{\substack{c\equiv 1(3)\\c\equiv 0(D)}} \frac{1}{\N(c)^2}\sum_{x  (c)^{*}} (\frac{x}{c})_3
e(\frac{b\overline{x}}{c}) \end{equation}
 $$\left\{
 \sum_m \sum_{k(c)} e(\frac{x k^3-mk}{c})
 \int_{\C} e( \frac{-tm}{c})g(\N(t)/X)
 V(\frac{4\pi
\sqrt{t^3b }}{c}) d^{+}t\right\}.$$

Change of variables $t\rightarrow \sqrt{X} t$, gives

\begin{equation}\label{eq:poisson1}  \frac{1}{6\sqrt{D_{\mathbf{K}}}} \sum_{\substack{c\equiv 1(3)\\c\equiv 0(D)}} \frac{1}{\N(c)^2}\sum_{x  (c)^{*}}
(\frac{x}{c})_3 e(\frac{b\overline{x}}{c}) \end{equation}
 $$\left\{
 \sum_m \sum_{k(c)} e(\frac{x k^3-mk}{c})
 \int_{\C} e( \frac{-t\sqrt{X}m}{c})g(\N(t))
 V(\frac{4\pi
\sqrt{(\sqrt{X}t)^3b }}{c}) d^{+}t\right\}.$$

Since $g$ and $V$ are of compact support, $\N(c) \sim X^{3/2}.$ Thus, it is sufficient to study \eqref{eq:poisson1} as \begin{equation}\label{eq:poisson2}  \sum_{\substack{c\equiv 1(3)\\c\equiv 0(D)}} \frac{1}{\N(c)^2}\sum_{x  (c)^{*}}
(\frac{x}{c})_3 e(\frac{b\overline{x}}{c}) \sum_m \sum_{k(c)} e(\frac{x k^3-mk}{c})
W_{m}(\frac{c}{X^{3/4}}),\end{equation} where  $W_{m}(y)= \frac{1}{6\sqrt{D}} \int_{\C} e( \frac{-tm}{X^{1/4}y})g(\N(t))
 V(\frac{4\pi
\sqrt{t^3b }}{y}) d^{+}t.$ When we need a true asymptotic we will go back to the original notation.  Notice the terms $\N(m)>\sqrt{X}$ are negligible by integration by parts in the $t$-variable. Negligible in this case means a bound \begin{equation}\label{eq:poisson8}  \sum_{\substack{c\equiv 1(3)\\c\equiv 0(D)}} \frac{1}{\N(c)^2}\sum_{x  (c)^{*}} (\frac{x}{c})_3
e(\frac{b\overline{x}}{c}) \sum_{\N(m)>\sqrt{X}} \sum_{k(c)} e(\frac{x k^3-mk}{c})
W_{m}(\frac{c}{X^{3/4}})=O(X^{-A}),\end{equation} for any positive integer $A>0.$ Trivially we have the bound $O(X^2)$ for \eqref{eq:poisson2}. While applying the Weil bound on the Kloosterman sum of \eqref{eq:kuznet}, see \cite{V}:

\begin{equation}\label{eq:weil}
S_3(r,s,c) \ll_{\K} (r,s,c)^{1/2} \N(c)^{1/2},
\end{equation} gives  the bound $O(X^{\frac{3}{2}}).$ This comes from the fact that $(n^3,b,c)^{1/2} $ can be as big as $\N(c)^{1/2}.$

\section{Using a Cubic identity}

 Our focus is on 
\begin{equation}\label{eq:poissonm} \frac{1}{X^{3/2}} \sum_{\substack{c\equiv1(3)\\c\equiv 0(D)}} \sum_{m} \frac{1}{\N(c)}\sum_{x  (c)^{*}}
 (\frac{x}{c})_3 e(\frac{b\overline{x}}{c})  \sum_{k(c)} e(\frac{x k^3-mk}{c})\ww_{m}(\frac{\N(c)}{X^{3/2}}),\end{equation} where we let $\ww_{m}(y)= \frac{1}{\N(y)}W_m(y).$ This notation helps us locate that $\N(c) \sim X^{3/2}.$

It is key to understand $$\sum_{k(c)} e(\frac{x k^3-mk}{c}).$$ 
Let us define $T(A,B,c):=\sum_{x(c)}e(\frac{Ax^3+Bx}{c}).$

We break this into 2 cases: $m=0$ and $m\neq0.$

\subsection{The case of $m=0$}\label{m011}

We prove in this subsection, \begin{prop}\label{0m0}
 For any $\epsilon>0,$ we have \begin{equation}\label{eq:poisson0} \frac{1}{X^{3/2}} \sum_{\substack{c\equiv1(3)\\c\equiv 0(D)}} \frac{1}{\N(c)}\sum_{x  (c)^{*}}
 (\frac{x}{c})_3 e(\frac{b\overline{x}}{c})  \sum_{k(c)} e(\frac{x k^3}{c})
\ww_{0}(\frac{\N(c)}{X^{3/2}})=O(X^{-3/4+\epsilon}).\end{equation}
\end{prop}

So this is significantly smaller then the bound of Theorem \ref{theorem}.

The arithmetic to study in this case is \begin{equation}\label{eq:m0m}
\sum_{x  (c)^{*}}
 (\frac{x}{c})_3 e(\frac{b\overline{x}}{c}) T(x,0,c).
\end{equation}

We reduce the study to prime power modulus by the following lemma.
\begin{lemma}\label{crt} Let $c=c_1 c_2, (c_1,c_2)=1,$ and $(A,c)=1,$ then  
 $$T(A,B,c_1c_2)=T(Ac_2^2,B,c_1) T(Ac_1^2,B,c_2).$$

\end{lemma}
\begin{proof} 
This is a direct application of Chinese remainder theorem.
\end{proof}

Now we state Lemma 8.7 from \cite{LP}.

\begin{lemma}\label{t00y}
Let $p \nmid 3A,$ then $$T(A,0,p^{k+3})=\N(p)^2 T(A,0,p^k).$$
\end{lemma}

Using the above 2 lemmas we can reduce study of \eqref{eq:m0m} to \begin{equation}\label{eq:m0m1}
\sum_{x  (p^k)^{*}}
 (\frac{x}{p^k})_3 e(\frac{b\overline{x}}{p^k}) T(zx,0,p^k)=\N(p)^{2l}\sum_{x  (p^k)^{*}}
 (\frac{x}{p^k})_3 e(\frac{b\overline{x}}{p^k}) T(zx,0,p^h),
\end{equation} where $(z,p)=1, k=h+3l, h(3).$

\begin{lemma}\label{psq}
Let $k>1,$ with $k=h+3l, h(3).$ For $(z,p)=1,$ $$\sum_{x  (p^k)^{*}}
 (\frac{x}{p^k})_3 e(\frac{b\overline{x}}{p^k}) T(zx,0,p^h)=0.$$
\end{lemma}

\begin{proof}
We first note the fact from \cite{BY} that $$g(r,p^k):=\sum_{x(p^k)^{*}} (\frac{x}{p^k})_3 e(\frac{rx}{p^k})=0$$ for $(r,p)=1$ and $k>1.$ Thus for $h=0,$ we are done. For $h=1,$ from \cite{LP} \begin{equation}\label{eq:spg}T(zx,0,p)=g(zx,p) + \overline{g(zx,p)}.\end{equation} A change of variables in the Gauss sum and the conjugate Gauss gives $$\overline{(\frac{zx}{p})_3}g(1,p) + (\frac{zx}{p})_3\overline{g(1,p)}.$$ The cubic residue character has conductor $p,$ so either the sum over $x(p^k)^{*}$ is a Ramanujan sum or a Gauss sum, both of which are zero. If $h=2$ then by \cite{LP} $T(xz,0,p^2)=\N(p),$ so again the Gauss sum is zero.  

\end{proof}

Thus the study of \eqref{eq:m0m} is reduced to square-free modulus. 

Now by Proposition 8.1 in \cite{LP} for $(A,c)=1,$ $$T(A,0,c)=\sum_{c_1c_2c_3^3=c} g(A,c_1)\overline{g(A,c_2)}\N(c_3)^2.$$

Incorporating this into \eqref{eq:m0m}, we have \begin{equation}
\sum_{c_1c_2=c} \sum_{x  (c)^{*}}
 (\frac{x}{c})_3 e(\frac{b\overline{x}}{c})g(x,c_1)\overline{g(x,c_2)}. 
\end{equation}

With a change of variables this reduces to $$\sum_{c_1c_2=c}g(1,c_1)\overline{g(1,c_2)}\sum_{x  (c)^{*}}  (\frac{x^2}{c_2})_3 e(\frac{b\overline{x}}{c}).$$

We have $$\sum_{x  (c)^{*}}  (\frac{x^2}{c_2})_3 e(\frac{b\overline{x}}{c})=\mu(c_1)g(b,c_2),$$ by an application of the Chinese remainder theorem. 
Now assume $(b,c_2)=1$ then \eqref{eq:m0m} reduces to $$\sum_{c_1c_2=c} \mu(c_1)g(b,c_1)(\frac{b^2}{c_2})_3\N(c_2).$$

If $(b,c_2)>1,$ then as $c_2$ is square-free, an orthogonality of characters argument shows the Gauss sum $g(b,c_2)=0.$

\begin{proof}\{{\it Proposition \ref{0m0}\}}
By use of the previous lemmas, \eqref{eq:poisson0} equals 
\begin{equation}\label{eq:poisson01}  \sum_{\substack{c_1c_2\equiv1(3)\\c_1c_2\equiv 0(D)}} \frac{\mu(c_1)g(b,c_1)}{\N(c_1)^2} \frac{(\frac{b^2}{c_2})_3}{\N(c_2)}
\int_{\C} g(\N(t))V(\frac{\sqrt{t^3(\sqrt{X}^3}}{c_1c_2})d^{+}t.\end{equation} 
By the support of $g$ and $V$ \eqref{eq:poisson01} is bounded above and below up to an absolute constant by $$\sum_{\substack{c_1c_2\equiv1(3)\\c_1c_2\equiv 0(D)}} \frac{\mu(c_1)g(b,c_1)}{\N(c_1)^2} \frac{(\frac{b^2}{c_2})_3}{\N(c_2)}H(\frac{\N(c_1c_2)}{X^{3/2}}),$$ with $H \in C^{\infty}_0(\R^{+}).$

So to understand the analytic properties of \eqref{eq:poisson01} it suffices to understand the above sum. Then by Mellin inversion, the above term is bounded for $\sigma$ large by \begin{equation}\label{eq:final0}
\frac{1}{2\pi i } \int_{\Re(s)=4} \tilde{H}(s)X^{\frac{3s}{2}} \left(\sideset{}{^*}\sum_{\substack{c_1\equiv1(3)\\c_1\equiv 0(D)}}\frac{\mu(c_1)g(b,c_1)}{\N(c_1)^{2+s}}\right)\left(\sum_{c_2} \frac{(\frac{b^2}{c_2})_3}{\N(c_2)^{1+s}} \right)ds.
\end{equation}
The $c_2$-sum has analytic continuation on the complex plane as $b$ is not a unit, while the $c_1$-sum, using trivial bounds, has analytic continuation to $\Re(s)=\frac{-1}{2}+\epsilon,$ for any $\epsilon>0.$
Shift the contour of the Mellin transform to $\Re(s)=\frac{-1}{2}+\epsilon,$ for a fixed $\epsilon,$ then \eqref{eq:final0} is $O(X^{-\frac{3}{4}+\epsilon}).$

Now assume $b=p^3,$ then \eqref{eq:poisson01} equals  \begin{equation}\label{eq:final08}
\frac{1}{2\pi i } \int_{\Re(s)=4} \tilde{H}(s)X^{\frac{3s}{2}} \left(\sideset{}{^*}\sum_{\substack{c_1\equiv1(3)\\c_1\equiv 0(D)}}\frac{\mu(c_1)g(b,c_1)}{\N(c_1)^{2+s}}\right)\left(\sum_{c_2} \frac{1}{\N(c_2)^{1+s}} \right)ds.
\end{equation}

This by an analogous argument has a similar bound of $O(X^{-\frac{3}{4}+\epsilon}).$


 \end{proof}

\subsection{The case of $m\neq 0$}

In this case we first need to deal with cubic exponential sums $T(A,B,c)$ with $(AB,c)>1.$ The identity we ultimately want to apply is applicable when $(AB,c)=1,$ so we need to reduce to such case.
Reducing to a prime modulus by Lemma \ref{crt}, we have the following local calculation:

\small
\begin{lemma}\label{zerob}
Let $(AB,p)=1$  then  
\begin{equation*}
 \sum_{x(p^k)} e(\frac{A x^3-p^j Bx}{p^k})= \\
\begin{cases}
\N(p)^{j}T(A,B,p^{k-\frac{3j}{2}}), & \mbox{if } j\leq k/2 \mbox{ and } j \mbox{ is even} \\ \\

0, & j\leq k/2, k\geq3 \mbox{ and } j \mbox{ is odd } \\ \\
\N(p), & k=2,j=1 \\ \\
\N(p)^{k/2}T(A,Bp^{h+\lfloor{k/4}\rfloor-\lceil{k/4}\rceil },p^{k-3\lceil{k/4}\rceil}), &  j=h+k/2, 0< h\leq k/4, k\in 2\N \\ \\
\N(p)^{k/2}T(A,0,p^{k-3\lceil{k/4}\rceil}) , &  j=h+k/2, h> k/4, k \in 2\N \\ \\
\N(p)^{(k+1)/2}T(A,0,p^{\lfloor(k-3)/4\rfloor}) , & j=h+(k+1)/2, \\
 & h \geq \lfloor(k-3)/4\rfloor, k \text{ odd}\\
\N(p)^{(k+1)/2}T(A,Bp^{h+\lfloor{(k+1)/4}\rfloor-\lceil{(k+1)/4}\rceil},p^{\lfloor(k-3)/4\rfloor}) , & \tiny{ j=h+(k+1)/2, }\\ 
& h < \lfloor(k-3)/4\rfloor, k \text{ odd}.
\end{cases}
\end{equation*}

\end{lemma}

\normalsize
\begin{proof}
Take the case $j\leq k/2.$ Then  we can rewrite $$\sum_{x(p^k)} e(\frac{A x^3-p^j Bx}{p^k})=p^j
\sum_{\substack{a(p^{k-j})\\3Aa^2\equiv 0(p^j)}} e(\frac{A a^3 + p^j B a}{p^k}),$$ using $x=a+p^{k-j}
b$ with  $a(p^{k-j}), b(p^j).$ Since $j=2l,$ $l \in \N$, $a=p^l y, y(p^{k-3l})$. Incorporating this into the cubic exponential sum we have $\N(p)^{j}T(A,B,p^{k-3j/2}).$ 

Similarly, the same procedure for $j\leq k/2$ and $j$ odd gives for $k\geq 3,$ $\N(p)^{j}T(p^{\frac{3(j+1)}{2}}A,p^{\frac{3j+1}{2}}B,p^{k}).$ As $j\geq 1,$ we will always reduce to a cubic exponential sum $T(p^rA,B,p^q),q>r\geq 1.$ It is easy to check by a linearizaion argument, as was done for the even $j$ case, that the sum is always zero.

It is an easy check to see for $k=2,j=1$ the sum is $\N(p).$

Now we break the cases of $j >k/2$ into cases where $k$ is even and odd, respectively.
Suppose first $k$ is even, then $$\sum_{x(p^k)} e(\frac{A x^3-p^j Bx}{p^k})= \N(p)^{k/2}
\sum_{\substack{a(p^{k/2})\\3a^2\equiv 0(p^{k/2})}} e(\frac{A a^3 + p^j B a}{p^k}).$$
Then $a=p^{\lceil{k/4}\rceil}y, y(p^{k/2-\lceil{k/4}\rceil}).$ Now if $j=h +k/2$, $h>k/4,$ the sum 
reduces to $\N(p)^{k/2}T(A,0,p^{k-3\lceil{k/4}\rceil}).$ If $h\leq k/4,$ one gets $$\N(p)^{k/2}\sum_{y(p^{k/2-\lceil{k/4}\rceil})} e(\frac{Ap^
{3\lceil{k/4}\rceil} y^3 + p^{h+k/2+\lceil{k/4}\rceil } B y}{p^k})$$ This reduces to $$\N(p)^{k/2}T(A,Bp^{h+\lfloor
{k/4}\rfloor +\lceil{k/4}\rceil },p^{k-3\lceil{k/4}\rceil}).$$ The last equality following from $$k/2 + \lceil{k/
4}\rceil -3\lceil{k/4}\rceil =k/2-\lceil{k/4}\rceil -\lceil{k/4}\rceil =\lfloor{k/4}\rfloor-\lceil{k/4}\rceil.$$ The last 
equation is $-1$ if $k \neq 4q,q \in \N$ and $0$ else. 

If $k$ is odd, one uses the decomposition $x=a+p^{\frac{k-1}{2}}b$ to get $$\sum_{a(p^{(k-1)/2})}e(\frac{a^3+p^ja}{p^k}) \sum_{b(p^{(k+1)/2})}e(\frac{3(a^2b+p^{k-1}ab^2}{p^{(k+1)/2}}).$$
Since $k$ is odd, $e(\frac{p^{k-1}ab^2}{p^{(k+1)/2}})=1.$
Now if $(a,p)=1$ the internal sum is $0$. If $a=p^l y,$ with $l < \frac{k+1}{4}$ then 
again the internal sum is zero. If $l \geq \frac{k+1}{4},$ then writing the $j=h+\frac{k+1}{2},1 \leq h\leq \frac{k-1}{2},$ the internal sum is $\N(p)^{\frac{k+1}{2}},$ so 
the sum in the $a$ variable  $$\N(p)^{\frac{k+1}{2}}\sum_{a\equiv p^{\lceil(k+1)/4\rceil}y(p^{(k-1)/2}} 
e(\frac{p^{3\lceil{(k+1)/4}\rceil}y^3+p^{(k+1)/2+h+\lceil{(k+1)/4}\rceil}y}{p^k}).$$ Then this equals $$
\N(p)^{\frac{k+1}{2}}\sum_{y(p^{\lfloor{(k-3)/4}\rfloor})} e(\frac
{Ay^3+p^{h+\lfloor{(k+1)/4}\rfloor-\lceil{(k+1)/4}\rceil}y}{p^{\lfloor(k-3)/4\rfloor}}),$$ by a similar 
argument to when $k$ is even above.

So if $h \geq \lfloor(k-3)/4\rfloor,$ then we get $$\N(p)^{\frac{k+1}{2}}T(A,0,p^{\lfloor(k-3)/4\rfloor}).$$
And else, $$\N(p)^{\frac{k+1}{2}}T(A,Bp^{h+\lfloor{(k+1)/4}\rfloor-\lceil{(k+1)/4}\rceil},p^{\lfloor(k-3)/4\rfloor}).$$

\end{proof}

We now state an identity studied by several authors \cite{DI}, \cite{K}, \cite{LP}, \cite{P4}, and \cite{Y}.

\begin{theorem}\label{cubth}
Let  $(xm,c)=1,$ then
\begin{equation}\label{eq:ident} \sum_{k(c)} e(\frac{x k^3+mk}{c})=\sum_{y(c)^{*}} (\frac{\overline{x}y}{c})_3 e(\frac{y-m^3\overline{3^3 xy}}{c}).\end{equation}

\end{theorem}

This is a beautiful identity relating a cubic exponential to a twisted Kloosterman sum.  As mentioned, \cite{MR} state this identity is equivalent to the Fundamental Lemma in the problem they study. 
The reason we choose our field to contain the cube roots of unity, is this identity does not seem to exist for primes $p\equiv 2(3).$ Perhaps, the lack of an identity in this case over $\Q$ is related to studying the cubic Dirichlet series in this paper for automorphic representations is ``unnatural," while the symmetric cube is natural. 

The next part of the paper is very technical, and on a first read through, one should keep in mind $b=1.$ Then one only needs Proposition \ref{pprime}. However, we want to employ Hecke operators to isolate a single representation on the spectral side of the trace formula, and for that we need a Fourier coefficient $a_{b}(\Pi)$ with either $b$ a prime or cubic power of a prime. This is why we need to deal with the ``ramified" cases in Proposition \ref{ppro} below.

We now use Lemma \ref{zerob} and Theorem \ref{cubth} to prove the following proposition.

\begin{prop}\text{\{Unramified Case\}} \label{pprime}
Let $1\leq j < k, (wB,p)=1.$ Suppose $(b,p)=1$ then \begin{equation}\label{eq:ppp}\sum_{A  (p^k)^{*}}
 (\frac{A}{p^k})_3 e(\frac{b\overline{wA}}{p^k})  \sum_{x(p^k)} e(\frac{Aw^2 x^3+p^jBx}{p^k})=0.\end{equation}  Similarly if $1<j=k$ the sum is zero.

 Let $k=j=1, (bwB,p)=1,$ then \begin{equation}\label{eq:ppp08}\sum_{A  (p)^{*}}
 (\frac{A}{p})_3 e(\frac{b\overline{wA}}{p})  \sum_{x(p^k)} e(\frac{Aw^2 x^3+pBx}{p})=(\frac{w}{p})_3\mu(p)g(1,p)+(\frac{\overline{b}}{p})_3\N(p).\end{equation} 
 
 \end{prop}
 
 \begin{proof}
We deal with the first case of Lemma \ref{zerob}. Using Theorem \ref{cubth}, $$\N(p)^{j}T(Aw^2,B,p^{k-3j/2})=\N(p)^{j} \sum_{y(p^{k-3j/2})^{*}} (\frac{\overline{Aw^2}y}{p^{k-3j/2}})_3 e(\frac{y-B^3\overline{3^3 Aw^2y}}{p^{k-3j/2}}).$$ Gathering the $A$-sum in \eqref{eq:ppp}, we have $$ \sum_{A  (p^k)^{*}}
 (\frac{A}{p^k})_3 (\frac{\overline{Aw^2}y}{p^{k-3j/2}})_3 e(\frac{\overline{Aw^2}(b-B^3p^{3j/2}\overline
 {3^3w}y)}{p^k}).$$ Then since $(b-B^3p^{3j/2}\overline{3^3w}y,p)=1,$ we use the fact from \cite
 {BY} that $$g(r,p^k):=\sum_{x(p^k)^{*}} (\frac{x}{p^k})_3 e(\frac{rx}{p^k})=0$$ for $(r,p)=1$ and 
 $k>1.$ The other cases follow analogously.
 The case of when $1<j=k$ follows analogously as above. 
The case of $j=k=1$ reduces the $x$-sum to Gauss sums which by a change of variables gives $(\frac{w}{p})_3\mu(p)g(1,p)+(\frac{\overline{b}}{p})_3\N(p).$

\end{proof}
 
 The general case for $(b,p)>1$ with $p$ the modulus of the exponential sum is more difficult. 
\vspace{.1 in}
\begin{prop}\label{ppro} Let $(wB,p)=1.$ If $b=p^l$ with $l\geq k>1$ then \begin{equation*}\sum_{A  (p^k)^{*}}
 (\frac{A}{p^k})_3 e(\frac{b\overline{wA}}{p^k})  \sum_{x(p^k)} e(\frac{Aw^2 x^3+Bx}{p^k})\end{equation*} is zero. If $l\geq k=1$ then the sum equals $\mu(p)(\frac{w}{p})_3g(1,p).$  \end{prop}

  \begin{proof}
If $(b,p)>1$ but $j=0,$ then again $(b-B^3\overline{3^3}y,p)=1$ if $\frac{b}{p^k} \neq 1$ by the
 same argument as in Proposition \ref{pprime}. However, if $\frac{b}{p^k}=1$ or $l=k$ then the $x$-sum becomes a Ramanujan sum which is only non-zero if $k=1.$ In this case, we have $$\sum_{A  (p)^{*}}
 (\frac{A}{p})_3   \sum_{x(p)} e(\frac{Aw^2 x^3+Bx}{p})=\sum_{A  (p)^{*}}
 (\frac{A}{p})_3 \sum_{t(p)}(\frac{t\overline{Aw^2}}{p})_3e(\frac{t-B^3\overline{3^3}\overline{Aw^2t}}{p})=\mu(p)(\frac{w}{p})_3g(1,p).$$
\end{proof}
One can consider the above proposition as the $j=0$ case of the following propositions for $b=\{p,p^3\}.$

\begin{prop}\label{aco0}
Let $(wB,p)=1, j\geq 1$ and $b=p,$ then there is a nontrivial contribution to $\sum_{A  (p^{k})^{*}} (\frac{A}{p})_3 e(\frac{b\overline{wA}}{p^k})  \sum_{x(p^{k})} e(\frac{Aw^2 x^3+p^jBx}{p^{k}})$ in the following cases:
 
 \begin{equation}\label{eq:chee} \sum_{A  (p^k)^{*}} (\frac{A}{p^k})_3 e(\frac{p\overline{wA}}{p^k})  \sum_{x(p^k)} e(\frac{Aw^2 x^3+p^jBx}{p^k})= 
  \begin{cases}
\phi(p)(\frac{w}{p})_3g(1,p)& \mbox{ if } j\geq 1, k=1\\
\N(p)^2(\frac{w}{p})_3g(1,p), & \mbox{ if } j\geq1,k=2,\\ 
\end{cases}
\end{equation}
 
 \end{prop}

 \begin{proof}
We start with the trivial case of $k=1,j=1.$ Here by a standard check analogous to Proposition \ref{ppro}, we have the sum equal to $\phi(p)(\frac{\overline{w}}{p})_3g(1,p).$

Look at the case with $k=2.$ We write $x(p^2), x=a+p b, a(p), b(p)$ to get $$\sum_{x(p^2)} e(\frac{Aw^2 x^3+p^jBx}{p^2})=\N(p) \sum_{\substack{a(p)\\3w^2Aa^2\equiv 0(p)}} e(\frac{Aw^2 a^3+p^jBa}{p^2})=\N(p).$$ Here we used the fact that $j\geq1$ otherwise we are in the case of Proposition \ref{ppro}. Then the $A$-sum equals by a standard calculation $\N(p)(\frac{\overline{w}}{p})_3g(1,p).$



Other than these trivial cases, following Proposition \ref{pprime}, one gathers the $A$-sum. It can be checked that for each case of Lemma \ref{zerob}, these cases are zero for $k>2.$ Indeed, we study
$$\sum_{A(p^k)^{*}} (\frac{A}{p^k})_3 e(\frac{\overline{wA}}{p^{k-1}}) \sum_{x(p^k)} e(\frac{Aw^2 x^3+p^jBx}{p^k}).$$ If $j=0,$ then we use Katz's identity to get $$\sum_{y(p^k)} (\frac{y}{p^k})_3 e(\frac{y}{p^k})\sum_{A(p^k)^{*}}e(\frac{\overline{A}(\overline{w}p-B^3\overline{y3^3})}{p^k})=0$$ as $((\overline{w}p-B^3\overline{y3^3},p)=1$ and the $A$-sum is a Ramanujan sum. 

Now assume $j>1$ then from Lemma \ref{zerob} for the $x$-sum, either there is no dependence on $A$ and/or the power of the exponential sum is less than $k,$ say it drops to $k-q,$  $0<q<k.$ If there is no dependence on $A$ then the outside $A$-sum of the left hand side of \eqref{eq:chee}  is just a Gauss sum to a prime power modulus greater than one (as $k>2$) and so is zero. Now if the power of $k$ drops and the $x$-sum looks like $T(A,0,p^{k-q})$ then by Lemma \ref{t00y} if $k-q=r+3l,$ $$T(A,0,p^{k-q})=\N(p)^{2l}T(A,0,p^r).$$ If $r=0$ we are in the previous case of $A$-independence. If $r=1,$ then by a change of variables in the $A$-sum, we have a Gauss sum to a prime power modulus which is zero. If $r=2,$ the Gauss sum is just $\N(p)$ and so $A$-independent.

If the power of $k$ drops but looks like $T(A,Bp^i,p^{k-q}),$ then if $i=0$ an application of Katz's identity and we are in the case of $j=0$ above. If $i \neq 0$ we can reapply Lemma \ref{t00y} until $i=0,$ and again apply Katz's identity and use the $j=0$ case that Ramanujan sums to prime power modulus are zero.


 

\end{proof}

 \begin{prop}\label{aco01}
  Let $ (wB,p)=1,j\geq 1$ and $b=p^3,$ then there is a nontrivial contribution to the above equation in the following cases

\begin{equation} \sum_{A  (p^k)^{*}} (\frac{A}{p^k})_3 e(\frac{p^3\overline{wA}}{p^k})  \sum_{x(p^k)} e(\frac{Aw^2 x^3+p^jBx}{p^k})= 
  \displaystyle
  \begin{cases}
\phi(p)(\frac{w}{p})_3g(1,p)& \mbox{ if } j\geq 1, k=1,\\
\phi(p^3)\N(p)^{2}& \mbox{ if } j\geq2, k=3,\\
\N(p)^{5}\left[ (\frac{w}{p})_3g(1,p)\mu(p)+\N(p)e(\frac{B^3\overline{3^3w}}{p})\right], & \mbox{ if } j=2,k=4,\\ \\
\N(p)^{5}\left[(\frac{w}{p})_3g(1,p)\mu(p)+\N(p)\right], & \mbox{ if } j\geq3,k=4,
\\ \N(p)^{k+2} e(\frac{B^3\overline{3^3w}}{p^{k-3}}) , & \mbox{ if } j=2,k\geq 5.
\end{cases}
\end{equation}

\end{prop}

\begin{proof}
We start with the trivial case of $k=1,j=1.$ Here we just have a Gauss sum and with a change of variables, we have the sum equal to $\phi(p)(\frac{\overline{w}}{p})_3g(1,p).$

 Using Proposition \ref{ppro} for $k=2,j=0,$ this case is zero. In the case $k=2,j\geq 1$ we write $x(p^2), x=a+p b, a(p), b(p)$ to get $$\sum_{x(p^2)} e(\frac{Aw^2 x^3+p^jBx}{p^2})=\N(p) \sum_{\substack{a(p)\\3w^2Aa^2\equiv 0(p)}} e(\frac{Aw^2 a^3+p^jBa}{p^2})=\N(p).$$ Here we used the fact that $j\geq1$ otherwise we are in the case of Proposition \ref{ppro}. Then the $A$-sum equals a complete character sum which is zero.

Take the case $k=3.$ We write $x(p^3), x=a+p^2 b, a(p^2), b(p)$ to get $$\sum_{x(p^3)} e(\frac{Aw^2 x^3+p^jBx}{p^3})=\N(p) \sum_{\substack{a(p^2)\\3w^2Aa^2\equiv 0(p)}} e(\frac{Aw^2 a^3+p^jBa}{p^3}).$$ As $3Aw$ are units, this equals $$\N(p) \sum_{l(p)} e(\frac{Aw^2 p^3l^3+p^{j+1}Bl}{p^3})=\N(p) \sum_{l(p)} e(\frac{Bl}{p^{2-j}}).$$
Clearly if $j \geq 2$ the $l$-sum is $\N(p)$ and zero for $j<2.$ Assuming that $j \geq 2,$ the $A$-sum is trivially $\phi(p^3).$

Take the case $k=4.$ By analogous calculation to the case $k=3$ we write $x(p^4), x=a+p^3 b, a(p^2), b(p)$ to get $$\N(p) \sum_{l(p^2)} e(\frac{Aw^2 p^3l^3+p^{j+1}Bl}{p^4}).$$ If $j=1$ then by another linearization argument the sum is zero. 
 
In the case $j=2$ we need to use Lemma \ref{zerob} to get  \begin{multline}\label{eq:ppp1}\sum_{A  (p^{4})^{*}} (\frac{A}{p})_3 e(\frac{p^3\overline{wA}}{p^4})  \sum_{x(p^{4})} e(\frac{Aw^2 x^3+p^2Bx}{p^{4}})=\\ \sum_{A  (p^{4})^{*}} (\frac{A}{p})_3 e(\frac{p^3\overline{wA}}{p^4})\left[\N(p)^2\sum_{t(p)^{*}} e(\frac{Aw^2t+Bt}{p})\right]=\\\N(p)^2\sum_{A  (p^{4})^{*}} (\frac{A}{p})_3 e(\frac{p^3\overline{wA}}{p^4})\sum_{t(p)^{*}} (\frac{\overline{Aw^2}t}{p})_3 e(\frac{t-B^3\overline{3^3Aw^2t}}{p})=\\ \N(p)^2\sum_{t(p)^{*}} (\frac{\overline{w^2}t}{p})_3 e(\frac{t}{p})  \sum_{A  (p^{4})^{*}} e(\frac{\overline{A}(\overline{w}-B^3\overline{3^3w^2t})}{p})=\\ \N(p)^5 (\frac{w}{p})_3\sum_{t(p)^{*}} (\frac{t}{p})_3 e(\frac{t}{p})  
 \sum_{\substack{q|p\\ t\equiv B^3\overline{3^3w}(q)}} \mu(\frac{p}{q}) \N(q)=\\  \N(p)^5 \left[g(1,p)\mu(p) (\frac{w}{p})_3+\N(p)e(\frac{B^3\overline{3^3w}}{p})\right].
  \end{multline}

  By similar analysis for the case $j\geq3,$ we have
 \begin{equation}\label{eq:l3}\N(p)^{5}\left[\N(p)+(\frac{w}{p})_3g(1,p)\mu(p)\right].\end{equation} 
 
 Now look at $k=5.$ From Proposition \ref{ppro} and Lemma \ref{zerob} we do not need to consider the cases $j=0,j=1.$ However, for $j=2,$ from Lemma \ref{zerob}, $$\sum_{A  (p^{5})^{*}} (\frac{A}{p^5})_3 e(\frac{p^3\overline{wA}}{p^5})  \sum_{x(p^{5})} e(\frac{Aw^2 x^3+p^2Bx}{p^{5}})=$$ 
 $$\N(p)^2 \sum_{A  (p^{5})^{*}} (\frac{A}{p^2})_3 e(\frac{\overline{wA}}{p^2}) T(Aw^2,B,p^2)=\N(p^7)e(\frac{B^3\overline{3^3w}}{p^2}).$$
 
 For $j\geq 3,$ by Lemma \ref{zerob} the $x$-sum is $\N(p)$ and so the $A$-sum is zero as it is a Gauss sum to prime power modulus.

Consider the case $k>5.$ Note in this case if the $x$-sum is independent of $A$ then the $A$-sum is zero as it is a Gauss sum to prime power modulus. By inspecting Lemma \ref{zerob}, except in the case $j=2,$ this is exactly what happens. So in this one case with $j=2$ we have $$\sum_{A  (p^{k})^{*}} (\frac{A}{p^k})_3 e(\frac{\overline{wA}}{p^{k-3}})  \sum_{x(p^{k})} e(\frac{Aw^2 x^3+p^2Bx}{p^{k}})=$$
$$\N(p)^2 \sum_{A  (p^{k})^{*}} (\frac{A}{p^k})_3 e(\frac{\overline{wA}}{p^{k-3}})T(A,B,p^{k-3})=\N(p)^{k+2} e(\frac{B^3\overline{3^3w}}{p^{k-3}}).$$ This calculation agrees with the case $k=5,j=2.$ 
So we have a nontrivial contribution in the case $k\geq 5,j=2.$

\end{proof}


\subsection{Reduction from cubic exponential sums to Ramanujan sums}

Let us look at the case $(mbw,c)=1, m\neq 0.$ Incorporating the identity of Theorem \ref{cubth} into the opened Kloosterman sum in \eqref{eq:poissonm} we have \begin{multline}\label{eq:myo}
\sum_{x  (c)^{*}}
(\frac{x }{c})_3 e(\frac{b\overline{wx}}{c})  \sum_{k(c)} e(\frac{\overline{w}(x k^3+mk)}{c}) =\\  \sum_{x  (c)^{*}}
(\frac{x }{c})_3 e(\frac{b\overline{wx}}{c})  \sum_{y(c)^{*}} (\frac{\overline{x} wy}{c})_3 e(\frac{y-m^3\overline{3^3 yw^2}\overline{x}}{c})= 
\\  (\frac{ w}{c})_3\sum_{y(c)^{*}}  (\frac{ y}{c})_3 e(\frac{y}{c}) \sum_{x  (c)^{*}}   e(\frac{\overline{x}(b\overline{w}-m^3\overline{3^3w^2y})}{c}). 
 \end{multline}
So what the identity did is take complicated cubic exponential sums to manageable Ramanujan sums. We should also point out that by Lemma \ref{pprime} if  $(m,c)>1,$ then if $c$ is square-full the sum is zero while if $c$ is square-free the sum is non-zero and the identity of Katz is not needed in this latter case.

 Now we simplify \eqref{eq:myo}.  
\begin{lemma}\label{ccc}
If $(mwb,c)=1,$  then \eqref{eq:myo} equals
$$ (\frac{ w}{c})_3\sum_{y(c)^{*}}  (\frac{ y}{c})_3 e(\frac{y}{c}) \sum_{x  (c)^{*}}   e(\frac{\overline{x}(b\overline{w}-m^3\overline{3^3w^2y})}{c})= (\frac{ w}{c})_3\sum_{(q),q|c} \mu(\frac{c}{q})\N(q)  \sum_{\substack{y(c)^{*}\\ y\equiv m^3\overline{b3^3w}(q)}} (\frac{y}{c})_3e(\frac{y}{c}).$$ 
\end{lemma}

\begin{proof}
 The interior sum is a Ramanujan sum, hence is equal to $$\sum_{q|(c,b\overline{w}-m^3\overline{3^3w^2y} )} \mu(\frac{c}{q})\N(q).$$ Inverting the $q$ and $y$-sum and a change of variables gives the result.
 \end{proof}

We now prove the $y$-sum is multiplicative in $c$ and $q$.

\begin{lemma}
Let $c=df, (d,f)=1$ and $q|c, q=wv,$ where $ w|d, v|f.$ Then $$\sum_{\substack{y(c)^{*}\\ y\equiv \overline{b}(q)}} (\frac{y}{c})_3e(\frac{y\overline{ 3^3}m^3}{c})=(\frac{d}{f})_3(\frac{f}{d})_3 \left[ \sum_{\substack{x(d)^{*}\\ x\equiv \overline{b}(w)}} (\frac{x}{d})_3e(\frac{x\overline{ 3^3}m^3}{d}) \right] \left[ \sum_{\substack{z(f)^{*}\\ z\equiv \overline{b}(v)}} (\frac{y}{f})_3e(\frac{y\overline{ 3^3}m^3}{f})\right].$$
\end{lemma}
\begin{proof}
This is an application of the Chinese Remainder Theorem. This is the same proof as for a normal Gauss sum, but we point out for the usual bijection $y \mod df \to (x \mod d, z \mod f)$ via $y=f\overline{f}x+d\overline{d}z,$ where $\overline{d}d\equiv 1(f), \overline{f}f\equiv 1(d),$ the congruence $$y\equiv  f\overline{f}x+d\overline{d}z \equiv 1(q)$$ implies $$f\overline{f}x \equiv 1 (w).$$ Further $ f\overline{f} \equiv 1 (w),$ hence the result.
\end{proof}

We now can focus on a prime and prime power modulus.

\begin{lemma}\label{hurt} Let $c=p^k,k>1$ and $q=p^r, 0\leq r \leq k.$ If $(mb,p)=1,$ then

 $$(\frac{w}{p^k})_3 \sum_{q|p^k} \mu(\frac{p^k}{q})\N(q)  \sum_{\substack{y(p^k)^{*}\\ y\equiv m^3\overline{b3^3w}(q)}} (\frac{y}{p^k})_3e(\frac{y}{p^k})=0$$ unless $r=k.$ If $r=k>1,$ then the sum equals $$(\frac{w}{p^k})_3 \N(p^k) (\frac{\overline{wb}}{p^k})_3e(\frac{m^3\overline{b3^3w}}{p^k})=N(p^k)(\frac{\overline{b}}{p^k})_3e(\frac{m^3\overline{b3^3w}}{p^k}).$$

If $r=k=1,$ then the sum equals $\N(p) (\frac{\overline{b}}{p})_3e(\frac{m^3\overline{wb3^3}}{p})+ (\frac{w}{p})_3\mu(p)g(1,p).$
\end{lemma}
\begin{proof}
 Assume $k>1,$ then the cubic residue character is a primitive character of conductor $p,$ hence the sum is $$(\frac{w}{p^k})_3 (\frac{wb}{p^k})_3e(\frac{\overline{ 3^3wb}m^3}{p^{k-r}})\sum_{l(p^{k-r)}} e(\frac{ l\overline{b w3^3}m^3}{p^{k-r}})=0,$$ as $(m,p)=1.$ 
Thus only $q=p^k$ contributes to \eqref{eq:poissonm}, and the LHS of Lemma \ref{ccc} equals $(\frac{w}{p^k})_3\N(p^k) (\frac{\overline{wb}}{p^k})_3e(\frac{m^3\overline{wb3^3}}{p^k}).$ If $k=1,$ then we have two terms $\N(p) (\frac{\overline{b}}{p})_3e(\frac{m^3\overline{wb3^3}}{p})+ (\frac{w}{p})_3\mu(p)g(1,p).$ \end{proof}

\begin{remark}
Note in the case that $c=p, (p,m)>1$ from Lemma \ref{pprime}, the answer agrees with the term in Lemma \ref{hurt} for $r=k=1$ with $m \equiv 0(p).$ 

\end{remark}
 
What is left over after after the use of this lemma to \eqref{eq:poissonm}? We can rewrite the equation as

\begin{multline}\label{eq:poissonm21} \frac{1}{X^{3/2}} \big[\sum_{\substack{c\equiv1(3)\\c\equiv 0(D)\\(c,b)=1}} \sum_{m} \frac{1}{\N(c)}\sum_{x  (c)^{*}}
 (\frac{x}{c})_3 e(\frac{b\overline{x}}{c})  \sum_{k(c)} e(\frac{x k^3-mk}{c})\ww_{m}(\frac{\N(c)}{X^{3/2}})+\\ \sum_{\substack{c\equiv1(3)\\c\equiv 0(D)\\(c,b)>1}} \sum_{m} \frac{1}{\N(c)}\sum_{x  (c)^{*}}
 (\frac{x}{c})_3 e(\frac{b\overline{x}}{c})  \sum_{k(c)} e(\frac{x k^3-mk}{c})\ww_{m}(\frac{\N(c)}{X^{3/2}})\big].
 \end{multline} 
 





Let us assume $(b,c)=1$ first. From Lemmas \ref{pprime}, \ref{hurt} using multiplicativity we have $$\sum_{x  (c)^{*}}(\frac{x}{c})_3 e(\frac{b\overline{x}}{c})  \sum_{k(c)} e(\frac{x k^3-mk}{c})=\sum_{f|c,(\frac{c}{f},f)=1}  \mu(f) g(1,f) (\frac{\overline{b}f}{\frac{c}{f}})_3 \N(\frac{c}{f}) e(\frac{ m^3\overline{fb3^3}}{\frac{c}{f}}).$$ Note also by the same lemmas this term is only non-zero when $(\frac{c}{f},m)=d$ with $d$ squarefree.

We can write using cubic reciprocity and inverting the $f$- and $c$-sums \begin{multline}\label{eq:aflemma}\frac{1}{X^{3/2}} \big[\sum_{\substack{c\equiv1(3)\\ c\equiv 0(D)\\(c,b)=1}} \sum_{m} \frac{1}{\N(c)}\sum_{x  (c)^{*}}(\frac{x}{c})_3 e(\frac{b\overline{x}}{c})  \sum_{k(c)} e(\frac{x k^3-mk}{c})\ww_{m}(\frac{\N(c)}
 {X^{3/2}})=\\ \frac{1}{X^{3/2}} \sum_{m\neq 0} \sum_{d|m,(d,b)=1} \mu^2(d)   \sum_{f} \frac{\mu(f) g(1,f)}{\N(f)} \sum_{\substack{fc\equiv1(3)\\fc
 \equiv 0(D)\\d|c\\(p\frac{c}{d},m)=1}}  (\frac{\overline{b}f}{c})_3  e(\frac{ m^3\overline{fb3^3}}{c}) \ww_{m}(\frac{\N(fc)}{X^{3/2}}).\end{multline}

The point being that the $m$-sum and $c$-sum are co-prime up to a square-free factor $d$ that is also co-prime to $p.$
  


   Notice trivially we have the bound $O(X^{1/2})$ for \eqref{eq:aflemma} by noting that the sum is largest when $\N(c)\sim X^{3/2}$ and $\N(f)\sim 1.$ This is in fact the correct size of the main term. In the next sections we show this is true after summing the $m,d,f$ and $c$-sums.
  
 For $(b,c)>1,$ we look only look at the case $b=p.$
 

 \subsection{The case of $(b,c)>1$ and $b=p$}

We need a lemma before we get to the sum we need to analyze.
 
 \begin{lemma} \label{sim}
 Let $(wcm,p)=1$ with $c$ squarefull. Then for $i,j,k \in \Z$ 
$$ \sum_{x  (c)^{*}}
 (\frac{x}{c})_3 e(\frac{p^ip^k\overline{wx}}{c})  \sum_{k(c)} e(\frac{p^kw^2x k^3-p^{j+k}mk}{c})=\N(c)(\frac{p^{-i} }{c})_3e(\frac{p^{k+3j-i}m^3\overline{3^3w}}{c}).$$
 
 \end{lemma}
 
 \begin{proof}
 
As $c$ is squarefull, by similar reasoning to Lemma \ref{hurt} we have $$ (\frac{p^{-k}  w}{c})_3\sum_{t(c)^{*}}  (\frac{ t}{c})_3 e(\frac{t}{c}) \sum_{x  (c)^{*}}   e(\frac{\overline{x}(p^{i+k}\overline{w}-m^3p^{3j+2k}\overline{3^3w^2t})}{c})=\N(c)(\frac{p^{-k+k-i}  w\overline{w}}{c})_3e(\frac{p^{k+3j-i}m^3\overline{3^3w}}{c})=$$
$$\N(c)(\frac{p^{-i} }{c})_3e(\frac{p^{k+3j-i}m^3\overline{3^3w}}{c}).$$

 \end{proof}

 Ultimately Lemma \ref{sim} is the same as Lemma \ref{hurt}, but it makes clear the dependence on parameter $b$ which is ultimately needed for isolating representations on the spectral side of the trace formula. 
 
 Let us give an example:
 
 \begin{example}\label{ex1}
Say $b=p, p$ prime. Write $c=f\cc$ with $(mf\cc,p)=1,$ $(\cc,f)=1$ and $f$ square-free with $\cc$ square-full. Then by the Chinese remainder theorem \begin{multline}\sum_{x  (pf\cc)^{*}}
 (\frac{x}{pf\cc})_3 e(\frac{p\overline{x}}{pf\cc})  \sum_{k(pf\cc)} e(\frac{x k^3-mk}{pf\cc})=\\ \left[\sum_{x  (p)^{*}}
 (\frac{x}{p})_3 e(\frac{p\overline{f\cc x}}{p})  \sum_{k(p)} e(\frac{\overline{f\cc}(x k^3-mk)}{p})\right]\left[\sum_{a  (f)^{*}}
 (\frac{x}{f})_3 e(\frac{p\overline{p\cc x}}{f})  \sum_{k(f)} e(\frac{\overline{p\cc}(x k^3-mk)}{f})\right] \times \\ \left[\sum_{d  (\cc)^{*}}
 (\frac{x}{\cc})_3 e(\frac{p\overline{pfx}}{\cc})  \sum_{k(\cc)} e(\frac{\overline{pf}(x k^3-mk)}{\cc})\right].\end{multline}
 
 By Proposition \ref{ppro} and Lemma \ref{hurt} this equals \begin{multline} \left[ (\frac{f\cc}{p})_3\mu(p) g(1,p) \right] \left[ \N(f)(\frac{\overline{p}}{f})_3 e(\frac{m^3\overline{p^2\cc3^3}}{f})+(\frac{p\cc}{f})_3\mu(f)g(1,f)\right]\left[ \N(\cc)(\frac{\overline{p}}{\cc})_3 e(\frac{m^3\overline{p^2f3^3}}{\cc})\right]= \\ (\frac{f\cc}{p})_3\mu(p) g(1,p)\N(f)(\frac{\overline{p}}{f})_3 e(\frac{m^3\overline{p^2\cc3^3}}{f})\N(\cc)(\frac{\overline{p}}{\cc})_3 e(\frac{m^3\overline{p^2f3^3}}{\cc})+\\ (\frac{f\cc}{p})_3\mu(p) g(1,p)(\frac{p\cc}{f})_3\mu(f)g(1,f)\N(\cc)(\frac{\overline{p}}{\cc})_3 e(\frac{m^3\overline{p^2f3^3}}{\cc})=\\ \mu(p) g(1,p)\N(f\cc) e(\frac{m^3\overline{p^23^3}}{f\cc})+\mu(p) g(1,p)\mu(f)g(1,f)\N(\cc)(\frac{f^2}{p})_3(\frac{\cc}{f})_3e(\frac{m^3\overline{p^2f3^3}}{\cc}).
 \end{multline}
 
 The last line following from the Chinese remainder theorem. Up to the power of $p$ in the exponential, the last terms above can be considered ``generic" terms in which we will sum over $\cc,f,$ and $m$ in the sections below.
 
 \end{example}

For the first case when the second equation in \eqref{eq:poissonm21} is non-zero, we use Proposition \ref{ppro}. There we only have a non-zero contribution when $p||c, p\nmid m.$ 
 The first case can be written using Chinese remainder theorem and Proposition \ref{ppro} as \begin{equation}\label{eq:dodo07} \mu(p)g(1,p)\sum_{\substack{m\\(m,p)=1\\m\neq 0}}  \sum_{\substack{c\equiv1(3)\\c\equiv 0(D)}} (\frac{c}{p})_3 \frac{1}{\N(pc)}\sum_{x  (c)^{*}}
 (\frac{x}{c})_3 e(\frac{p\overline{px}}{c})  \sum_{k(c)} e(\frac{\overline{p}(x k^3-mk)}{c})\ww_{m}(\frac{\N(pc)}{X^{3/2}})\big].\end{equation}

We apply Lemma \ref{sim}($i=1,k=-1,j=0$) in this case to the $x$- and $k$-sums modulo $c,$ we get for each squarefree $f,f|c$ 
 
 \begin{equation}\label{eq:sqep07}
\mu(f) g(1,f)  \left[\N(c) (\frac{p^2f}{c})_3 (\frac{f}{p})_3e(\frac{\overline{f p^23^3}m^3}{c})\right].
\end{equation}
Note this is virtually the same calculation as done in the above example. 
 
 So we can write \eqref{eq:dodo07} as \begin{multline}\label{eq:baflemm09}\mu(p)g(1,p) \sum_{\substack{m\\(m,p)=1\\m\neq 0}} \sum_{\substack{c\equiv1(3)\\c\equiv 0(D)}} (\frac{c}{p})_3  \frac{1}{\N(pc)}\sum_{x  (c)^{*}}
 (\frac{x}{c})_3 e(\frac{p\overline{px}}{c})  \sum_{k(c)} e(\frac{x k^3-mk}{c})\ww_{m}(\frac{\N(pc)}{X^{3/2}})=\\ \frac{\mu(p)g(1,p)}{\N(p)}  \sum_{ (m,p)=1,m\neq 0}\sum_{d|m,(d,p)=1}\mu^2(d) \sum_{f} \frac{\mu(f) g(1,f)(\frac{f^2}{p})_3}{\N(f)}\times \\ \sum_{\substack{fc\equiv1(3)\\fc\equiv 0(D)\\d|c\\(\frac{c}{d},mp)=1}} \left[ (\frac{c}{p})_3 (\frac{c}{p^2})_3\right](\frac{c}{f})_3  e(\frac{\overline{f p^23^3}m^3}{c}) \ww_{m}(\frac{\N(pfc)}{X^{3/2}}).\end{multline} 
 Note the condition $(d,p)=1$ is again imposed as $(c,p)=1$ and $d|c.$
 
The first term from Proposition \ref{aco0} when $p||c, m\equiv 0(p)$ also gives a non-zero contribution to \eqref{eq:poissonm21} and can be written as \begin{equation}\label{eq:dodo} \phi(p)g(1,p)\sum_{\substack{m\\m\neq 0}} \sum_{\substack{c\equiv1(3)\\c\equiv 0(D)}} (\frac{c}{p})_3 \frac{1}{\N(pc)}\sum_{x  (c)^{*}}
 (\frac{x}{c})_3 e(\frac{p\overline{px}}{c})  \sum_{k(c)} e(\frac{\overline{p}(x k^3-pmk)}{c})\ww_{pm}(\frac{\N(pc)}{X^{3/2}})\big].\end{equation}

 Again apply Lemma \ref{sim}$(i=1,k=-1,j\geq1)$ to get for each squarefree $f$ such that $f|c$ 
 
 \begin{equation}\label{eq:sqep}
\mu(f) g(1,f)  \left[\N(c) (\frac{p^2f}{c})_3(\frac{f}{p})_3 e(\frac{\overline{f 3^3}pm^3}{c})\right].
\end{equation}

We reduce this case to studying \begin{multline}\label{eq:baflemm} \phi(p)g(1,p)\sum_{\substack{c\equiv1(3)\\c\equiv 0(D)}} (\frac{c}{p})_3 \sum_{\substack{m\\m\neq 0}} \frac{1}{\N(pc)}\sum_{x  (c)^{*}}
 (\frac{x}{c})_3 e(\frac{p\overline{px}}{c})  \sum_{k(c)} e(\frac{\overline{p}(x k^3-pmk)}{c})\ww_{pm}(\frac{\N(pc)}{X^{3/2}})=\\ \frac{\phi(p)g(1,p)}{\N(p)}  \sum_{ m\neq 0}\sum_{d|m,(d,p)=1}\mu^2(d)   \sum_{f} \frac{\mu(f) g(1,f) (\frac{f^2}{p})_3}{\N(f)} \sum_{\substack{fc\equiv1(3)\\fc\equiv 0(D)\\d|c\\(\frac{c}{d},mp)=1}}   \left[ (\frac{c}{p})_3 (\frac{c}{p^2})_3\right] (\frac{c}{f})_3 \times \\ e(\frac{p m^3\overline{f3^3}}{c}) \ww_{pm}(\frac{\N(pfc)}{X^{3/2}}).\end{multline}  
 
 The second term of Proposition \ref{aco0} is when $p^2||c, p|m:$ 
 
 \begin{equation}\label{eq:dodo} \N(p)^2g(1,p)\sum_{\substack{m \\m\equiv 0(p)}} \sum_{\substack{c\equiv1(3)\\c\equiv 0(D)}} (\frac{c}{p})_3  \frac{1}{\N(p^2c)}\sum_{x  (c)^{*}}
 (\frac{x}{c})_3 e(\frac{p\overline{p^2x}}{c})  \sum_{k(c)} e(\frac{\overline{p^2}(x k^3-mk)}{c})\ww_{m}(\frac{\N(p^2c)}{X^{3/2}})\big].\end{equation}

 Which simplifies using Lemma \ref{sim}$(i=1,k=-2,j\geq1)$ to 
 
\begin{equation}\label{eq:momo} g(1,p)    \sum_{ m ,m\equiv 0(p) }\sum_{d|m,(d,p)=1}\mu^2(d)  \sum_{\substack{fc\equiv1(3)\\fc\equiv 0(D)\\d|c\\(\frac{c}{d},m)=1}} \sum_{f} \frac{\mu(f) g(1,f) }{\N(f)} \left[(\frac{p^2}{c})_3(\frac{p}{c})_3 \right] (\frac{c}{f})_3  e(\frac{ m^3\overline{f3^3}}{c}) \ww_{m}(\frac{\N(p^2fc)}{X^{3/2}}).\end{equation}
 

 

 \subsection{The case of $(b,c)>1$ and $b=p^3$} 

We give another example
\begin{example}\label{ex2}
Say $b=p^i, p$ prime. Write $c=p^kf\cc$ with $(mf\cc,p)=1,$ $(\cc,f)=1$ and $f$ square-free with $\cc$ square-full. Then by the Chinese remainder theorem \begin{multline}\sum_{x  (p^kf\cc)^{*}}
 (\frac{x}{p^kf\cc})_3 e(\frac{p^i\overline{x}}{p^kf\cc})  \sum_{k(p^kf\cc)} e(\frac{x k^3-p^jmk}{p^kf\cc})=\\ \left[\sum_{x  (p^k)^{*}}
 (\frac{x}{p^k})_3 e(\frac{p^i\overline{f\cc x}}{p^k})  \sum_{k(p^k)} e(\frac{\overline{f\cc}(x k^3-p^jmk)}{p^k})\right]\left[\sum_{a  (f)^{*}}
 (\frac{x}{f})_3 e(\frac{p^i\overline{p^k\cc x}}{f})  \sum_{k(f)} e(\frac{\overline{p^k\cc}(x k^3-p^jmk)}{f})\right] \times \\ \left[\sum_{d  (\cc)^{*}}
 (\frac{x}{\cc})_3 e(\frac{p^i\overline{p^kfx}}{\cc})  \sum_{k(\cc)} e(\frac{\overline{p^kf}(x k^3-p^jmk)}{\cc})\right].\end{multline}
 
 By Proposition \ref{ppro},\ref{aco01}, and Lemma \ref{hurt} this equals \begin{multline} \left[ \mbox{ Apply Prop. } \ref{ppro},\ref{aco01} \right] \left[ \N(f)(\frac{\overline{p^i}}{f})_3 e(\frac{m^3p^{3j-i-k}\overline{\cc3^3}}{f})+(\frac{p^k\cc}{f})_3\mu(f)g(1,f)\right]\times \\ \left[ \N(\cc)(\frac{\overline{p^i}}{\cc})_3 e(\frac{m^3p^{3j-i-k}\overline{f3^3}}{\cc})\right]= \\ \bigg[ \mbox{ Apply Prop. } \ref{ppro},\ref{aco01} \bigg] \bigg( (\frac{\overline{f\cc}}{p^i})_3\N(f\cc) e(\frac{m^3\overline{p^{3j-i-k}3^3}}{f\cc})+\\ (\frac{f}{p^k})_3(\frac{\cc}{f})_3(\frac{\overline{\cc}}{p^i})_3\mu(f)g(1,f)\N(\cc)e(\frac{m^3\overline{p^{3j-i-k}f3^3}}{\cc})\bigg).
 \end{multline}
 
 The last line following from the Chinese remainder theorem. Up to the power of $p$ in the exponential, the last terms above can be considered ``generic" terms in which we will sum over $\cc,f,$ and $m$ in the sections below.
 
 \end{example}

From Proposition \ref{ppro} the first case to consider is $k=1, j=0$ which we can write as \begin{equation}\label{eq:dodo77} \mu(p)g(1,p)\sum_{\substack{m\\(m,p)=1\\m\neq 0}} \sum_{\substack{c\equiv1(3)\\c\equiv 0(D)\\(p,c)=1}} (\frac{c}{p})_3  \frac{1}{\N(pc)}\sum_{x  (c)^{*}}
 (\frac{x}{c})_3 e(\frac{p^3\overline{px}}{c})  \sum_{k(c)} e(\frac{\overline{p}(x k^3-mk)}{c})\ww_{m}(\frac{\N(pc)}{X^{3/2}})\big].\end{equation}

This can be reduced using Lemma \ref{sim}$(i=3,k=-1,j=0)$ to $$\frac{\mu(p)g(1,p)}{\N(p)} \sum_{m\neq 0} \sum_{d|m,(d,p)=1}\mu^2(d) \sum_{f,(f,p)=1} \frac{\mu(f) g(1,f)(\frac{f}{p^2})_3 }{\N(f)}\sum_{\substack{fc\equiv1(3)\\fc\equiv 0(D)\\d|c\\ (\frac{c}{d},mp)=1}} (\frac{c}{f})_3 e(\frac{ m^3\overline{p^4f3^3}}{c}) \ww_{m}(\frac{\N(pfc)}{X^{3/2}}).$$

The next case is when $p||c, m\equiv 0(p)$ also gives a non-zero contribution to \eqref{eq:poisso} and can be written as \begin{equation}\label{eq:dodo} \phi(p)g(1,p)\sum_{\substack{m\\m\neq 0}} \sum_{\substack{c\equiv1(3)\\c\equiv 0(D)}} (\frac{c}{p})_3 \frac{1}{\N(pc)}\sum_{x  (c)^{*}}
 (\frac{x}{c})_3 e(\frac{p^3\overline{px}}{c})  \sum_{k(c)} e(\frac{\overline{p}(x k^3-pmk)}{c})\ww_{pm}(\frac{\N(pc)}{X^{3/2}})\big].\end{equation}

 Again apply Lemma \ref{sim}$(i=3,k=-1,j\geq1)$ to get for each squarefree $f$ such that $f|c$ 
 
 \begin{equation}\label{eq:sqep}
\mu(f) g(1,f)  \left[\N(c)(\frac{f}{cp})_3 e(\frac{\overline{f 3^3}pm^3}{c})\right].
\end{equation}

We reduce this case to studying \begin{multline}\label{eq:baflemm} \phi(p)g(1,p)\sum_{\substack{c\equiv1(3)\\c\equiv 0(D)}} (\frac{c}{p})_3 \sum_{\substack{m\\m\neq 0}} \frac{1}{\N(pc)}\sum_{x  (c)^{*}}
 (\frac{x}{c})_3 e(\frac{p\overline{px}}{c})  \sum_{k(c)} e(\frac{\overline{p}(x k^3-pmk)}{c})\ww_{pm}(\frac{\N(pc)}{X^{3/2}})=\\ \frac{\phi(p)g(1,p)}{\N(p)}  \sum_{ m\neq 0}\sum_{d|m,(d,p)=1}\mu^2(d)  \sum_{\substack{fc\equiv1(3)\\fc\equiv 0(D)\\d|c\\(\frac{c}{d},mp)=1}}   \sum_{f} \frac{\mu(f) g(1,f) (\frac{f}{p^2})_3}{\N(f)}  (\frac{c}{pf})_3  e(\frac{ m^3\overline{f3^3p}}{c}) \ww_{pm}(\frac{\N(pfc)}{X^{3/2}}).\end{multline}

From Proposition \ref{aco01} we have a non-zero contribution when $p^3||c, p^2|m,$ which is

\begin{equation}\label{eq:dodo0} \phi(p^3)\N(p^2)\sum_{\substack{m\equiv 0(p^2)\\m\neq 0}}\sum_{\substack{c\equiv1(3)\\c\equiv 0(D)\\(p,c)=1}} \frac{1}{\N(p^3c)}\sum_{x  (c)^{*}}
 (\frac{x}{c})_3 e(\frac{\overline{p^3}p^3\overline{x}}{c})  \sum_{k(c)} e(\frac{\overline{p^3}(x k^3-mk)}{c})\ww_{m}(\frac{\N(p^3c)}{X^{3/2}})\big].\end{equation}

Writing $c=\cc y$ and applying Lemma \ref{sim}$(i=3,k=-3,j=2)$ to get for $f|y,$
\begin{equation}\label{eq:sqep07}
\mu(f) g(1,f) \left[\N(c)(\frac{f}{c})_3 e(\frac{\overline{f 3^3}m^3}{c})\right].
\end{equation}

We thus need to analyze 
$$ \N(p)\phi(p) \frac{1}{X^{3/2}} \sum_{m\neq 0} \sum_{d|m,(d,p)=1}\mu^2(d) \sum_{f} \sum_{\substack{fc\equiv1(3)\\fc\equiv 0(D)\\d|c\\ (\frac{c}{d},mp)=1}}   \frac{\mu(f) g(1,f)}{\N(f)}  \mathbf{1_{p}}(fc)(\frac{c}{f})_3  e(\frac{ m^3\overline{f3^3}}{c}) \ww_{p^2m}(\frac{\N(p^3fc)}{X^{3/2}}).$$

\subsubsection{The case of $p^4||c$ with $p^2||m$}\label{pca}

 

 
 Here we have again by Chinese remainder theorem and Proposition \ref{aco01}, \begin{multline} \label{eq:try}\N(p)^{5}\sum_{(m,p)=1 ,m\neq 0} \sum_{\substack{c\equiv1(3)\\c\equiv 0(D)\\(c,p)=1}}   \left[ (\frac{c}{p})_3g(1,p)\mu(p)+\N(p)e(\frac{m^3\overline{3^3c}}{p})\right] \frac{1}{\N(p^4c)}\sum_{x  (c)^{*}}
 (\frac{x}{c})_3 e(\frac{p^3\overline{p^4x}}{c}) \times \\ \sum_{k(c)} e(\frac{\overline{p^4}(x k^3-p^2mk)}{c})\ww_{p^2m}(\frac{\N(p^4c)}{X^{3/2}})\big]\end{multline}

By analogous arguments, we then look at \begin{multline*}\N(p) \sum_{(m,p)=1, m\neq 0} \sum_{d|m,(d,p)=1}\mu^2(d) \sum_{f} \frac{\mu(f) g(1,f)  (\frac{f}{p^4})_3  }{\N(f)} \sum_{\substack{fc\equiv1(3)\\fc\equiv 0(D)\\d|c\\ (\frac{c}{d},mp)=1}}     (\frac{c}{f})_3 \times \\    \left[ (\frac{fc}{p})_3g(1,p)\mu(p)+\N(p)e(\frac{m^3\overline{3^3fc}}{p})\right]  e(\frac{ m^3\overline{pf3^3}}{c}) \ww_{p^2m}(\frac{\N(p^4fc)}{X^{3/2}}).\end{multline*}

We split the above equation into two sums: \begin{multline*}\N(p)  g(1,p)\mu(p)  \sum_{(m,p)=1, m\neq 0}\sum_{d|m,(d,p)=1}\mu^2(d)  \sum_{f}  \frac{\mu(f) g(1,f)  (\frac{f}{p^2})_3  }{\N(f)} \times \\ \sum_{\substack{fc\equiv1(3)\\fc\equiv 0(D)\\d|c\\ (\frac{c}{d},mp)=1}}  (\frac{c}{fp})_3 e(\frac{ m^3\overline{pf3^3}}{c}) \ww_{p^2m}(\frac{\N(p^4fc)}{X^{3/2}}),\end{multline*}

and
\begin{multline*}\N(p)^2   \sum_{(m,p)=1, m\neq 0} \sum_{d|m,(d,p)=1}\mu^2(d) \sum_{f}  \frac{\mu(f) g(1,f)  (\frac{f}{p})_3  }{\N(f)} \times \\ \sum_{\substack{fc\equiv1(3)\\fc\equiv 0(D)\\d|c\\ (\frac{c}{d},mp)=1}}    (\frac{c}{f})_3 e(\frac{ m^3\overline{f3^3}}{pc}) \ww_{p^2m}(\frac{\N(p^4fc)}{X^{3/2}}),\end{multline*}

\subsubsection{The case of $p^4||c$ with $p^3|m$}\label{pca8}


In this case we have \begin{multline} \label{eq:try3}\N(p)^{5} \sum_{\substack{m\neq 0\\ m\equiv 0(p^3)}}  \sum_{\substack{c\equiv1(3)\\c\equiv 0(D)\\(c,p)=1}} \left[(\frac{c}{p})_3g(1,p)\mu(p)+\N(p)\right] \frac{1}{\N(p^4c)}\sum_{x  (c)^{*}}
 (\frac{x}{c})_3 e(\frac{p^3\overline{p^4x}}{c}) \times \\ \sum_{k(c)} e(\frac{\overline{p^4}(x k^3-p^3mk)}{c})\ww_{p^3m}(\frac{\N(p^4c)}{X^{3/2}})\big]\end{multline}

Similar to the above case using Lemma \ref{sim}$(i=3,k=-4,j=3)$ we get  \begin{multline}\N(p)    \sum_{ m\neq 0} \sum_{d|m,(d,p)=1}\mu^2(d) \sum_{f} \sum_{\substack{fc\equiv1(3)\\fc\equiv 0(D)\\d|c\\ (\frac{c}{d},mp)=1}} \frac{\mu(f) g(1,f)  (\frac{f}{p^4})_3  }{\N(f)}  (\frac{c}{f})_3\left[(\frac{fc}{p})_3g(1,p)\mu(p)+ \N(p)\right] \times \\ e(\frac{ p^2m^3\overline{f3^3}}{c}) \ww_{p^3m}(\frac{\N(p^4fc)}{X^{3/2}}).\end{multline}

We write this as two sums: \begin{multline*}\N(p) g(1,p)\mu(p)   \sum_{ m\neq 0} \sum_{d|m,(d,p)=1}\mu^2(d) \sum_{f} \sum_{\substack{fc\equiv1(3)\\fc\equiv 0(D)\\d|c\\ (\frac{c}{d},mp)=1}} \frac{\mu(f) g(1,f)  (\frac{f}{p^2})_3  }{\N(f)}  (\frac{c}{pf})_3 \times \\  e(\frac{ p^2m^3\overline{f3^3}}{c}) \ww_{p^3m}(\frac{\N(p^4fc)}{X^{3/2}}),\end{multline*} and 

$$\N(p)^2    \sum_{ m\neq 0} \sum_{d|m,(d,p)=1}\mu^2(d) \sum_{f} \sum_{\substack{fc\equiv1(3)\\fc\equiv 0(D)\\d|c\\ (\frac{c}{d},mp)=1}} \frac{\mu(f) g(1,f)  (\frac{f}{p})_3  }{\N(f)}    (\frac{c}{f})_3 e(\frac{ p^2m^3\overline{f3^3}}{c}) \ww_{p^3m}(\frac{\N(p^4fc)}{X^{3/2}}).$$

\subsubsection{The case of $p^{k}||c$ with $j=2$}

\begin{multline} \label{eq:trykk}\N(p)^{k+2}\sum_{(m,p)=1 ,m\neq 0} \sum_{\substack{c\equiv1(3)\\c\equiv 0(D)\\(c,p)=1}}  e(\frac{m^3\overline{3^3c}}{p^{k-3}}) \frac{1}{\N(p^kc)}\sum_{x  (c)^{*}}
 (\frac{x}{c})_3 e(\frac{p^3\overline{p^kx}}{c}) \times \\  \sum_{z(c)} e(\frac{\overline{p^k}(x z^3-p^2mz)}{c})\ww_{p^2m}(\frac{\N(p^kc)}{X^{3/2}})\big]\end{multline}

Using Lemma \ref{sim}$(i=3,k,j=2),$ an analogous calculation to above is

\begin{multline}\N(p)^2 \sum_{k=5}^\infty    \sum_{(m,p)=1, m\neq 0} \sum_{d|m,(d,p)=1}\mu^2(d) \sum_{f} \sum_{\substack{fc\equiv1(3)\\fc\equiv 0(D)\\d|c\\ (\frac{c}{d},mp)=1}}   \frac{\mu(f) g(1,f)  (\frac{f}{p})_3  }{\N(f)}  (\frac{c}{f})_3 \times \\ e(\frac{m^3\overline{3^3fc}}{p^{k-3}})e(\frac{ m^3\overline{f3^3p^{k-3}}}{c})  \ww_{p^2m}(\frac{\N(p^k fc)}{X^{3/2}}),\end{multline}

 \section{Executing the $c$- and $m$-sums}
 
 \subsection{Unramified case: $(bm,c)=1$}

We make no restriction on $b$ until it is needed. 

Using elementary reciprocity: $$\frac{\overline{A}}{B} +\frac{\overline{B}}{A} \equiv \frac{1}{AB} (1),$$ we have since $(c,3fb)=1,$ \begin{equation}
e(\frac{ m^3\overline{3^3fb}}{c}) =e(\frac{-m^3\overline{c}}{3^3fb})e(\frac{m^3}{3^3fcb}).
\end{equation}

In this section we prove

\begin{prop}\label{pb3}
Let $g \in C^{\infty}_0(\R^{+}), \int_{\R} g(t)\sqrt{t}dt=1.$ For any $\epsilon> 0,$

 \begin{equation}\frac{1}{X^{3/2}} \sum_{m\neq 0,(m,p)=1} \sum_{d|m} \mu^2(d)   \sum_{f} \frac{\mu(f) g(1,f)}{\N(f)} \sum_{\substack{fc\equiv1(3)\\fc\equiv 0(D)\\d|c\\(p\frac{c}{d},m)=1}}  (\frac{\overline{b}f}{c})_3  e(\frac{ m^3\overline{fb3^3}}{c}) 
 \ww_{m}(\frac{\N(fc)}{X^{3/2}})=\end{equation}

\begin{equation} \left\{ \begin{array}{ll} \frac{4\pi^3}{9} \frac{X^{1/2} g(1,p)A_{b,D} }{6^2(\sqrt{-3})^4\N(Dp^{1/2})} h(V,1/3,0)+O(X^{2/7+\epsilon})& \mbox{if } b=p \\ O(X^{2/7+\epsilon}) & \mbox{if } b=p^j, j>1 \end{array} \right. \end{equation}
Here $A_{b,D}$ is defined in \eqref{eq:hug}.
\end{prop}

Recalling the definition of $\ww_{m}$  we can use cubic reciprocity to get the left hand side of Proposition \ref{pb3} and the first equation of \eqref{eq:aflemma} equal to 

\begin{multline}\label{eq:duuf00} 
\frac{1}{6\sqrt{-3}}\sum_{m\neq 0,(m,p)=1} \sum_{d|m,(d,p)=1} \frac{\mu^2(d)}{\N(d)}   \sum_{f} \frac{\mu(f) g(1,f)}{\N(f)^2} \sum_{\substack{fdc\equiv1(3)\\fdc
 \equiv 0(D)\\(pc,m)=1}} \frac{1}{\N(c)} (\frac{\overline{b}f}{dc})_3  e(\frac{ m^3\overline{fb3^3}}{dc}) \times \\ \int_{\C} e( \frac{-t\sqrt{X}m}{fdc})g(\N(t))
 V(\frac{4\pi\sqrt{(\sqrt{X}t)^3 b }}{fdc}) d^{+}t\end{multline}

The goal of the next section is use Poisson summation on the $c$-sum and show that the zeroth frequency is the main term. 

\subsubsection{Poisson summation with truncation}

Let us truncate the $f$-sum at $\N(f) \leq X^{A}.$ We optimize $A$ later. Writing $c$ as $dc$ and noting $\N(fdc)\sim X^{3/2}$ and $\N(m)\ll X^{1/2},$ we can bound \begin{multline}\label{eq:duuf} 
\frac{1}{6\sqrt{-3}}\sum_{m\neq 0,(m,p)=1} \sum_{d|m,(d,p)=1} \frac{\mu^2(d)}{\N(d)}   \sum_{\N(f)\geq X^{A}} \frac{\mu(f) g(1,f)}{\N(f)^2} \sum_{\substack{fdc\equiv1(3)\\fdc
 \equiv 0(D)\\(pc,m)=1}} \frac{1}{\N(c)} (\frac{\overline{b}f}{dc})_3  e(\frac{ m^3\overline{fb3^3}}{dc}) \times \\ \int_{\C} e( \frac{-t\sqrt{X}m}{fdc})g(\N(t))
 V(\frac{4\pi\sqrt{(\sqrt{X}t)^3 b }}{fdc}) d^{+}t\end{multline} by $X^{1/2+\epsilon-A/2}$ by using the standard bound on Gauss sums.

On the complementary piece of the $f$-sum from \eqref{eq:duuf00}, we want to isolate the $c$-sum in order to perform Poisson summation on it. The sum to investigate is \begin{multline}\label{eq:duuf0} 
\frac{1}{6\sqrt{-3}} \sum_{m\neq 0,(m,p)=1} \sum_{d|m,(d,p)=1} \frac{\mu^2(d)}{\N(d)}   \sum_{\N(f)\leq X^{A}} \frac{\mu(f) g(1,f)}{\N(f)^2} \sum_{\substack{fdc\equiv1(3)\\fdc
 \equiv 0(D)}} \frac{1}{\N(c)} (\frac{\overline{b}f}{dc})_3 \mathbf{1}_m(c)  e(\frac{ m^3\overline{fb3^3}}{dc}) \times \\ \int_{\C} e( \frac{-t\sqrt{X}m}{fdc})g(\N(t))
 V(\frac{4\pi\sqrt{(\sqrt{X}t)^3 b }}{fdc}) d^{+}t.\end{multline}

Using elementary reciprocity and detecting the congruence condition $c\equiv 1(3)$ by ray class Hecke characters modulo $3,$ we get

\begin{multline}\label{eq:newmo2}
 \frac{1}{\phi(3)6\sqrt{-3}} \sum_{\psi(3)}\frac{\psi(D)}{\N(D)}\sum_{m\neq 0,(m,p)=1} \sum_{d|m,(d,p)=1}\frac{\mu^2(d)}{\N(d)}   \sum_{\N(f)\leq X^A} \frac{\mu(f) \psi(f)g(1,f)}{\N(f)^2}   \times \\  \sum_{c} \frac{1}{\N(c)}\psi(dc) (\frac{\overline{dc}}{b})_3(\frac{dc}{f})_3 \mathbf{1}_m(c)  e(\frac{-m^3\overline{Ddc}}{fb3^3})e(\frac{m^3}{3^3bDfdc})\times \end{multline} $$ \int_{\C} e( \frac{-t\sqrt{X}m}{Dfdc})g(\N(t))
 V(\frac{4\pi
\sqrt{(\sqrt{X}t)^3 b }}{Dfdc}) d^{+}t. $$

 For the simplest case to understand for Poisson summation we assume first that the $m$- and $f$-sums are coprime. Then we perform Poisson summation on $c$ modulo $3^3fmb$ to get the $c$-sum equal to 

\begin{equation}\label{eq:newep}
\frac{1}{\sqrt{-3}\N(3^3mfb)} \sum_{k \in \ri} \bigg[\sum_{y(3^3mfb)} \psi(y)(\frac{\overline{dy}}{b})_3(\frac{dy}{f})_3\mathbf{1}_m(y)    e(\frac{-m^3\overline{Ddy}}{3^3fb}) e(\frac{-ky}{3^3mfb})\bigg]  \times \end{equation} $$\int_{\C}  \int_{\C}  e(\frac{m^3}{3^3bDfdz})  e( \frac{-t\sqrt{X}m}{Dfdz})g(\N(t))e(\frac{-kz}{3^3mfb})V(\frac{4\pi\sqrt{(\sqrt{X}t)^3b }}{Dfdz}) d^{+}t \frac{d^{+}z}{|z|^2}.
$$


We aim to show the $k=0$ term is the main term. We require bounds on the derivatives inside the above integral.

Note as $|m| \ll X^{1/4},$ the derivatives of $ e(\frac{m^3}{3^3bDfdz}) ,  e( \frac{-t\sqrt{X}m}{Dfdz}),$ and $V(\frac{4\pi\sqrt{(\sqrt{X}t)^3b }}{Dfdz})$ are bounded by  $|\frac{f d}{X^{3/4}}|.$




Assuming $k\neq 0,$ we apply integration by parts $2q$-times in the $z$-variable, which using the above estimates, gives the integral is bounded by $$\frac{\N(dm)^q \N(f)^{2q}}{\N(k)^{q}X^{3q/2}}.$$

Using the Weil bound for Kloosterman sums, $$\sum_{y(3^3mfb)} \psi(y)(\frac{\overline{y}}{b})_3(\frac{y}{f})_3  e(\frac{-m^3\overline{Dy}}{fb3^3}) e(\frac{-ky}{3^3mfb}) \ll 
(m^3,k,f)^{1/2} \N(f)^{1/2}=\N(f)^{1/2}$$ as we assumed $(m,f)=1.$ So \eqref{eq:newep} is bounded by $$\frac{1}
{\N(f)^{1/2}}\times \frac{\N(dm)^q \N(f)^{2q}}{\N(k)^{q}X^{3q/2}}=\frac{\N(dm)^q\N(f)^{2q-1/2}}{\N(k)^{q}X^{3q/2}}$$ Incorporating this back into \eqref{eq:newmo2}, and noting the integral is of size $O(1)$ (due to multiplicative Haar measure) we need to estimate 

$$\frac{1}{X^{3q/2}} \sum_{ \N(m)\ll X^{1/2}} \frac{1}{\N(m)^{1-q}} \sum_{d|m} \frac{1}{\N(d)^{1-q}}  \sum_{\N(f)\ll X^A} \frac{1}{\N(f)^{2-2q}}  \sum_{k} \frac{1}{\N(k)^{q}}=$$
$$\frac{1}{X^{3q/2}}\sum_{\N(d)\ll X^{1/2}} \frac{1}{\N(d)^{2-2q}}  \sum_{ \N(m)\ll \frac{X^{1/2}}{\N(d)}} \frac{1}{\N(m)^{1-q}}   \sum_{\N(f)\ll X^A} \frac{1}{\N(f)^{2-2q}}  \sum_{k} \frac{1}{\N(k)^{q}}.$$ This sum is bounded by an elementary calculation by $X^{q(2A-\frac{1}{4})-\frac{1}{2}-A},$ assuming $q \geq 2.$

For the best result we let $q=2$ and combine this estimate with the complementary $f$-sum estimate in \eqref{eq:duuf}. We can optimize the bounds using $A=\frac{3}{7}.$ Then for the first sum \eqref{eq:duuf} we have the bound for both of $O(X^{2/7+\epsilon}).$ 

Now consider the general case where $(m,f)>1.$ We can write this as $(m,f)=l>1$ and sum over $l.$ Interchanging sums we have

\begin{multline}\label{eq:newmo3}
 \frac{1}{\phi(3)6(\sqrt{-3})^2} \sum_{\psi(3)}\frac{\psi(D)}{\N(D)}\sum_{\N(l) \ll X^{\min(1/2,A)}} \frac{\mu(l) \psi(l)g(1,l)}{\N(l)^2}  \sum_{\substack{m\neq 0,(m,p)=1\\N(m) \ll \frac{X^{1/2}}{\N(l)}}} \sum_{d|m} \frac{\mu^2(d)}{\N(d)} \times \\   \sum_{\N(f)\leq \frac{X^A}{\N(l)}} \frac{\mu(f) \psi(f)g(1,f)}{\N(f)^2}     \sum_{c} \frac{1}{\N(c)}\psi(dc) (\frac{\overline{dc}}{b})_3(\frac{dc}{lf})_3 \mathbf{1}_{lm}(c)  e(\frac{-l^2m^3\overline{Ddc}}{fb3^3})e(\frac{l^2m^3}{3^3bDfdc})\times \\ \int_{\C} e( \frac{-t\sqrt{X}m}{Dfdc})g(\N(t))
 V(\frac{4\pi
\sqrt{(\sqrt{X}t)^3 b }}{Dlfdc}) d^{+}t\end{multline} 
Using a completely analogous argument of Poisson summation ($c\mod 3^3mlf$), integration by parts and the Weil bound for the Kloosterman, we have a similar estimate $$\frac{1}{X^{3q/2}} \sum_{\N(l) \ll X^{\min(A,1/2)}} \frac{1}{\N(l)^{2-4q}} \sum_{ \N(m)\ll \frac{X^{1/2}}{\N(l)}} \frac{1}{\N(m)^{1-2q}} \sum_{d|m} \frac{1}{\N(d)^{1-2q}}  \sum_{\N(f)\ll \frac{X^A}{\N(l)}} \frac{1}{\N(f)^{2-4q}}  \sum_{k} \frac{1}{\N(k)^{q}}.$$

This reduces to $$X^{q(2A-\frac{1}{4})-\frac{1}{2}-A}\times  \sum_{k} \frac{1}{\N(k)^{q}} \sum_{\N(l) \ll X^{\min(A,1/2)}} \frac{1}{\N(l)^{4q}},$$ which obviously if we choose $q= 2$ and $A=\frac{3}{7}$ gives the same bound as the above case when $(m,f)=1.$

We will show the $k=0$ term is the dominant term and contributes a main term of size $X^{1/2}.$ The term to look at is

  \begin{multline}\label{eq:know02}
 \frac{1}{\phi(3)6(\sqrt{-3})^2\N(3^3bD)} \sum_{\psi(3)} \sum_{(m,b)=1}  \frac{1}{\N(m)} \sum_{d|m} \frac{\mu^2(d)}{\N(d)}   \sum_{f} \frac{\mu(f) \psi(f) g(1,f)}{\N(f)^2}\times \\ \bigg[\sum_{y(3^3mfb)} \psi(dy)(\frac{\overline{dy}}{b})_3(\frac{dy}{f})_3\mathbf{1}_m(y)    e(\frac{-m^3\overline{Ddy}}{3^3fb})\bigg]   \int_{\C}  \int_{\C}  e(\frac{m^3}{bDfz})  e( \frac{-t\sqrt{X}m}{Dfdz})g(\N(t)) \times \\ V(\frac{4\pi\sqrt{(\sqrt{X}t)^3b }}{Dfdz}) d^{+}t \frac{d^{+}z}{|z|^2} \end{multline}


 We now consider the arithmetic exponential sum, $$\sum_{y(3^3mfb)} \psi(dy)(\frac{\overline{dy}}{b})_3(\frac{dy}{f})_3\mathbf{1}_m(y)    e(\frac{-m^3\overline{Ddy}}{3^3fb}).$$ We would like to use Chinese remainder theorem to simplify the sums. We know already $(m,b)=1, (bf,3)=1.$ If we assume that $f=b,$ then we conclude easily that the sum reduces to $\sum_{y(b^2)}  e(\frac{-m^3\overline{3^3Ddy}}{b^2})=0.$ So now assume $(f,b)=1.$ Now let $(m,f)=l.$ We assume for now that $l>1.$ Then as $f$ is squarefree we know $(\frac{f}{l},l)=1,$ so $(3^3l^2\frac{m}{l},b\frac{f}{l})=1.$ The above then decomposes (with a relabeling $\frac{f}{l} \to f, \frac{m}{l} \to m$) as \begin{multline}\label{eq:gex} \left[\sum_{z(3^3l^2m)} \psi(dz) \mathbf{1}_{lm}(z)(\frac{dz}{l})_3  e(\frac{-l^2m^3\overline{fbDdz}}{3^3})\right] \left[\sum_{y(f)}(\frac{dy}{f})_3 e(\frac{-m^3\overline{3^3bDdy}}{f})\right]\times \\ \left[\sum_{z(b)}(\frac{\overline{dz}}{b})_3 e(\frac{-m^3\overline{3^3Dfdz}}{b})\right].\end{multline} As $(l,3)=1$ it easy to conclude that the first exponential sum is non-zero only if $l=1,$ so we conclude $(m,f)=1.$ Likewise, if $(m,3)=1$ and $\psi \neq \mathbf{1}_{3^3}$ then the sum is also zero as we have a Gauss sum to a prime power modulus with a primitive character $\psi,$ so to have a non-zero contribution we need $m \equiv 0(3),\psi= \mathbf{1}_{3}$ for a non-zero contribution. So with an elementary calculation, we are left with $$\phi(3^3m)(\frac{\overline{bD}}{f})_3 g_2(1,f)(\frac{fD}{b})_3g(1,b).$$ Note the dependence on $d$ goes away by a change of variables.  
 
Using the results on exponential sums above, we have \eqref{eq:know02} equaling \begin{multline}\label{eq:know0}
 \frac{g(1,b)\phi(3^3)}{\phi(3)6^2(\sqrt{-3})^2\N(3^3bD)} \sum_{(m,b)=1}  \frac{\phi(m)}{\N(m)} \sum_{d|m} \frac{\mu^2(d)}{\N(d)}  \sum_{(f),(f,bDdm)=1} \frac{\mu(f) }{\N(f)^2}  \times \\ \int_{\C}  \int_{\C}  e(\frac{m^3}{bDfdz})  e( \frac{-t\sqrt{X}m}{Dfdz})g(\N(t))V(\frac{4\pi\sqrt{(\sqrt{X}t)^3b }}{Dfdz}) d^{+}t \frac{d^{+}z}{|z|^2} \end{multline}
using $g(1,f)g_2(1,f)=\N(f).$ The extra factor of $6$ is from making the sum over $f$ into an ideal sum $(f).$



\subsubsection{Poisson summation in the $m$-sum}\label{mbro}
 
 We first invert the $d$- and $m$-sums \begin{multline}\label{eq:know07}
 \frac{g(1,b)\phi(3^3)}{\phi(3)6^2(\sqrt{-3})^2\N(3^3bD)}\sum_{d,(d,b)=1} \frac{\mu^2(d)\phi(d)}{\N(d)^2}  \sum_{(m,b)=1}  \frac{\phi(m)}{\N(m)}   \sum_{(f),(f,bDdm)=1} \frac{\mu(f) }{\N(f)^2}  \times \\ \int_{\C}  \int_{\C}  e(\frac{d^2m^3}{bDfz})  e( \frac{-t\sqrt{X}m}{Dfz})g(\N(t))V(\frac{4\pi\sqrt{(\sqrt{X}t)^3b }}{Dfdz}) d^{+}t \frac{d^{+}z}{|z|^2}. \end{multline}

 We apply M\"{o}bius inversion to remove the condition $(f,m)=1,$ change $z \to \frac{1}{z},$ and rearrange sums \begin{multline}\label{eq:know007}
 \frac{g(1,b)\phi(3^3)}{\phi(3)6^2(\sqrt{-3})^2\N(3^3bD)}\sum_{d,(d,b)=1} \frac{\mu^2(d)\phi(d)}{\N(d)^2}\sum_{(k),(k,bDd)=1} \frac{\phi(k)\mu(k)^2}{\N(k)^3}   \sum_{(f),(f,bDd)=1} \frac{\mu(f) }{\N(f)^2} \sum_{(m,b)=1}  \frac{\phi(m)}{\N(m)}    \times \\ \int_{\C}  \int_{\C}  e(\frac{d^2k^2m^3z}{bDf})  e( \frac{-t\sqrt{X}dmz}{Df})g(\N(t))V(\frac{4\pi\sqrt{(\sqrt{X}t)^3b }z}{Dkfd}) d^{+}t \frac{d^{+}z}{|z|^2}.\end{multline} Again the rearrangement is fine as the $k,f,d$-sums are finite due to the support of $V$ and the $m$-sum is absolutely convergent by integration by parts.

Remove the condition $(m,b)=1$ by another M\"{o}bius inversion, as well as write $$\phi(m)=\sum_{h|m} \mu(h)\N(\frac{m}{h}).$$ Inverting the $h$- and $m$-sums gives \begin{multline}\label{eq:know0077}
 \frac{g(1,b)\phi(3^3)}{\phi(3)6^2(\sqrt{-3})^2\N(3^3bD)}\sum_{r|b} \mu(r)  \sum_{d,(d,b)=1} \frac{\mu^2(d)\phi(d)}{\N(d)^2}\sum_{(k),(k,bDd)=1} \frac{\phi(k)\mu(k)^2}{\N(k)^3}   \sum_{(f),(f,bDd)=1} \frac{\mu(f) }{\N(f)^2}  \times \\ \sum_{h,(h,b)=1} \frac{\mu(h)}{\N(h)} \sum_{(m,b)=1}     \int_{\C}  \int_{\C}  e(\frac{d^2k^2(rhm)^3z}{bDf})  e( \frac{-t\sqrt{X}drhmz}{Df})g(\N(t))V(\frac{4\pi\sqrt{(\sqrt{X}t)^3b }z}{Dkf}) d^{+}t \frac{d^{+}z}{|z|^2} .\end{multline} Then another integration by parts argument shows the $m$-sum can be limited to size $\N(m) \ll \sqrt{X},$ with an error term of size $O(X^{-N})$ If we add a $m=0$ term we can perform Poisson summation on the $m$-sum to get

 \begin{multline}\label{eq:knnw}
 \frac{g(1,b)\phi(3^3)}{\phi(3)6^2(\sqrt{-3})^3\N(3^3bD)}\sum_{r|b} \mu(r) \sum_{d,(d,b)=1} \frac{\mu^2(d)\phi(d)}{\N(d)^2}\sum_{(k),(k,bDd)=1} \frac{\phi(k)\mu(k)^2}{\N(k)^3}   \sum_{(f),(f,bDd)=1} \frac{\mu(f) }{\N(f)^2} \times \\ \sum_{h,(h,b)=1} \frac{\mu(h)}{\N(h)}  \sum_{ q \in \ri}    \int_{\C}  \int_{\C}  \int_{\C}  e(\frac{d^2k^2(rhw)^3z}{bDf})  e( \frac{-3t\sqrt{X}rhwz}{Df})g(\N(t))\times \\ V(\frac{4\pi\sqrt{(\sqrt{X}t)^3b }z}{Dkf}) d^{+}t \frac{d^{+}z}{|z|^2}e(-qw)d^{+}w. \end{multline}

It easy to check that the $m=0$ term itself is of bounded size independent of $X,$ as all of the sums are absolutely convergent and the integral is $O(1).$


With a change of variables $z \to \frac{Dkfz}{2(\sqrt{X}t)^{3/4}}$ we have 

\begin{multline}\label{eq:p0p0}
 \frac{g(1,b)\phi(3^3)}{\phi(3)6^2(\sqrt{-3})^3\N(3^3bD)}\sum_{r|b} \mu(r) \sum_{d,(d,b)=1} \frac{\mu^2(d)\phi(d)}{\N(d)^2}\sum_{(k),(k,bDd)=1} \frac{\phi(k)\mu(k)^2}{\N(k)^3}   \sum_{(f),(f,bDd)=1} \frac{\mu(f) }{\N(f)^2} \sum_{ q \in \ri}    \times \\ \int_{\C}  \int_{\C}  \int_{\C}  e(\frac{k^3d^3(rw)^3z}{2(\sqrt{X}t)^{3/2}b})  e( \frac{-3kdrwz}{2(\sqrt{X}t)^{1/2}})g(\N(t))V(2\pi\sqrt{b}z) d^{+}t \frac{d^{+}z}{|z|^2}e(-qw)d^{+}w. \end{multline}

By an integration by parts argument in the $w$-variable it is easy to check that the $q=0$ term is the only non-negligible term (other terms contributing $O(X^{-N})$ for any $N>0$).

So we only consider from now on the $q=0$ and making another change of variables $w \to \frac{w}{kdhr}$ leaves 
\begin{multline}\label{eq:p0p012}
 \frac{g(1,b)\phi(3^3)}{\phi(3)6^2(\sqrt{-3})^3\N(3^3bD)}\sum_{r|b} \frac{\mu(r)}{\N(r)} \sum_{d,(d,b)=1} \frac{\mu^2(d)\phi(d)}{\N(d)^3}\sum_{(k),(k,bDd)=1} \frac{\phi(k)\mu(k)^2}{\N(k)^4}  \sum_{(f),(f,bDd)=1} \frac{\mu(f) }{\N(f)^2}  \times \\  \sum_{h,(h,b)=1} \frac{\mu(h)}{\N(h)^2}  \int_{\C}  \int_{\C}  \int_{\C}  e(\frac{w^3z}{2(\sqrt{X}t)^{3/2}b})  e( \frac{-3wz}{2(\sqrt{X}t)^{1/2}})g(\N(t))\N(t)V(2\pi\sqrt{b}z) d^{+}t \frac{d^{+}z}{|z|^2}d^{+}w. \end{multline}

The arithmetic in front of the integrals is now simplified. We have \begin{multline}\label{eq:hug}\sum_{r|b} \frac{\mu(r)}{\N(r)} \sum_{d,(d,b)=1} \frac{\mu^2(d)\phi(d)}{\N(d)^3}\sum_{(k),(k,bDd)=1} \frac{\phi(k)\mu(k)^2}{\N(k)^4}   \sum_{(f),(f,bDd)=1} \frac{\mu(f) }{\N(f)^2} \sum_{h,(h,b)=1} \frac{\mu(h)}{\N(h)^2} =\\ \sum_{r|b} \frac{\mu(r)}{\N(r)} \sum_{(q_1),(q_1,Db)=1}\frac{\phi(q_1)^2\mu(q_1)^5}{\N(q_1)^7} \sum_{(q_2),(q_2,Db)=1}\frac{\phi(q_2)\mu(q_2)^4}{\N(q_2)^5}\sum_{d,(d,b)=1}\times \\ \frac{\mu^2(d)\phi(d)}{\N(d)^3}\sum_{(k),(k,bD)=1}  \frac{\phi(k)\mu(k)^2}{\N(k)^4}   \sum_{(f),(f,bD)=1} \frac{\mu(f) }{\N(f)^2} \sum_{h,(h,b)=1} \frac{\mu(h)}{\N(h)^2} =\\ \sum_{r|b} \frac{\mu(r)}{\N(r)} \sum_{r_1|Db}\frac{\phi(r_1)^2\mu(r_1)^6}{\N(r_1)^7}\sum_{r_2|Db}\frac{\phi(r_2)\mu(r_2)^5}{\N(r_2)^5} \sum_{r_3|Db}\frac{\phi(r_3)\mu(r_3)^3}{\N(r_3)^3}\times \\ \sum_{r_4|Db} \frac{\phi(r_4)\mu(r_4)^3}{\N(r_4)^4} \sum_{r_5|Db}\frac{\mu(r_5)^3}{\N(r_5)^2}  \sum_{r_6|b}\frac{\mu(r_6)^2}{\N(r_6)^2} \sum_{(q_1)}\frac{\phi(q_1)^2\mu(q_1)^5}{\N(q_1)^7} \sum_{(q_2)}\frac{\phi(q_2)\mu(q_2)^4}{\N(q_2)^5} \times \\\sum_{d} \frac{\mu^2(d)\phi(d)}{\N(d)^3}\sum_{(k)} \frac{\phi(k)\mu(k)^2}{\N(k)^4}   \sum_{(f)} \frac{\mu(f) }{\N(f)^2}\sum_{(h)} \frac{\mu(h)}{\N(h)^2}=:A_{b,D}.\end{multline}

First we have $$ \sum_{(q)}\frac{\phi(q)^2\mu(q)^5}{\N(q)^j} =\frac{\zeta_{\K}(j-1)^2}{\zeta_{\K}(2(j-1))^2\zeta_{\K}(j-2)\zeta_{\K}(j)},$$ and $$\sum_{(q)}\frac{\phi(q)\mu(q)^l}{\N(q)^j}=\left\{ \begin{array}{ll}
     \frac{\zeta_{\K}(j-1)}{\zeta_{\K}(2(j-1))\zeta_{\K}(j)}  & \text{if } l\mbox{ is even} \medskip \\
         		 \frac{\zeta_{\K}(j)}{\zeta_{\K}(2j)\zeta_{\K}(j-1)}   & \text{if } l\mbox{ is odd}\end{array} \right..$$

So \begin{multline}\sum_{(q_1)}\frac{\phi(q_1)^2\mu(q_1)^5}{\N(q_1)^7} \sum_{(q_2)}\frac{\phi(q_2)\mu(q_2)^4}{\N(q_2)^5} \sum_{d} \frac{\mu^2(d)\phi(d)}{\N(d)^3}\sum_{(k)} \frac{\phi(k)\mu(k)^2}{\N(k)^4} \times \\   \sum_{(f)} \frac{\mu(f) }{\N(f)^2}\sum_{(h)} \frac{\mu(h)}{\N(h)^2}=\\\frac{\zeta_{\K}(6)^2}{\zeta_{\K}(12)\zeta_{\K}(12)\zeta_{\K}(5)\zeta_{\K}(7)}\times \frac{\zeta_{\K}(4)}{\zeta_{\K}(8)\zeta_{\K}(5)}\times \frac{\zeta_{\K}(2)}{\zeta_{\K}(4)\zeta_{\K}(3)}\times \frac{\zeta_{\K}(3)}{\zeta_{\K}(6)\zeta_{\K}(4)}\times \frac{1}{\zeta_{\K}(2)^2}=\\ \frac{\zeta_{\K}(6)}{\zeta_{\K}(12)^2\zeta_{\K}(8)\zeta_{\K}(7)\zeta_{\K}(5)^2\zeta_{\K}(4)\zeta_{\K}(2)}=:Z_K.\end{multline}

Therefore, \begin{multline}\label{eq:harda} A_{b,D}=Z_K\bigg( \sum_{r|b} \frac{\mu(r)}{\N(r)} \sum_{r_1|Db}\frac{\phi(r_1)^2\mu(r_1)^6}{\N(r_1)^7}\sum_{r_2|Db}\frac{\phi(r_2)\mu(r_2)^5}{\N(r_2)^5} \sum_{r_3|Db}\frac{\phi(r_3)\mu(r_3)^3}{\N(r_3)^3}\times \\ \sum_{r_4|Db} \frac{\phi(r_4)\mu(r_4)^3}{\N(r_4)^4} \sum_{r_5|Db}\frac{\mu(r_5)^3}{\N(r_5)^2}  \sum_{r_6|b}\frac{\mu(r_6)^2}{\N(r_6)^2} \bigg).\end{multline}.

With another change of variables $w \to (\sqrt{X}t)^{1/2}w,$  \eqref{eq:p0p012} equals





 \begin{equation}\label{eq:p0p01w}
 \frac{X^{1/2}g(1,b)A_{b,D}\phi(3^3)}{\phi(3)6^2(\sqrt{-3})^3\N(3^3bD)}  \int_{\C}  \int_{\C} V(2\pi z\sqrt{b}) e(\frac{w^3z}{2b})  e(\frac{ -3wz}{2})\frac{d^{+}z}{|z|^2}d^{+}w.
\end{equation}

\begin{remark}
Note if $b=p^j,j>1$ then $g(1,b)=\sum_{y(b)}  (\frac{y}{b})_3  e(\frac{y}{b})=0,$ so \eqref{eq:p0p01w} is zero and there is not a main term. This proves Proposition \ref{pb3} in this case.
\end{remark}

By the above remark we can reduce to the case $b=p.$

Now the $w$-integral has been fortunately studied in \cite{Ku1}, giving the formula \begin{multline}\int_{\C} e(\frac{1}{2}(\frac
{zw^3}{b}-3wz))d^{+}w=\\ \frac{\N(b)^{1/3}}{\N(z)^{1/3}}\int_{\C} e(\frac{1}{2}(w^3-3wb^{1/3}z^{2/3}))d^{+}w=\\ \frac{\N(b)^{1/3}}{\N(z)^{1/3}} \frac{\pi^2}
{3\sin(\pi/3)}\big|b^{1/3}z^{2/3}\big|(|J_{-1/3}(2\pi zb^{1/2})|^2+|J_{1/3}(2\pi zb^{1/2})|
^2)=\\ \N(b)^{1/2} \frac{\pi^2}
{3\sin(\pi/3)}(|J_{-1/3}(2\pi zb^{1/2})|^2+|J_{1/3}(2\pi zb^{1/2})|^2). \end{multline} 

Note $\frac{\phi(3^3)}{\N(3^3)}=\frac{2}{3}.$

Inputting this into our main term we get \begin{multline}\label{eq:unb}
\frac{2\pi^2}{9} \frac{X^{1/2} g(1,p)A_{p,D} \int_{\C}g(\N(t))\sqrt{\N(t)}d^{+}t }{6^2(\sqrt{-3})^3\N(Dp^{1/2})}   \times \\\int_{\C} V(2\pi z\sqrt{p})\frac{(|J_{-1/3}(2\pi zp^{1/2})|^2+|J_{1/3}(2\pi zp^{1/2})|^2)}{\sin(\pi/3)}\frac{d^{+}z}{|z|^2}.
\end{multline}
Now a change of variables $z \to \frac{z}{2\pi \sqrt{p}},$ unhinges the archimedean integral from the arithmetic $p.$ This $p$ dictates the $p$-th Fourier coefficient on the spectral side of the trace formula.

\begin{multline}\label{eq:unb}
\frac{2\pi^2}{9}\frac{X^{1/2} g(1,p)A_{p,D} \int_{\C}g(\N(t))\sqrt{\N(t)}d^{+}t }{6^2(\sqrt{-3})^3\N(Dp^{1/2}) }\int_{\C} V(z)\frac{(|J_{-1/3}(z)|^2+|J_{1/3}(z)|^2)}{\sin(\pi/3)}\frac{d^{+}z}{|z|^2}.
\end{multline}


We make a change of variables to polar coordinates giving $$\int_{\C}g(\N(t))\sqrt{\N(t)}d^{+}t =\frac{2\pi}{\sqrt{-3}} \int_{\R}g(t)\sqrt{t}dt,$$ and as we assumed $g$ is chosen to have $\int_{\R}g(t)\sqrt{t}dt=1,$ the main term is \begin{equation}\label{eq:unb0w}
\frac{4\pi^3}{9} \frac{X^{1/2} g(1,p)A_{p,D} }{6^2(\sqrt{-3})^4\N(Dp^{1/2}) } \int_{\C}V(z) \frac{(|J_{-1/3}(z)|^2+|J_{1/3}(z)|^2)}{\sin(\pi/3)}\frac{d^{+}z}{|z|^2}.
\end{equation}

Now it is an easy but tedious check that \begin{equation}
B_{1/3,0}(z)=\frac{(|J_{-1/3}( z)|^2+|J_{1/3}( z)|^2)}{\sin(\pi/3)},
\end{equation}

so \eqref{eq:poissonm21} equals a main term plus the error from truncating the $f$-sum above, \begin{equation}\label{eq:afku}
\frac{4\pi^3}{9} \frac{X^{1/2} g(1,p)A_{p,D} }{6^2(\sqrt{-3})^4\N(Dp^{1/2})} h(V,1/3,0)+O(X^{2/7+\epsilon}).
\end{equation}

\subsection{Ramified calculation: $(c,b)>1$}
\subsubsection{The case of $b=p$}

In this case, we showed \eqref{eq:dodo} the ramified calculation simplifies to three sums  \begin{multline*}\frac{\mu(p)g(1,p)}{\N(p)}  \sum_{ (m,p)=1,m\neq 0}\sum_{d|m,(d,p)=1}\mu^2(d) \sum_{f} \frac{\mu(f) g(1,f)(\frac{f^2}{p})_3}{\N(f)} \times \\ \sum_{\substack{fc\equiv1(3)\\fc\equiv 0(D)\\d|c\\(\frac{c}{d},mp)=1}} \left[ (\frac{c}{p})_3 (\frac{c}{p^2})_3\right](\frac{c}{f})_3  e(\frac{\overline{f p^23^3}m^3}{c}) \ww_{m}(\frac{\N(pfc)}{X^{3/2}}),\end{multline*} \begin{multline*}\frac{\phi(p)g(1,p)}{\N(p)}  \sum_{ m\neq 0}\sum_{d|m,(d,p)=1}\mu^2(d)  \sum_{f} \frac{\mu(f) g(1,f) (\frac{f^2}{p})_3}{\N(f)} \times \\ \sum_{\substack{fc\equiv1(3)\\fc\equiv 0(D)\\d|c\\(\frac{c}{d},mp)=1}}   \left[ (\frac{c}{p})_3 (\frac{c}{p^2})_3\right] (\frac{c}{f})_3 e(\frac{p m^3\overline{f3^3}}{c}) \ww_{pm}(\frac{\N(pfc)}{X^{3/2}}),\end{multline*}  \begin{multline*}g(1,p) \sum_{ m ,m\equiv 0(p) }\sum_{d|m,(d,p)=1}\mu^2(d) \sum_{f} \frac{\mu(f) g(1,f) }{\N(f)} \times \\ \sum_{\substack{fc\equiv1(3)\\fc\equiv 0(D)\\d|c\\(\frac{c}{d},m)=1}}  \left[(\frac{p^2}{c})_3(\frac{p}{c})_3 \right] (\frac{c}{f})_3  e(\frac{ m^3\overline{f3^3}}{c}) \ww_{m}(\frac{\N(p^2fc)}{X^{3/2}}).\end{multline*}

These are associated to $(i=1,k=-1,j=0),(i=1,k=-1,j\geq1),(i=1,k=-2,j\geq1),$ respectively.

  \begin{itemize}

\item The case of $k=1,j=0.$

Let us look at the first sum. Removing the coprimality conditions via M\"{o}bius inversion on the $c$-sum, an application of Elementary reciprocity, and a similar Poisson summation modulo $3^3fp^2$ analogous to getting \eqref{eq:gex}, the essential arithmetic to study is the sum $$\sum_{y(3^3fp^2)} \psi(dy)\mathbf{1}_{cm}(dy) e(\frac{-m^3\overline{dy}}{3^3p^2f}).$$ This is the arithmetic exponential sum coming from $k=0$ term in the analogous Poisson summation, which by analogous analysis to the last section is the only term we need to look at. The key point being that the difference between the ramified terms here and the terms we studied above is soley through the parameter $b=p$ which does not affect any kind of integration by parts or convergence issues. Again, we must have $m\equiv 0(3)$ for a non-trivial contribution, but then we are left with $$\sum_{y(p^2)} e(\frac{-m^3\overline{fdy}}{p^2})=0.$$ So the first sum is zero. 

\item The case of $k=1,j\geq 1.$

We now look at the second sum above. We can follow the unramified case up to \begin{multline}\label{eq:newmop}
 \frac{1}{\phi(3)6\sqrt{-3}} \frac{\phi(p)g(1,p)}{\N(p)^2} \sum_{\psi(3)}\frac{\psi(D)}{\N(D)}\sum_{m\neq 0,} \sum_{d|m,(d,p)=1} \frac{\mu^2(d)}{\N(d)}   \sum_{\N(f)\leq X^A} \frac{\mu(f) \psi(f)g(1,f)(\frac{f^2}{p})_3}{\N(f)^2}   \times \\  \sum_{c} \frac{1}{\N(c)}\psi(dc) \mathbf{1}_{pm}(y) (\frac{dc}{f})_3  e(\frac{-pm^3\overline{Ddc}}{f3^3})e(\frac{pm^3}{3^3bDfdc})\times \end{multline} $$ \int_{\C} e( \frac{-t\sqrt{X}pm}{Dfdpc})g(\N(t))
 V(\frac{4\pi
\sqrt{(\sqrt{X}t)^3 p }}{Dfdpc}) d^{+}t. $$

After a now standard Poisson summation modulo $3^3fpm$ with only the $k=0$ term larger than $O(X^{1/4+\epsilon})$ we have 

 \begin{multline}\label{eq:knowop}
\frac{\phi(p)g(1,p)}{\N(p)^3} \frac{1}{\phi(3)6\sqrt{-3}\N(3^3D)} \sum_{\psi(3)} \psi(D))\sum_{f} \frac{\mu(f) \psi(f)g(1,f)(\frac{Dp^2}{f})_3}{\N(f)^3}  \times \\ \sum_{(m,f)=1}  \bigg[\sum_{y(3^3fpm)}  \mathbf{1}_{pm}(dy)(\frac{dy}{f})_3\psi(dy)e(\frac{-pm^3\overline{Ddy}}{3^3f}) \bigg]  \times \\ 
\int_{\C}  \int_{\C}  e(\frac{pm^3}{3^3Dfdz})  e( \frac{-t\sqrt{X}m}{Dfdz})g(\N(t))V(\frac{4\pi\sqrt{(\sqrt{X}t)^3p}}{Dlfpdz}) d^{+}t \frac{d^{+}z}{|z|^2}. \end{multline}

The $y$-sum equals $\phi(3^3mp)g_2(1,f)(\frac{p\overline{Dh}}{f})_3 $ with $m\equiv 0(3)$ and $\psi \equiv 1(3)$ by similar to analysis to \eqref{eq:gex}.

This leaves
\begin{multline}\label{eq:know02}
\frac{\phi(p)^2g(1,p)}{\N(p)^3}  \frac{\phi(3^3)}{\phi(3)6^2\sqrt{-3}\N(3^3D)} \sum_{m\neq 0}  \frac{\phi(m)}{\N(m)} \sum_{d|m,(d,p)=1} \frac{\mu^2(d)}{\N(d)}  \sum_{(f),(f,bDm)=1} \frac{\mu(f) }{\N(f)^2}  \times \\ \int_{\C}  \int_{\C}  e(\frac{pm^3}{Dfdz})  e( \frac{-t\sqrt{X}m}{Dfdz})g(\N(t))V(\frac{4\pi\sqrt{(\sqrt{X}t)^3p }}{Dfdpz}) d^{+}t \frac{d^{+}z}{|z|^2} \end{multline}

By a similar analysis to section \ref{mbro} in applying M\"{o}bius inversion to the $f$-sum as well as Poisson summation in the $m$-sum we get 

\begin{multline}\label{eq:knnw02}
 \frac{\phi(p)^2g(1,p)}{\N(p)^3} \frac{\phi(3^3)}{\phi(3)6^2(\sqrt{-3})^2\N(3^3D)}\sum_{(d,p)=1} \frac{\mu^2(d)\phi(d)}{\N(d)^2}\sum_{(k),(k,pDd)=1} \frac{\phi(k)\mu(k)^2}{\N(k)^3}   \sum_{(f),(f,bDd)=1} \frac{\mu(f) }{\N(f)^2} \times \\ \sum_{h,(h,b)=1} \frac{\mu(h)}{\N(h)}  \sum_{ q \in \ri}    \int_{\C}  \int_{\C}  \int_{\C}  e(\frac{pd^2k^2(hw)^3z}{Df})  e( \frac{-3t\sqrt{X}hwz}{Df})g(\N(t))\times \\ V(\frac{4\pi\sqrt{(\sqrt{X}t)^3p }z}{Dkfp}) d^{+}t \frac{d^{+}z}{|z|^2}e(-qw)d^{+}w, \end{multline}
this equation is completely analogous to \eqref{eq:knnw}.


A change of variables $z \to \frac{1}{z}, z \to \frac{Dkfp^{1/2}z}{2(\sqrt{X}t)^{3/2}},w \to \frac{w}{kdh}$ we have  \begin{multline}\label{eq:p0p01202}
  \frac{\phi(p)^2g(1,p)}{\N(p)^3}  \frac{\phi(3^3)}{\phi(3)6^2(\sqrt{-3})^2\N(3^3D)} \sum_{(d,p)=1} \frac{\mu^2(d)\phi(d)}{\N(d)^3}\sum_{(k),(k,bDd)=1} \frac{\phi(k)\mu(k)^2}{\N(k)^4}  \times \\  \sum_{(f),(f,bDd)=1} \frac{\mu(f) }{\N(f)^2}  \sum_{h,(h,b)=1} \frac{\mu(h)}{\N(h)^2}  \int_{\C}  \int_{\C}  \int_{\C}  e(\frac{p^{3/2}w^3z}{2(\sqrt{X}t)^{3/2}})  e( \frac{-3p^{1/2}wz}{2(\sqrt{X}t)^{1/2}})g(\N(t))\N(t)\times \\ V(2\pi z) d^{+}t \frac{d^{+}z}{|z|^2}d^{+}w. \end{multline}

Let $$B_{p,D}:=\sum_{(d,p)=1} \frac{\mu^2(d)\phi(d)}{\N(d)^3}\sum_{(k),(k,bDd)=1} \frac{\phi(k)\mu(k)^2}{\N(k)^4}   \sum_{(f),(f,bDd)=1} \frac{\mu(f) }{\N(f)^2}  \sum_{h,(h,b)=1} \frac{\mu(h)}{\N(h)^2} .$$ The only difference between this term and $A_{p,D}$ is the lack of coprimality here of the $m$-sum  with $b$, and so $B_{p,D}=\frac{A_{p,D}}{\left(\sum_{r|p} \frac{\mu(r)}{\N(r)}\right)}.$

Then by analogous arguments \eqref{eq:p0p01202} equals \begin{multline} \frac{X^{1/2}\phi(p)^2g(1,p)B_{p,D}}{\N(p)^{7/2}}  \frac{4\pi^3}{9} \frac{1 }{6^2(\sqrt{-3})^3\N(D)} h(V,1/3,0)=\\ \frac{X^{1/2}g(1,p)(1-\frac{1}{\N(p)})A_{p,D}}{\N(p)^{3/2}} \frac{4\pi^3}{9} \frac{1 }{6(\sqrt{-3})^3\N(D)}h(V,1/3,0)\end{multline}

\item The case of $k=2,j\geq 1.$

The third sum is $$g(1,p)\sum_{ m\equiv 0(p) }\sum_{d|m,(d,p)=1}\mu^2(d)  \sum_{\substack{fc\equiv1(3)\\fc\equiv 0(D)\\d|c\\(\frac{c}{d},m)=1}} \sum_{f} \frac{\mu(f) g(1,f) }{\N(f)} \left[(\frac{p^2}{c})_3(\frac{p}{c})_3 \right] (\frac{c}{f})_3  e(\frac{ m^3\overline{f3^3}}{c}) \ww_{m}(\frac{\N(p^2fc)}{X^{3/2}})
+O(X^{1/4+\epsilon}).$$

We can write it as \begin{multline}\label{eq:newmop3}
 \frac{1}{\phi(3)6\sqrt{-3}} \frac{g(1,p)}{\N(p)^2} \sum_{\psi(3)}\frac{\psi(D)}{\N(D)}\sum_{m\neq 0} \sum_{d|m,(d,p)=1} \frac{\mu^2(d)}{\N(d)}   \sum_{\N(f)\leq X^A} \frac{\mu(f) \psi(f)g(1,f)(\frac{f^2}{p})_3}{\N(f)^2}   \times \\  \sum_{c} \frac{1}{\N(c)}\psi(dc) \mathbf{1}_{pm}(y) (\frac{dc}{f})_3  e(\frac{-m^3\overline{Ddc}}{f3^3})e(\frac{m^3}{3^3Dfdc})  \int_{\C} e( \frac{-t\sqrt{X}pm}{Dfdp^2c})g(\N(t))
 V(\frac{4\pi\sqrt{(\sqrt{X}t)^3 p }}{Dfdp^2c}) d^{+}t. \end{multline}

By a similar argument to the previous case we get \begin{equation}\frac{X^{1/2}g(1,p)A_{p,D}}{\N(p)^{5/2}} \frac{4\pi^3}{9} \frac{1 }{6^2(\sqrt{-3})^3\N(D)}h(V,1/3,0)+O(X^{2/7+\epsilon}).\end{equation} 

\end{itemize}

\subsubsection{The case of $b=p^3$}
\begin{itemize}

\item The case of $k=1, j=0.$

The first term to analyze is $$\frac{\mu(p)g(1,p)}{\N(p)} \sum_{m\neq 0,(m,p)=1} \sum_{d|m,(d,p)=1}\mu^2(d) \sum_{f,(f,p)=1} \frac{\mu(f) g(1,f) }{\N(f)}\sum_{\substack{fc\equiv1(3)\\fc\equiv 0(D)\\d|c\\ (\frac{c}{d},mp)=1}} (\frac{c}{pf})_3 e(\frac{ m^3\overline{p^4f3^3}}{c}) \ww_{m}(\frac{\N(pfc)}{X^{3/2}}).$$

Like the previous section it suffices to analyze the $k=0$ Poisson summation term's exponential sum: $$ \sum_{y(p^4f3^3)} (\frac{y}{pf})_3e(\frac{-m^3\overline{y}}{p^4f3^3}).$$ By Chinese remainder theorem, the $p$-sum reduces to $$ \sum_{y(p^4)}(\frac{y}{p})_3 e(\frac{-m^3\overline{3^3fy}}{p^4})=0.$$

\item The second case is $k=1, j\geq 1.$

We study $$\frac{\phi(p)g(1,p)}{\N(p)}  \sum_{ m\neq 0}\sum_{d|m,(d,p)=1}\mu^2(d)  \sum_{\substack{fc\equiv1(3)\\fc\equiv 0(D)\\d|c\\(\frac{c}{d},mp)=1}}   \sum_{f} \frac{\mu(f) g(1,f) (\frac{f}{p^2})_3}{\N(f)}  (\frac{c}{pf})_3  e(\frac{ m^3\overline{3^3fp}}{c}) \ww_{pm}(\frac{\N(pfc)}{X^{3/2}}).
$$

Using elementary reciprocity, M\"{o}bius inversion for the condition $(m,Dc)=1$ and performing Poisson summation modulo $3^3fp,$ the key arithmetic to look at is \begin{equation}\label{eq:aap93}\sum_{y(3^3pf)} (\frac{y}{pf})_3 e(\frac{m^3\overline{Dy}}{pf})=\phi(3^3)(\frac{f}{p})_3(\frac{\overline{D}}{fp})_3g_2(1,f)g_2(1,p).\end{equation} 

Putting this sum in and looking at only the $k=0$ term we have \begin{multline}\label{eq:grrr}
\frac{\phi(p)}{\N(p)} \frac{\phi(3^3)}{\phi(3)6^2\sqrt{-3}\N(3^3D)} \sum_{m\neq 0}  \frac{\phi(m)}{\N(m)} \sum_{d|m,(d,p)=1} \frac{\mu^2(d)}{\N(d)}  \sum_{(f),(f,pDm)=1} \frac{\mu(f) }{\N(f)^2}  \times \\ \int_{\C}  \int_{\C}  e(\frac{m^3}{Dfdpz})  e( \frac{-t\sqrt{X}m}{Dfdz})g(\N(t))V(\frac{4\pi\sqrt{(\sqrt{X}t)^3 }p^{1/2}}{Dfdz}) d^{+}t \frac{d^{+}z}{|z|^2}+O(X^{2/7+\epsilon}). \end{multline}

By analogous arguments (M\"{o}bius inversion, Poisson summation) as above this equals
$$X^{1/2}A_{p,D}\N(p)^{1/2} \frac{4\pi^3}{9} \frac{1 }{6(\sqrt{-3})^3\N(D)}h(V,1/3,0)+O(X^{2/7+\epsilon}).$$

\item The next case is $p^3||c, p^2|m,$ where we concluded to study
$$\N(p)\phi(p) \frac{1}{X^{3/2}} \sum_{m\neq 0} \sum_{d|m,(d,p)=1}\mu^2(d) \sum_{f} \sum_{\substack{fc\equiv1(3)\\fc\equiv 0(D)\\d|c\\ (\frac{c}{d},mp)=1}}   \frac{\mu(f) g(1,f)}{\N(f)}  \mathbf{1_{p}}(fc)(\frac{c}{f})_3  e(\frac{ m^3\overline{f3^3}}{c}) \ww_{p^2m}(\frac{\N(p^3fc)}{X^{3/2}}).$$

Using elementary reciprocity, M\"{o}bius inversion for the condition $(m,Dc)=1$ and performing Poisson summation modulo $3^3fp,$ the key arithmetic to look at is \begin{equation}\label{eq:aap}\sum_{y(3^3pf)} \mathbf{1}_{p}(y)(\frac{y}{f})_3\psi(y)e(\frac{m^3\overline{Dy}}{3^3f}).\end{equation}  Again $m\equiv 0(3)$ and $\psi$ is trivial to have a non-negligible contribution, and only the $k=0$ term contributes with error $O(X^{2/7+\epsilon}).$ By the Chinese remainder theorem, the sum reduces to $$\phi(3^3p)\sum_{y(f)} (\frac{y}{f})_3 e(\frac{m^3\overline{y}}{f})=\phi(3^3p) g_2(1,f).$$ Putting this sum in we have \begin{multline}\label{eq:krr}
\frac{\phi(p)^2}{\N(p)^3}  \frac{\phi(3^3)}{\phi(3)6^2\sqrt{-3}\N(3^3D)} \sum_{m\neq 0}  \frac{\phi(m)}{\N(m)} \sum_{d|m,(d,p)=1} \frac{\mu^2(d)}{\N(d)}  \sum_{(f),(f,pDm)=1} \frac{\mu(f) }{\N(f)^2}  \times \\ \int_{\C}  \int_{\C}  e(\frac{m^3}{Dfdz})  e( \frac{-t\sqrt{X}m}{Dfdpz})g(\N(t))V(\frac{4\pi\sqrt{(\sqrt{X}t)^3 }}{Dfdp^{3/2}z}) d^{+}t \frac{d^{+}z}{|z|^2} \end{multline}

By analogous arguments (M\"{o}bius inversion, Poisson summation) as above this equals
$$\frac{X^{1/2}A_{p,D}(1-\frac{1}{\N(p)})}{\N(p)^{3/2}} \frac{4\pi^3}{9} \frac{1 }{6(\sqrt{-3})^3\N(D)}h(V,1/3,0)+O(X^{2/7+\epsilon}).$$

\item The case of $p^4||c, p^2||m.$
Recall in Section \ref{pca} that in this case we consider the two sums: \begin{multline*}\N(p)  g(1,p)\mu(p)  \sum_{(m,p)=1, m\neq 0}\sum_{d|m,(d,p)=1}\mu^2(d)  \sum_{f} \sum_{\substack{fc\equiv1(3)\\fc\equiv 0(D)\\d|c\\ (\frac{c}{d},mp)=1}}  \frac{\mu(f) g(1,f)  (\frac{f}{p^2})_3  }{\N(f)} (\frac{c}{fp})_3 \times \\ e(\frac{ m^3\overline{pf3^3}}{c}) \ww_{p^2m}(\frac{\N(p^4fc)}{X^{3/2}}),\end{multline*}

and
\begin{multline}\N(p)^2   \sum_{(m,p)=1, m\neq 0} \sum_{d|m,(d,p)=1}\mu^2(d) \sum_{f} \sum_{\substack{fc\equiv1(3)\\fc\equiv 0(D)\\d|c\\ (\frac{c}{d},mp)=1}}   \frac{\mu(f) g(1,f)  (\frac{f}{p})_3  }{\N(f)}  (\frac{c}{f})_3 e(\frac{ m^3\overline{f3^3}}{pc}) \ww_{p^2m}(\frac{\N(p^4fc)}{X^{3/2}}),\end{multline}

 Using a similar elementary reciprocity, M\"{o}bius inversion  and Poisson summation modulo $3^3pf$ as in the previous cases the first critical sum is: $$\sum_{y(3^3pf)} (\frac{y}{fp})_3 e(\frac{-m^3 \overline{Dy}}{3^3pf}).$$ This can be simplified to  
 
$$g_2(1,p)g_2(1,f) (\frac{\overline{Df}}{p})_3 (\frac{\overline{Dp}}{f})_3 \phi(3^3).$$

We are left with, using analogous analysis as the previous case, with the first sum being

\begin{multline}\label{eq:krr9}
\frac{\mu(p)}{\N(p)^3}  \frac{\phi(3^3)}{\phi(3)6^2\sqrt{-3}\N(3^3D)}  \sum_{(m,p)=1, m\neq 0}   \frac{\phi(m)}{\N(m)} \sum_{d|m,(d,p)=1} \frac{\mu^2(d)}{\N(d)}  \sum_{(f),(f,pDm)=1} \frac{\mu(f) }{\N(f)^2}  \times \\ \int_{\C}  \int_{\C}  e(\frac{m^3}{Dfdpz})  e( \frac{-t\sqrt{X}m}{Dfdp^2z})g(\N(t))V(\frac{4\pi\sqrt{(\sqrt{X}t)^3 }}{Dfdp^{5/2}z}) d^{+}t \frac{d^{+}z}{|z|^2} \end{multline}

By analogous arguments as above this equals
$$\frac{X^{1/2}B_{p,D}\mu(p)(1-\frac{1}{\N(p)})}{\N(p)^{7/2}} \frac{4\pi^3}{9} \frac{1 }{6(\sqrt{-3})^3\N(D)}h(V,1/3,0)+O(X^{2/7+\epsilon})=$$
$$\frac{X^{1/2}A_{p,D}\mu(p)}{\N(p)^{7/2}} \frac{4\pi^3}{9} \frac{1 }{6(\sqrt{-3})^3\N(D)}h(V,1/3,0)+O(X^{2/7+\epsilon}).$$

Analogous to getting \eqref{eq:aap}, the $k$-sum in the second sum is 
$$\sum_{y(3^3pf)} (\frac{y}{f})_3 \mathbf{1}_p(y) e(\frac{ m^3\overline{Dpy}}{3^3f})=\N(3^3)(\frac{\overline{Df}}{p})_3\phi(p)g_2(1,f).$$

The second sum is by similar reasoning 

\begin{multline}\label{eq:krr123}
\frac{\phi(p)}{\N(p)^3}  \frac{\phi(3^3)}{\phi(3)6^2\sqrt{-3}\N(3^3D)} \sum_{(m,p)=1, m\neq 0}  \frac{\phi(m)}{\N(m)} \sum_{d|m,(d,p)=1} \frac{\mu^2(d)}{\N(d)}  \sum_{(f),(f,pDm)=1} \frac{\mu(f) }{\N(f)^2}  \times \\ \int_{\C}  \int_{\C}  e(\frac{m^3}{Dfdpz})  e( \frac{-t\sqrt{X}m}{Dfdp^2z})g(\N(t))V(\frac{4\pi\sqrt{(\sqrt{X}t)^3 }}{Dfdp^{5/2}z}) d^{+}t \frac{d^{+}z}{|z|^2} \end{multline}
which is 

So our final term is $$\frac{X^{1/2}B_{p,D}\phi(p)(1-\frac{1}{\N(p)})}{\N(p)^{7/2}} \frac{4\pi^3}{9} \frac{1 }{6(\sqrt{-3})^3\N(D)}h(V,1/3,0)+O(X^{2/7+\epsilon})=$$
$$\frac{X^{1/2}A_{p,D}(1-\frac{1}{\N(p)})}{\N(p)^{5/2}} \frac{4\pi^3}{9} \frac{1 }{6(\sqrt{-3})^3\N(D)}h(V,1/3,0)+O(X^{2/7+\epsilon}).$$

\item Now we look at the case of $j\geq3.$

Recall we study the two sums: \begin{multline*}\N(p) g(1,p)\mu(p)   \sum_{ m\neq 0} \sum_{d|m,(d,p)=1}\mu^2(d) \sum_{f}\frac{\mu(f) g(1,f)  (\frac{f}{p^2})_3  }{\N(f)} \sum_{\substack{fc\equiv1(3)\\fc\equiv 0(D)\\d|c\\ (\frac{c}{d},mp)=1}}   (\frac{c}{pf})_3  \times \\ e(\frac{ p^2m^3\overline{f3^3}}{c}) \ww_{p^3m}(\frac{\N(p^4fc)}{X^{3/2}}),\end{multline*} and 

\begin{multline*}\N(p)^2    \sum_{ m\neq 0} \sum_{d|m,(d,p)=1}\mu^2(d) \sum_{f} \frac{\mu(f) g(1,f)  (\frac{f}{p})_3  }{\N(f)} \sum_{\substack{fc\equiv1(3)\\fc\equiv 0(D)\\d|c\\ (\frac{c}{d},mp)=1}}     (\frac{c}{f})_3 e(\frac{ p^2m^3\overline{f3^3}}{c}) \ww_{p^3m}(\frac{\N(p^4fc)}{X^{3/2}}).\end{multline*}

Like the previous case, the critical sums after Poisson summation are $$\sum_{y(3^3pf)} (\frac{y}{pf})_3 e(-\frac{p^2m^3 \overline{y}}{3^3f})$$ and $$\sum_{y(3^3pf)}  \mathbf{1}_p(y)(\frac{y}{f})_3 e(-\frac{p^2m^3 \overline{y}}{3^3f}).$$ The first sum using the Chinese remainder theorem in the $p$-variable is a complete character sum and so is zero.

The second sum reduces to $$\sum_{y(3^3pf)} \mathbf{1}_p(y) (\frac{y}{f})_3 e(\frac{p^2m^3 \overline{Dy}}{3^3f})=\phi(3^3p)(\frac{\overline{D}}{f})_3(\frac{f^2}{p})_3g_2(1,f)$$

 So we only consider the second sum, which after the usual manipulations is
 
\begin{multline}\label{eq:krr123}
\frac{\phi(p)}{\N(p)^3}  \frac{\phi(3^3)}{\phi(3)6^2\sqrt{-3}\N(3^3D)} \sum_{m\neq 0}  \frac{\phi(m)}{\N(m)} \sum_{d|m,(d,p)=1} \frac{\mu^2(d)}{\N(d)}  \sum_{(f),(f,pDm)=1} \frac{\mu(f) }{\N(f)^2}  \times \\ \int_{\C}  \int_{\C}  e(\frac{p^2m^3}{Dfdz})  e( \frac{-t\sqrt{X}m}{Dfdpz})g(\N(t))V(\frac{4\pi\sqrt{(\sqrt{X}t)^3 }}{Dfdp^{5/2}z}) d^{+}t \frac{d^{+}z}{|z|^2} .\end{multline}

  By an analogous calculation this equals $$\frac{X^{1/2}B_{p,D}(1-\frac{1}{\N(p)})}{\N(p)^{7/2}} \frac{4\pi^3}{9} \frac{1 }{6(\sqrt{-3})^3\N(D)}h(V,1/3,0)+O(X^{2/7+\epsilon})=$$
  $$\frac{X^{1/2}A_{p,D}}{\N(p)^{7/2}} \frac{4\pi^3}{9} \frac{1 }{6(\sqrt{-3})^3\N(D)}h(V,1/3,0)+O(X^{2/7+\epsilon}).$$

\item Finally, we look at the case of $p^{k}||c$ with $j=2.$

\begin{multline}\N(p)^2 \sum_{k=5}^\infty    \sum_{(m,p)=1, m\neq 0} \sum_{d|m,(d,p)=1}\mu^2(d) \sum_{f} \sum_{\substack{fc\equiv1(3)\\fc\equiv 0(D)\\d|c\\ (\frac{c}{d},mp)=1}}   \frac{\mu(f) g(1,f)  (\frac{f}{p^k})_3  }{\N(f)}  (\frac{c}{f})_3 \times \\ e(\frac{m^3\overline{3^3fc}}{p^{k-3}})e(\frac{ m^3\overline{f3^3p^{k-3}}}{c})   \ww_{p^2m}(\frac{\N(p^k fc)}{X^{3/2}}),\end{multline}

We have by elementary reciprocity  $$e(\frac{m^3\overline{3^3fc}}{p^{k-3}})e(\frac{ m^3\overline{f3^3p^{k-3}}}{c})=e(\frac{m^3\overline{3^3fc}}{p^{k-3}})e(\frac{ -m^3\overline{3^3fc}}{p^{k-3}})e(\frac{m^3\overline{3^3f}}{p^{k-3}c})=e(\frac{m^3\overline{3^3f}}{p^{k-3}c}).$$
 
 Standard manipulations give 
 \begin{multline}\label{eq:krrk}
\phi(p) \N(p)^2 \sum_{k=5}^\infty \frac{1}{\N(p)^k}  \frac{\phi(3^3)}{\phi(3)6^2\sqrt{-3}\N(3^3D)} \sum_{m\neq 0,(m,p)=1}  \frac{\phi(m)}{\N(m)} \sum_{d|m,(d,p)=1} \frac{\mu^2(d)}{\N(d)}  \sum_{(f),(f,pDm)=1} \frac{\mu(f) }{\N(f)^2}  \times \\ \int_{\C}  \int_{\C}  e(\frac{m^3}{Dfdp^{k-3}z})  e( \frac{-t\sqrt{X}m}{Dfdp^{k-2}z})g(\N(t))V(\frac{4\pi\sqrt{(\sqrt{X}t)^3 }}{Dfdp^{k-3/2}z}) d^{+}t \frac{d^{+}z}{|z|^2} .\end{multline}
 
 By another applications of elementary reciprocity and looking at the essential exponential sum from the Poisson summation to consider is $$\sum_{y(3^3pf)} (\frac{y}{f})_3 \mathbf{1}_p(y) e(\frac{ -m^3\overline{p^{k-3}y}}{3^3f}).$$ Which simplifies to $$(\frac{\overline{f}}{Dp^k})_3 g_2(1,f)\phi(3^3p).$$
 
Analogously, we get \begin{multline}\label{eq:wei}\sum_{k=5}^\infty \frac{1}{\N(p)^k} X^{1/2}A_{p,D}\phi(p) \N(p)^{1/2} \frac{4\pi^3}{9} \frac{X^{1/2} }{6(\sqrt{-3})^3\N(D)}h(V,1/3,0)+O(X^{2/7+\epsilon})=\\(\frac{1}{(1-\frac{1}{\N(p)})}-(1+\frac{1}{\N(p)}+\frac{1}{\N(p)^2}+\frac{1}{\N(p)^3}+\frac{1}{\N(p)^4}))\phi(p) \N(p)^{1/2} \frac{4\pi^3}{9} \frac{X^{1/2} }{6(\sqrt{-3})^3\N(D)}h(V,1/3,0)\\+O(X^{2/7+\epsilon}).\end{multline}
 
 An elementary calculation shows $$(\frac{1}{(1-\frac{1}{\N(p)})}-(1+\frac{1}{\N(p)}+\frac{1}{\N(p)^2}+\frac{1}{\N(p)^3}+\frac{1}{\N(p)^4}))\phi(p) \N(p)^{1/2} =\frac{1}{\N(p)^{7/2}}.$$

So the final calculation for this case is $$\frac{X^{1/2}}{\N(p)^{7/2}}\frac{4\pi^3}{9} \frac{1 }{6(\sqrt{-3})^3\N(D)}h(V,1/3,0)+O(X^{2/7+\epsilon}).$$

\end{itemize}

\section{Summary of asymptotic for $b=\{p,p^3\}$ and proof of Theorem \ref{theorem}}\label{recall}

For $b=p,$ we have 
\begin{multline}\label{eq:aaah}
\frac{1}{X} \sum_{(n) }g(\N(n)/X)
\left(\sum_{\Pi \neq \mathbf{1}}
h(V,\nu_{\Pi},p_{\Pi})a_{n^3}(\Pi)\overline{a_{p}(\Pi)}+
\{CSC_{n^3,b}\} \right)=\\ \frac{X^{1/2}g(1,p)A_{p,D}}{\N(p)^{1/2}}\frac{4\pi^3}{9} \frac{1 }{6^2(\sqrt{-3})^3\N(D)}h(V,1/3,0)+\\ \frac{X^{1/2}g(1,p)(1-\frac{1}{\N(p)})A_{p,D}}{\N(p)^{3/2}} \frac{4\pi^3}{9} \frac{1 }{6(\sqrt{-3})^3\N(D)}h(V,1/3,0)+\\ \frac{X^{1/2}g(1,p)A_{p,D}}{\N(p)^{5/2}} \frac{4\pi^3}{9} \frac{1 }{6^2(\sqrt{-3})^3\N(D)}h(V,1/3,0)+\\ O(X^{\frac{2}{7}+\epsilon})=\\ X^{1/2}\frac{4\pi^3}{9} \frac{1 }{6^2(\sqrt{-3})^3\N(D)}h(V,1/3,0)\bigg[\frac{g(1,p)A_{p,D}}{\N(p)^{1/2}}\bigg(1+\frac{(1-\frac{1}{\N(p)})}{\N(p)}+ \frac{1}{\N(p)^{2}}\bigg)\bigg]+ O(X^{\frac{2}{7}+\epsilon})=\\ X^{1/2}A_{p,D}\frac{g(1,p)(1+\frac{1}{\N(p)})}{\N(p)^{1/2}}\frac{4\pi^3}{9} \frac{1 }{6(\sqrt{-3})^3\N(D)}h(V,1/3,0)+O(X^{\frac{2}{7}+\epsilon}).
\end{multline}

For $b=p^3,$ we have 
\begin{multline}\label{eq:aaah3}
\frac{1}{X} \sum_{(n) }g(\N(n)/X)
\left(\sum_{\Pi \neq \mathbf{1}}
h(V,\nu_{\Pi},p_{\Pi})a_{n^3}(\Pi)\overline{a_{p^3}(\Pi)}+
\{CSC_{n^3,p^3}\} \right)=\\ \frac{X^{1/2}A_{p,D}(1-\frac{1}{\N(p)})}{\N(p)^{3/2}} \frac{4\pi^3}{9} \frac{1 }
{6(\sqrt{-3})^3\N(D)}h(V,1/3,0)+\\ X^{1/2}A_{p,D}\N(p)^{1/2} \frac{4\pi^3}{9} 
\frac{1 }{6(\sqrt{-3})^3\N(D)}h(V,1/3,0)+\\ \frac{X^{1/2}A_{p,D}\mu(p)}{\N(p)^{7/2}} \frac{4\pi^3}{9} 
\frac{1 }{6(\sqrt{-3})^3\N(D)}h(V,1/3,0)+\frac{X^{1/2}A_{p,D}(1-\frac{1}{\N(p)})}{\N(p)^{5/2}} \frac{4\pi^3}{9} \frac{1 }{6(\sqrt{-3})^3\N(D)}h(V,1/3,0)+\\ \frac{X^{1/2}A_{p,D}}{\N(p)^{7/2}} \frac{4\pi^3}{9} \frac{1 }{6(\sqrt{-3})^3\N(D)}h(V,1/3,0)+\frac{X^{1/2}}{\N(p)^{7/2}}\frac{4\pi^3}{9} \frac{1 }{6(\sqrt{-3})^3\N(D)}h(V,1/3,0)+O(X^{2/7+\epsilon})=\\ X^{1/2}A_{p,D}\N(p)^{1/2}(1+\frac{1}{\N(p)^2}) \frac{4\pi^3}{9} \frac{1 }{6(\sqrt{-3})^3\N(D)}h(V,1/3,0)+O(X^{2/7+\epsilon})
\end{multline}

\begin{proof}\{{\it Theorem \ref{theorem}\}}

Recall for a metaplectic representation $\Pi,$ $$\lambda_p a_{n^3t}(\Pi)=a_{n^3tp^3}(\Pi)+\frac{g_2(1,p)a_{n^3tp}(\Pi)(-1,p^2)}{\N(p)}.$$ For $n=t=1$ we have $$\lambda_p a_{1}(\Pi)=a_{p^3}(\Pi)+\frac{g_2(1,p)a_{p}(\Pi)(-1,p^2)}{\N(p)}.$$ We assume now that metaplectic forms are normalized so that $a_{1}(\Pi)=1.$ Then $$\lambda_p =a_{p^3}(\Pi)+\frac{g_2(1,p)a_{p}(\Pi)(-1,p^2)}{\N(p)}.$$ So from the summary in \eqref{eq:aaah},\eqref{eq:aaah3} we have 

\begin{multline*}\frac{1}{X} \sum_{(n) }g(\N(n)/X)
\left(\sum_{\Pi \neq \mathbf{1}}
h(V,\nu_{\Pi},p_{\Pi})a_{n^3}(\Pi)\overline{a_{p}(\Pi)}+
\{CSC_{n^3,b}\} \right)=\\X^{1/2}A_{p,D}\frac{g(1,p)}{\N(p)^{1/2}}(1+\frac{1}{\N(p)})\frac{4\pi^3}{9} \frac{1 }{6(\sqrt{-3})^3\N(D)}h(V,1/3,0)+O(X^{\frac{2}{7}+\epsilon})\end{multline*} and 

\begin{multline*}\frac{1}{X} \sum_{(n) }g(\N(n)/X)
\left(\sum_{\Pi \neq \mathbf{1}}
h(V,\nu_{\Pi},p_{\Pi})a_{n^3}(\Pi)\overline{a_{p^3}(\Pi)}+
\{CSC_{n^3,p^3}\} \right)=\\X^{1/2}A_{p,D}\N(p)^{1/2}(1+\frac{1}{\N(p)^2}) \frac{4\pi^3}{9} \frac{1 }{6(\sqrt{-3})^3\N(D)}h(V,1/3,0)+O(X^{\frac{2}{7}+\epsilon}).\end{multline*}

So \begin{multline} \frac{1}{X}\left(\sum_{\Pi \neq \mathbf{1}}
h(V,\nu_{\Pi},p_{\Pi})a_{n^3}(\Pi)\overline{\lambda_{p}(\Pi)}+
\{CSC_{n^3,p}\} \right)=\\X^{1/2}A_{p,D}\frac{4\pi^3}{9} \frac{1 }{6(\sqrt{-3})^3\N(D)}h(V,1/3,0)\bigg[\N(p)^{1/2}(1+\frac{1}{\N(p)^2}) + \frac{g_2(1,p)(-1,p^2)}{\N(p)}\frac{g(1,p)}{\N(p)^{1/2}}(1+\frac{1}{\N(p)})\bigg]+O(X^{\frac{2}{7}+\epsilon})\end{multline}

with \begin{multline*}\label{eq:crazya}A_{p,D}=\bigg( \sum_{r|b} \frac{\mu(r)}{\N(r)} \sum_{r_1|Db}\frac{\phi(r_1)^2\mu(r_1)^6}{\N(r_1)^7}\sum_{r_2|Db}\frac{\phi(r_2)\mu(r_2)^5}{\N(r_2)^5} \sum_{r_3|Db}\frac{\phi(r_3)\mu(r_3)^3}{\N(r_3)^3} \sum_{r_4|Db} \frac{\phi(r_4)\mu(r_4)^3}{\N(r_4)^4} \times \\ \sum_{r_5|Db}\frac{\mu(r_5)^3}{\N(r_5)^2}  \sum_{r_6|b}\frac{\mu(r_6)^2}{\N(r_6)^2} \bigg)\frac{\zeta_{\K}(6)}{\zeta_{\K}(12)^2\zeta_{\K}(8)\zeta_{\K}(7)\zeta_{\K}(5)^2\zeta_{\K}(4)\zeta_{\K}(2)}.\end{multline*}

If we write \begin{equation*}\label{eq:crazya1} K_{p,D}:=A_{p,D}\frac{4\pi^3}{9} \frac{1 }{6(\sqrt{-3})^3\N(D)}\bigg[\N(p)^{1/2}(1+\frac{1}{\N(p)^2}) + \frac{g_2(1,p)(-1,p^2)}{\N(p)}\frac{g(1,p)}{\N(p)^{1/2}}(1+\frac{1}{\N(p)})\bigg]\end{equation*} we conclude Theorem \ref{theorem}.

\end{proof}


\section{Isolating a single form on the spectral side of the trace formula}\label{hec}


We show now how to reduce the above equality to a finite dimensional one. First we need to state some Theorems.

\begin{theorem}\{\text{Bruggeman-Motohashi Inversion formula}\}
For any function $V$ that is even, smooth and compactly supported on $\C^{*}$, let $$ h(V,\nu,p) :=    \frac{\pi}{2}      \int_{\C^{*}} V(z)B_{{\nu,p}}(z)\frac{d^{+}z}{|z|^2}, $$  

then we have $$\mathbf{B}[ h(V,\nu,p)](z)=V(z).$$ Here for  $f(\nu,p)$ a nice function defined in \cite{L}(Chapter 11.1), \begin{equation}
\mathbf{B}[f ](z):=\sum_{p \in \Z} \int_{\Re(\nu)=0} B_{{\nu,p}}(z)f(\nu,p)(p^2-\nu^2)d\nu.
\end{equation}
\end{theorem}

Now we use Hypothesis \ref{stu} in the following propositions. This is very analogous to the matching process in the appendix of  \cite{V}:
\begin{theorem}\label{venkm1}
Let $t_j$ be a discrete subset of $\mathbb{R}$ with $\{j: t_j \leq T\} \ll T^{r}$ for some r and $p \in Z
$ with $\{p:p\leq P\} \ll P^q$ for some $q.$ Let, for each $j,p,$ there be given a function $c_{X}(t_j,p)$ 
depending on $X$, so that $c_{X}(t_j,p) \ll (t_{j}p)^{r'}$ for some $r'$- the implicit constant 
independent of $X.$ Suppose that one has an equality 
\begin{equation}
\lim_{X \to \infty} \sum_{j,p} c_{X}(t_j,p) h(V,t_j)=0
\end{equation}
for all $h(V,t_j,p)$ that correspond via Bruggeman-Motohashi inversion to $V$. Then $\lim_{X \to \infty} c_{X}(t_j)$ exists for each $\{t_j,p\}$ and equals 0.  This equality holds for all functions $h$ for which both sides converge.
\end{theorem}  

\begin{prop}\label{venkm2}
Given $\{j_0,p_0\} \in \mathbb{N}^2, \epsilon > 0$ and an integer $N >0,$ there is a $V$ of compact 
support so that $h(V,t_0,p_0)=1,$ and for all $j' \neq j_0,p\neq p_0, h(V,t_{j'},p) \ll \epsilon[(1+|t_{j'}|)(1+|t_{p}|)]^{-N}.$  
\end{prop}

\begin{proof}\{\text{Theorem \ref{venkm1} via Proposition \ref{venkm2}}\}
Exactly same proof as Proposition 17 implying Theorem 7 from the appendix of \cite{V}.
\end{proof}

\begin{proof}\{\text{Proposition \ref{venkm2}}\}
See Appendix 2.
\end{proof}

Specifically, to use these above statements we let \begin{equation}c_{X}(t_j,p)=\frac{1}{X}  \sum_{\substack{\Pi \\ t_{\Pi}=\{t_j,p\}}}  \left[\sum_n g( \N(n)/X)a_{n^3}(\Pi)\lambda_p(\Pi)-\delta_{\Pi}(1/3,0)K_{p,D}\right]+O(X^{2/7+\epsilon})\end{equation} where $$\delta_{\Pi}(1/3,0)=\left\{ \begin{array}{ll}
     1   & \text{if }  \{t_j,p\}=\{1/3,0\} \medskip \\
         		0   & \text{if } \text{ else}
\end{array} \right.. $$


Then by using Hypothesis \ref{stu}, Theorem \ref{venkm1}, and Proposition \ref{venkm2} for a fixed $t_{\Pi}=\{t_{j_0},p_0\},$ we have the identity 
\begin{equation}\label{eq:finyo}\frac{1}{X}  \sum_{\substack{\Pi \\ t_{\Pi}=\{t_{j_0},p_0\}}}  \sum_n g( \N(n)/X)a_{n^3}(\Pi)\lambda_p(\Pi)=\delta_{\Pi}(1/3,0)X^{1/2}K_{p,D}h(V,1/3,0)+O(X^{2/7+\epsilon}).\end{equation}

From here, we would like to follow \cite{H2} in going from a finite dimensional case to a single metaplectic form, but the Hecke theory is not as simple as in the automorphic spectrum. However we can get around this by an application of strong multiplicity one for metaplectic forms. We follow an argument of White in \cite{W}.

\begin{prop}\label{hiso}
Let $g \in C^\infty_0(\R^{+}).$ Assume $t_{\Pi}\neq \{1/3,0\}$ and there exists a $\delta \in \R$ such that for almost all prime ideals $b,$ we have the bound $$\sum_{\substack{\Pi \\ t_{\Pi}=\{t_{j_0},p_0\}}}  \sum_n g( \N(n)/X)a_{n^3}(\Pi)\lambda_p(\Pi)=O(X^{\delta}),$$ where $n$ runs through the integral ideals of $\Q[\omega]$ and $\lambda_p(\Pi)$ is the $p$-th Hecke eigenvalue of $\Pi.$ Then for all $\Pi$ with $t_{\Pi}=\{t_{j_0},p_0\},$ we have  $$\sum_n g( \N(n)/X)a_{n^3}(\Pi)=O(X^{\delta}).$$
\end{prop}

\begin{proof}
Each metaplectic form has a natural associated metaplectic representation. So let us assume the number of inequivalent representations associated with the metaplectic forms with $t_{\Pi}=\{t_{j_0},p_0\}$ is $r,$ and label them $\Pi_1,...,\Pi_r.$ We recall Flicker in \cite{F} proved strong multiplicity one for the spectrum on metaplectic groups. This result can be considered a corollary of the Shimura correspondence also proved in \cite{F}. 
By strong multiplicity we can choose $r$ prime ideals $b_1,..,b_r$ such that the matrix $$A:=\begin{pmatrix}
\lambda_{\Pi_1}(b_1) & \ldots & \lambda_{\Pi_r}(b_1)\\ \lambda_{\Pi_1}(b_2) & ...& \lambda_{\Pi_r}(b_2) \\ \vdots &\vdots &\vdots \\\lambda_{\Pi_1}(b_r) & \ldots & \lambda_{\Pi_r}(b_r) \end{pmatrix}$$ is non-singular. 

Let $$L_X(b_i):= \sum_{\Pi_r} \sum_{(n)} g( \N(n)/X)a_{n^3}(\Pi_r)\lambda_{b_i}(\Pi_r),$$ and $$F_X(\Pi_j):=   \sum_{(n)} g( \N(n)/X)a_{n^3}(\Pi_j).$$  Then by \eqref{eq:finyo} we have 
$$\begin{pmatrix}
\lambda_{\Pi_1}(b_1) & \ldots & \lambda_{\Pi_r}(b_1)\\ \lambda_{\Pi_1}(b_2) & ...& \lambda_{\Pi_r}(b_2) \\ \vdots &\vdots &\vdots \\ \lambda_{\Pi_1}(b_r) & \ldots & \lambda_{\Pi_r}(b_r) \end{pmatrix}\begin{pmatrix} F_X(\Pi_1)\\ \vdots \\ F_X(\Pi_r) \end{pmatrix}= \begin{pmatrix} L_X(b_1)\\ \vdots \\ L_X(b_r) \end{pmatrix}= O(X^{\delta})\begin{pmatrix} 1\\ \vdots \\ 1 \end{pmatrix}.$$

Mutliplying the above equation by $A^{-1}=(a_{ij})$ gives $$F_X(\Pi_k)= a_{k1}X^{\delta}+a_{k2}X^{\delta}+...+a_{kr}X^{\delta}=O(X^{\delta})$$ for $1 \leq k \leq r.$

\end{proof}

Assuming $t_{\Pi}\neq \{1/3,0\}$ and using Proposition \ref{hiso} we now have $$\frac{1}{X} \sum_{(n)} g(\frac{\N(n)}{X})a_{n^3}(\Pi)=O(X^{2/7+\epsilon}).$$

If $t_{\Pi}= \{1/3,0\},$ then using results of (\cite{F}, 5.3) the eigenspace is generated by residues of Eisenstein series and is one-dimensional. So we also get for $t_{\Pi}= \{1/3,0\},$ $$
\lambda_{p}(\Pi)\frac{1}{X} \sum_{(n)} g(\frac{\N(n)}{X})a_{n^3}(\Pi)=X^{1/2}K_{p,D}+ O(X^{2/7+\epsilon}).$$

So for any $\Pi$, by Mellin inversion we have \begin{multline}\label{eq:acf}
\lambda_{p}(\Pi)\frac{1}{X} \sum_{(n)} g(\frac{\N(n)}{X})a_{n^3}(\Pi)=\lambda_{p}(\Pi) \frac{1}{2\pi i}\int_{(\sigma)}\hat{g}(s)X^{s-1} \left(\sum_{n \in \ri} \frac{a_{n^3}(\Pi)}{\N(n)^s}\right)ds=\\ \delta_{\Pi}(1/3,0)X^{1/2}K_{p,D} +O(X^{2/7+ \epsilon}).
\end{multline}

As $g$ is an arbitrary test function we conclude that the Dirichlet series $$\sum_{n \in \ri} \frac{a_{n^3}(\Pi)}{\N(n)^s}$$ has analytic continuation for $\Re(s)> \frac{9}{7}+\epsilon$ with at most a pole at $s=\frac{3}{2}.$ This concludes Corollary \ref{iso}. Again if we knew information for the asymptotic with $\lambda_{p^n}$ for all $n$ then we could apply a linear independence of characters argument and remove the assumption of Hypothesis \ref{stu}.

\section{Appendix 1}
Here we discuss how one can understand the same problem of determing the poles of our cubic Dirichlet series for a $GL_2$ automorphic form over $\Q,$ with trivial central character. This would be steps toward a beyond endoscopy approach to the symmetric cube L-function. We follow exactly the same steps to get to a sum similar to \eqref{eq:poisson2} but with only the $m=0$ term,  \begin{equation} \label{eq:m0} F(X):=\frac{1}{X^{3/2}}\sum_{c} \frac{1}{c}\sum_{x  (c)^{*}}
e(\frac{\overline{x}l}{c})  \sum_{k(c)} e(\frac{x k^3}{c})
W_{0}(\frac{c}{X^{3/2}}).\end{equation} 
  
We will show $F(X)=O(X^{-1/2}).$

\subsection{Properties of cubic exponential sums and cubic gauss sums.}
Let $f(x)$ be a polynomial with integral coefficients, then adopting the notation of \cite{P1}, we define $$S(f(x),c)=\sum_{j=0}^{c-1} e(\frac{f(j)}{c}).$$

We first need some lemmas found in \cite{P1}.

\begin{lemma}\label{eq:mult}
If $c_1,c_2$ are coprime then $$S(Ax^3,c_1c_2)=S(Ac_2^2x^3,c_1) \cdot S(Ac_1^2x^3,c_2).$$
\end{lemma}

\begin{lemma}\label{cub}

Suppose $(A,p)=1,$  one has 
\begin{eqnarray*}
S(Ax^3,p^k) & = & p^{\frac{2k}{3}} \hspace{8pt} \textit{if } k\equiv 0\mod 3 \\  & = &p^{\frac{2k-2}{3}}S(Ax^3,p)  \hspace{8pt}  \textit{if } k\equiv 1 \mod 3 \\   & = &p^{\frac{2k-1}{3}}  \hspace{8pt} \textit{if } k\equiv 2 \mod 3, p \neq 3 \\ & = & p^{\frac{2k-1}{3}} (1+2\cos(\frac{2\pi A}{9}))  \hspace{8pt} \textit{if } k\equiv 2 \mod 3, p = 3
\end{eqnarray*}
Also if $p\equiv 2\mod 3,$ or if $p=3$ then $$S(Ax^3,p)=0.$$

\end{lemma}

\begin{lemma}\label{p13}
Let $p \equiv 1(3)$ and let $p=\pi \overline{\pi}$ be a decomposition of $p$ into prime factors in $\Z[e^{\frac{2\pi i}{3}}],$ with $\pi \equiv 1(3).$ Then $$S(Ax^3,p)=g(A,\pi) + \overline{g(A,\pi)}$$ where 
$$g(A,\pi)= \sum_{y(\pi)} \left(\frac{y}{\pi}\right)_{3} e\left(Tr(\frac{Ay}{\pi})\right).$$
\end{lemma}

Using these lemmas one can prove
\begin{lemma}\label{eq:pgen}
If $(3A,c)=1,$ then $$S(Ax^3,c)=\sum_{\substack{\delta, d \\ N(\delta)\cdot d^3=c}} g(A,\delta) \cdot d^2,$$ where $\delta \in \Z[\omega], \delta \equiv 1(3)$ and $d \in \N.$

\end{lemma}

Our focus will naturally be on $$\sum_{x  (c)^{*}}
e(\frac{\overline{x}}{c})  \sum_{k(c)} e(\frac{x k^3}{c}).$$ To make our use of Lemma \ref{cub} more explict we write this as \begin{equation}\label{eq:finite} \sum_{A  (c)^{*}}
e(\frac{\overline{A}}{c})  \sum_{k(c)} e(\frac{A k^3}{c}).\end{equation}   We make a change of variables  $k \to \overline{A}k,$ and $ A \to \overline{A},$ to get \eqref{eq:finite} equals \begin{equation}\label{eq:finite1} \sum_{A  (c)^{*}}
e(\frac{A}{c})  \sum_{k(c)} e(\frac{A^2 k^3}{c}).\end{equation}

Let $c_3:= \frac{c}{3^j},$ where $3^j || c,$ so clearly $(c_3,3^j)=1.$ Then using multiplicativity of Lemma \ref{eq:mult}   \eqref{eq:finite1} equals \begin{equation}\label{eq:prodp} 
\left[\sum_{A  (3^{j})^{*}}e(\frac{A  }{3^j})   \sum_{k(3^{j})} e(\frac{(A c_3)^2 k^3}{3^j})\right] \left[\sum_{A  (c_3)^{*}}
e(\frac{A}{c_3})  \sum_{k(c_3)} e(\frac{3^{2j} A^2 k^3}{c_3})\right]
\end{equation}

Then Lemma \ref{eq:pgen} gives \eqref{eq:prodp} equals \begin{equation}\label{eq:prodp1} 
\left[\sum_{A  (3^{j})^{*}}e(\frac{A  }{3^j})   \sum_{k(3^{j})} e(\frac{(A c_3)^2 k^3}{3^j})\right] \left[\sum_{A  (c_3)^{*}}
e(\frac{A}{c_3})  \sum_{\substack{\delta, d \\ N(\delta)\cdot d^3=c_3}} g(A^2 3^{2j},\delta) \cdot d^2\right].
\end{equation}

We focus on the sums associated with $c_3.$ For the time we will ignore the subscript $c_3$ and replace it with $c.$ Using the fact that \begin{equation}\label{eq:gus} g(r,n)=\overline{(\frac{r}{n})_3}g(1,n),\end{equation} for $(r,n)=1$ and $\overline{(\frac{r^2}{n})_3}=(\frac{r}{n})_3,$ we open up the cubic gauss sum to get \begin{equation}\label{eq:gs}
 \sum_{\substack{\delta, d \\ N(\delta)\cdot d^3=c}} d^2 \sum_{y(\delta)} (\frac{ \overline{3^{2j}} y}{\delta})_3 e(Tr(\frac{y}{\delta})) \sum_{A  (N(\delta)\cdot d^3)^{*}}
e(\frac{A}{N(\delta)\cdot d^3})(\frac{A}{\delta})_3.
\end{equation}

Now the interior sum is a gauss sum that is only non-zero when $d=1.$  Therefore \eqref{eq:gs} equals 

\begin{equation}\label{eq:gs1}
 \sum_{\substack{\delta\\N(\delta)=c\\ \delta \equiv 1(3)} }(\frac{ \overline{3^{2j}}}{\delta})_3  \sum_{y(\delta)} (\frac{  y}{\delta})_3 e(Tr(\frac{y}{\delta})) \sum_{A  (N(\delta))^{*}}
e(\frac{A}{N(\delta)})(\frac{A}{\delta})_3.
\end{equation}

Let \begin{equation}
 \tau((\frac{\cdot}{n})) =  \sum_{1 \leq x \leq N(n)} \left(\frac{x}{n}\right)_3 e^{\frac{2 \pi i x}{N(n)}},
\end{equation} then following the arguments of \cite{BY} we have $$\tau((\frac{\cdot}{\pi}))=g(1,\pi).$$ 

Then \eqref{eq:gs1} equals \begin{align}  \sum_{\substack{\delta\\N(\delta)=c\\ \delta \equiv 1(3)} }(\frac{  \overline{3^{2j}}  }{\delta})_3  \sum_{y(\delta)} (\frac{  y}{\delta})_3 e(Tr(\frac{y}{\delta})) \sum_{A  (N(\delta))^{*}}
e(\frac{A}{N(\delta)})(\frac{A}{\delta})_3
 &= \\  \sum_{\substack{\delta\\N(\delta)=c\\ \delta \equiv 1(3)} }(\frac{  \overline{3^{2j}} }{\delta})_3  \sum_{y(\delta)} (\frac{  y}{\delta})_3 e(Tr(\frac{y}{\delta}))\sum_{A  (\delta)}
e(Tr(\frac{A}{\delta}))(\frac{A}{\delta})_3 &= \sum_{\substack{\delta\\N(\delta)=c\\ \delta \equiv 1(3)} }(\frac{ \overline{3^{2j}} }{\delta})_3 g(1,\delta)^2. \end{align}

The last equality, using properties of the cubic gauss sum can be written as $$\sum_{\substack{\delta\\N(\delta)=c\\ \delta \equiv 1(3)} } g(3^j,\delta)^2.$$

\subsection{Evaluating at $m=0$.}

Remember our goal is to evaluate  \begin{equation} \label{eq:m0} \frac{1}{X^{3/2}}\sum_{c} \frac{1}{c}\sum_{x  (c)^{*}}
e(\frac{\overline{x}l}{c})  \sum_{k(c)} e(\frac{x k^3}{c})
W_{0}(\frac{c}{X^{3/2}}).\end{equation} 

By the previous section, this boils down to studying \begin{equation}\label{eq:m01} \frac{1}{X^{3/2}} \sum_{j=0}  \sum_{\substack{\delta \\ \delta \equiv 1(3)} } \frac{ \sum_{A  (3^{j})^{*}} S(A^2 N(\delta)^{2},3^j)e(\frac{A  }{3^j})  }{3^j}   \frac{g(3^j,\delta)^2}{N(\delta)}W_{0}(\frac{N(\delta)3^j}{X^{3/2}}).\end{equation}

One would like to execute the sum over $j$ and $\delta$ separately, and it can almost be done upon inspecting Lemma \ref{cub}. Namely, the pivotal case is $j \equiv 2(3).$ So we split the $j$-sum into residue classes modulo $3.$  

\subsection{$j\equiv 2(3)$}

In this case, $S(A^2 N(\delta)^{2},3^j)=3^{\frac{2j-1}{3}}(1+ 2\cos(\frac{2\pi A^2 N(\delta)^{2}}{9})).$ Here we use from \cite{P2} that $$2\cos(\frac{2\pi A}{9})=e(1/9)\chi_9(A)+e(-1/9)\overline{\chi_9(A)},$$ where $\chi_9(A)$ is the primitive character of order 3 and conductor 9 given by $\chi_9(1+3v)=e(v/3).$ Incorporating this, \begin{multline}\label{eq:3l} \sum_{A  (3^{j})^{*}}S(A^2 N(\delta)^{2},3^j)e(\frac{A  }{3^j}) =\\  3^{\frac{2j-1}{3}}\left(\mu(1) +  e(1/9)\chi_9(N(\delta)^2)\sum_{A  (3^{j})^{*}}\chi_9(A^2) e(\frac{A  }{3^j}) + e(-1/9)\overline{\chi_9(N(\delta)^2)}\sum_{A  (3^{j})^{*}}\overline{\chi_9(A^2) } e(\frac{A  }{3^j}) \right)
\end{multline}
Note we only get a non trivial contribution from $j=2$ to the Gauss sum. So for $j=2$, \eqref{eq:3l} equals $$ 1+ e(1/9)\chi_9(N(\delta)^2)\overline{\tau(\chi_9)} + e(-1/9)\overline{\chi_9(N(\delta)^2)}\tau(\chi_9).$$

This gives 
 \eqref{eq:m01} equaling \begin{multline}\label{eq:3che}
 \frac{1}{9}\sum_{\substack{\delta \\ \delta \equiv 1(3)} }  \frac{g(9,\delta)^2}{N(\delta)}W_{0}(\frac{N(\delta)9}{X^{3/2}}) + \\ 
\frac{1}{9}e(1/9) \overline{\tau(\chi_9)}  \left( \sum_{\substack{\delta \\ \delta \equiv 1(3)} }  \frac{\chi_9(N(\delta)^2) g(9,\delta)^2}{N(\delta)}W_{0}(\frac{N(\delta)9}{X^{3/2}}) \right) + \\ 
\frac{1}{9}e(-1/9)\tau(\chi_9)  \left( \sum_{\substack{\delta \\ \delta \equiv 1(3)} }  \frac{\overline{\chi_9(N(\delta)^2)} g(9,\delta)^2}{N(\delta)}W_{0}(\frac{N(\delta)9}{X^{3/2}})\right).\end{multline}

We also use the fact that $\chi_9(N(\delta))=(\frac{\omega}{\delta})_3,$ and \ref{eq:gus}, to get
\begin{multline}\label{eq:3chee}
 \frac{1}{9}\sum_{\substack{\delta \\ \delta \equiv 1(3)} }  \frac{g(9,\delta)^2}{N(\delta)}W_{0}(\frac{N(\delta)9}{X^{3/2}}) + \\ 
\frac{1}{9}e(1/9) \overline{\tau(\chi_9)}  \left( \sum_{\substack{\delta \\ \delta \equiv 1(3)} }  \frac{ g(9\omega^2,\delta)^2}{N(\delta)}W_{0}(\frac{N(\delta)9}{X^{3/2}}) \right) + \\ 
\frac{1}{9}e(-1/9)\tau(\chi_9)  \left( \sum_{\substack{\delta \\ \delta \equiv 1(3)} }  \frac{ g(9\omega,\delta)^2}{N(\delta)}W_{0}(\frac{N(\delta)9}{X^{3/2}})\right).\end{multline}

Now we use the key result of \cite{P3}. Let $Q(s,\mu)=\sum_{\substack{\delta \in \Z[\omega] \\ \delta \equiv 1(3)}} \frac{g(\mu,\delta)^2}{N(\delta)^s}.$ \begin{theorem}\label{pat} $Q(s,\mu)$ converges absolutely if $\Re(s) > 2$ and can be continued to a meromorphic function in the entire plane. $Q(s,\mu)$ is regular in the half plane $\Re(s) \geq 5/3$ except for a possible pole at $s=5/3.$
\end{theorem}

This is the result that is very similar to the equidistribution result of \cite{HBP}. The differences is they look at the Dirichlet series \begin{equation}
\sum_{\substack{\delta \in \Z[\omega] \\ \delta \equiv 1(3)}} \frac{\widetilde{g(1,\delta)}\Lambda(\delta)}{N(\delta)^s},
\end{equation} where $\widetilde{g(1,\delta)}=\frac{g(1,\delta)}{|\delta|}.$

Letting $\hat{W_0}(s)$ be the Mellin transform of $W_0,$ then by Mellin inversion the first sum in \eqref{eq:3che} equals \begin{equation}
1/9\int_{c-i\infty}^{c+i\infty} \hat{W_0}(s) X^{3s/2}Q(1+s,9)ds,
\end{equation} with $c=4.$
Then by a contour shift to $c=2/3 - \epsilon,$ we get at most a term $O(X),$ coming from the possible pole at $s=5/3.$ 

We now address  \eqref{eq:m01} for the cases $j \equiv 0,1(3).$ Fortunately these cases are much easier to deal with.

\subsection{$j \equiv 0(3)$}

Here $S(A^2 N(\delta)^{2},3^j)=3^{\frac{2j}{3}},$ by Lemma \ref{cub}. Therefore, $$ \sum_{A  (3^{j})^{*}} S(A^2 N(\delta)^{2},3^j)e(\frac{A  }{3^j}) =3^{\frac{2j}{3}}.$$ We also need the fact that $$(\frac{3^{k+3l}}{\delta})_3=(\frac{3^{k}}{\delta})_3.$$ Here in the case of $j \equiv 0(3), g(3^j,\delta)=g(1,\delta).$ So letting $k=3j,$ we have 

\begin{equation}\label{eq:m010}  \sum_{k=0}  \sum_{\substack{\delta \\ \delta \equiv 1(3)} } \frac{1  }{3^{2k}}   \frac{g(1,\delta)^2}{N(\delta)}W_{0}(\frac{N(\delta)3^{3k}}{X^{3/2}}).\end{equation} 
 
Again by Mellin inversion and an application of Theorem \ref{pat}, we get a term of size $O(X).$

\subsection{$j \equiv 1(3)$}

The case here is trivial as $$S(A^2 N(\delta)^{2},3^j)=3^{\frac{2j-2}{3}}S(A^2N(\delta)^2,3)=0,$$ by Lemma \ref{cub}.  This  completes are calculation for $m=0.$

\section{Appendix 2}
We prove Proposition \ref{venkm2} in this second appendix.

\begin{proof}\{\text{Proposition \ref{venkm2}}\}
The proof of this proposition is very similar to the proof of Proposition 17 from the appendix of \cite{V}, as the transform in question is still a $J$-Bessel transform. 

\begin{lemma}\label{d00}
Let $A$ be a positive integer. Suppose there exists a function $h(\nu,p)$ holomorphic for $|\Re(\nu)|\leq A+\epsilon$ for some $\epsilon,$ and rapidly decaying along vertical lines in the variable $\nu.$ Let $V=F_h$ be the function associated to $h$ by the Bruggeman-Motohashi inversion formula. Then one has an estimate for $V$ and its derivatives near 0, given by $V^{l}(z) \ll z^{A-l}.$
\end{lemma}
\begin{proof}
Same as Lemma 24 in \cite{V}, expand the $J$-Bessel function in the definition of $B_{\nu,p}(z)$ into its power series and estimate.
\end{proof}

We now discuss how to get good decay at $\infty.$ 

\begin{lemma}
Suppose for $h(it,p)\ll_A e^{-A|t|}$ for all $A$. Then the function $V=F_h$ has the asymptotic expansion: $$V(z)=\cos^2(z-\pi/4)\sum_{n=1}^\infty \sum_{m=1}^\infty A_{n,m}(h) z^{-n}+ \sin^2(z-\pi/4)\sum_{n=1}^\infty \sum_{m=1}^\infty B_{n,m}(h) z^{-n}.$$ The error of truncating at $n$ is of order $O(z^{-(n+1)}),$ and $$A_{n,m}(h)=\sum_{p \in \Z} \int_{\R} a_{n,m}(it) h(it,p)(p^2+t^2)dt.$$ Here $a_{n,m}(z)$ are certain polynomials in $z$. Each derivative $V^{(l)}(z)$ also has an asymtotic expansion equal to that obtained by term by term differentiation.  
\end{lemma}
\begin{proof}
Use a product of the standard asymptotic of the $J$-Bessel function at $\infty$, as in \cite{V}. For the asymptotic see \cite{GR}(8.451).  Expanding the Bessel function inside of the integral in terms of the asymptotic and interchanging the sum and integral using the exponential decay of $h$ we get the result. For the derivative of $V$ one uses the asymptotic expansion again and the fact $J_{\nu}^{'}(z)=\frac{1}{2}(J_{\nu-1}(z)-J_{\nu+1}(z)).$
\end{proof}

\begin{lemma}
Fix  $M,N,p_0 \in \Z$ and $j_0 \in \N$ and $\epsilon>0.$ There exists a holomorphic function $h
(\nu,p)$ with decay $h(it,p)\ll_A e^{-A|t|}$ along any vertical line, so that $h(it_{j_0},p_0)=1$ and 
for $j \neq j_0$ or $p\neq p_0,$  $h(it_j,p) \ll \epsilon [(1+|t_j|)(1+|p|)]^{-N},$ so that also $A_n(h)$ 
vanish for $n \leq 2M.$ This implies $V(z)=F_h(z)=O(z^{-M}).$
\end{lemma}
\begin{proof}

Again the idea is analogous to \cite{V}. We wan to show there exists an $h$ such that $A_n(h)=0$ for $n\leq M.$ For $T_1,T_2 \in \R$ and $\delta_1,\delta_2>0,$ define
$$h_{T_1,T_2,\delta_1,\delta_2}(z,w):=\frac{1}{\delta_1\delta_2 \pi} \exp(\frac{-(z^2+T_1^2)}{\delta_1}) \exp(\frac{-(w^2-T_2^2)}{\delta_2}).$$ Then $h_{T_1,T_2,\delta_1,\delta_2}(z,w)$ is holomorphic and as $\delta_i \to 0$ localizes around $(\pm iT_1,T_2).$ It also clearly has the exponential decay $h(it,p)\ll_A e^{-A|t|}.$ From this constructed test function we will build a function $$h_{\delta_1,\delta_2}(z,w)=\sum_{j=1}^{Q_1} \sum_{k=1}^{Q_2} c_j c_k h_{T_j,T_k,\delta_1,\delta_2}(z,w)$$ that fulfills the linear constraints $(A_n(h)=0, n \leq M)$ and growth conditions $(h(it_j,p) \ll \epsilon [(1+|t_j|)(1+|p|)]^{-N})$ of the lemma at the same time. 

As $h_{\delta_1,\delta_2}(z,w)$ is holomorphic in vertical strips by Lemma \ref{d00} it has rapid decay near $0$. Its  decay at $\infty$ is controlled by the coefficients $$A_{n,m}(h_{\delta_1,\delta_2})= \sum_{j=1}^{Q_1} \sum_{k=1}^{Q_2}\sum_{p \in \Z} c_jc_k\int_{\R} a_{n,m}(it) h_{\delta_1,\delta_2} (it,p)(p^2+t^2)dt,$$ with a similar expansion for the coefficients $B_{n,m}(h).$ 
As $\delta_1,\delta_2 \to 0$ the localization of $h$ implies the integral associated to $A_{n,m}(h_
{\delta_1,\delta_2})$ approaches $$ \sum_{j=1}^{Q_1} \sum_{k=1}^{Q_2} c_jc_k a_{n,m}(iT_j) (T_j^2+T_k^2)+O_{T_j,T_k}(\delta_1\delta_2).$$ Suppose now we 
want to know exactly the asymptotic behavior of $V=F_h$ for $j,k \in \{1,..,M\}.$ We wish to choose $c_{j,k}$ and $T_j,T_k$ such that $A_{j,k}(h_
{\delta_1,\delta_2})=v_{j,k},$ for some sequence $\{v_{j,k}\}_{j,k=1}^M$ (similarly for $B_{j,k}$). Suppose $Q_1=Q_2=M,$ then we wish to solve the linear system for unknown matrix $\{c_{j,k}\}_{j,k}^M$ of size $M\times M,$ $$  [ \{c_{j,k}\}_{j,k}^M] D_{\delta_1,\delta_2}=\{v_{j,k}\}_{j,k=1}^M.$$ Here $D_{\delta_1,\delta_2}$ is the $M \times M$ matrix with $\{j,k\}$ entry $a_{n,m}(iT_j)  (T_j^2+T_k^2)+O_{T_j,T_k}(\delta_1\delta_2).$ Similar to \cite{V}, it can be checked that the determinant of the matrix $D_{0,0}$ is analytic function of $T_j,T_k$ and does not vanish in open sets of $T_j,T_k.$ Therefore we can apply Cramer's rule in each column to solve for the $c_{j,k}.$
One can then perform this same operation including the coefficients associated to $B_{j,k} (h_
{\delta_1,\delta_2}),$ by looking at matrices of size $2M.$ 

Now choose a second function $h_{it_{j_0},p_0,\delta_1,\delta_2}(z,w)$ that is localized at the 
specified points of our lemma. Then construct an $h_{\delta_1,\delta_2}(z)$ as above such that the 
associated $T_j,T_k$ are not $t_{j_0},p_0,$ and so that the asymptotic behavior of $F_{h_
{\delta_1,\delta_2}}$  matches that of $-F_{h_{it_{j_0},p_0,\delta_1,\delta_2}}$ up to the $M$-th term 
of the asymptotic expansion. In other words the associated $\{v_{j,k}\}_{j,k=1}^M$ associated to  $h_
{\delta_1,\delta_2}(z)$ are exactly the negatives of those associated to whatever coefficients are 
associated to the asymptotic expansion of $h_{it_{j_0},p_0,\delta_1,\delta_2}(z,w).$

Then the function $h'_{\delta_1,\delta_2}=h_{it_{j_0},p_0,\delta_1,\delta_2}+h_{\delta_1,\delta_2}$ will have the property that $F_{h'_{\delta}} \ll z^N$ near 0 and $F_{h'_{\delta}}(z) \ll z^{-M}$ near $\infty.$

By choosing $\delta_1,\delta_2>0$ sufficiently small and by our construction above, $h'(it_{j_0},p_0):=h'_{\delta_1,\delta_2}(it_{j_0},p_0)=1$ while for any other $\{t_j,p\}$ combination $h'(it_j,p) \ll \epsilon(\delta_1,\delta_2)[(1+|t_j|)(1+|p|)]^{-N},$ for any integer $N.$
\end{proof}

\{\text{Conclusion of Proposition \ref{venkm2}}\}

By the above lemmas we know that there exists $V=F_{h'}$ associated to a $h'$ that satisfies the conclusion of Proposition \ref{venkm2} except it is not of compact support. We can truncate this function following \cite{V} again. Let $g$ be a smooth compactly supported function on $\R$ such that  $g(x)=1$ on $[-1,1]$ and equal zero outside of $[-2.2].$ Then define $V_r(z)=F_{h'}(z)g(\log(|z|)r).$ Pointwise this converges to $V(z)=F_{h'}(z)$ as $ \r \to 0,$ and is compactly supported.  We then ask the behavior of $h(V-V_r,\nu,p)?$ Following \cite{BMo} in the proof Theorem 11.1, 
one gets the estimate for any smooth function $f$ with good decay properties at $\infty$ $$h(f,\nu,p) \ll (1+|\nu|+|p|)^{-N}$$ for any integer $N.$ Indeed, their proof depends on bounding the Mellin transform of $f$ optimally which depends on applying integration by parts which then will then depend on the decay properties of $f$ at $\infty.$ This is a similar argument to \cite{V}. 

Therefore, choosing compactly supported test function instead of those we constructed above, does not affect the properties of the Bessel transform. We conclude that there exists a compactly supported $V$ that has the properties of $F_{h'},$ and the proposition is complete.  

\end{proof}

\end{document}